\documentclass[a4paper]{article}

    \usepackage{geometry}
    \usepackage{amsfonts,amsmath,amsthm,amssymb,mathtools}
    \usepackage{graphicx}
    \usepackage{mathrsfs}
    \usepackage{lmodern}
    \usepackage[shortlabels]{enumitem}
    \usepackage{bm}
    \usepackage{eucal}
    \usepackage{tikz,tikz-cd}
    \usepackage{subcaption}
    
    \definecolor{allrefcolors}{rgb}{.05,.45,.6}
    \usepackage[colorlinks=true,allcolors=allrefcolors]{hyperref}
    \usepackage[nameinlink,noabbrev,capitalise]{cleveref}
    
    \usetikzlibrary{calc,intersections}
    \usetikzlibrary{decorations.markings,arrows}
    \usetikzlibrary{shapes.geometric}

    \tikzset{-<-/.style={decoration={
    	markings,
    	mark=at position 0.5 with {\arrow{<}}},postaction={decorate}}
    	}
    \tikzset{->-/.style={decoration={
    	markings,
    	mark=at position 0.5 with {\arrow{>}}},postaction={decorate}}
    	}

    \newsavebox\CBox
    \newcommand\hcancel[2][0.5pt]{%
    	\ifmmode\sbox\CBox{$#2$}\else\sbox\CBox{#2}\fi%
    	\makebox[0pt][l]{\usebox\CBox}%
    	\rule[0.5\ht\CBox-#1/2]{\wd\CBox}{#1}}
    
    \newcommand{\hooklongrightarrow}{\lhook\joinrel\longrightarrow}
    
    \DeclareMathOperator*{\hocolim}{hocolim}
    \DeclareMathOperator*{\holim}{holim}
    \DeclareMathOperator{\Rhom}{Rhom}
    
    \newtheorem{theorem}{Theorem}
    \newtheorem{corollary}[theorem]{Corollary}
    \newtheorem{definition}[theorem]{Definition}
    \newtheorem{lemma}[theorem]{Lemma}
    \newtheorem{proposition}[theorem]{Proposition}

    \newtheorem{question}[theorem]{Question}
    \newtheorem{remark}[theorem]{Remark}

    \crefname{theorem}{Theorem}{Theorems}
    \crefname{corollary}{Corollary}{Corollaries}
    \crefname{definition}{Definition}{Definitions}
    \crefname{lemma}{Lemma}{Lemmas}
    \crefname{proposition}{Proposition}{Propositions}
    \crefname{assumption}{Assumption}{Assumptions}
    \crefname{conjecture}{Conjecture}{Conjectures}
    \crefname{question}{Question}{Questions}
    \crefname{remark}{Remark}{Remarks}
    \crefname{example}{Example}{Examples}
    \crefname{claim}{Claim}{Claims}
    \crefname{section}{Section}{Sections}
    \crefname{figure}{Figure}{Figures}

    \AddToHook{env/proposition/begin}{\crefalias{theorem}{proposition}}
    \AddToHook{env/lemma/begin}{\crefalias{theorem}{lemma}}
    \AddToHook{env/corollary/begin}{\crefalias{theorem}{corollary}}
    \AddToHook{env/definition/begin}{\crefalias{theorem}{definition}}
    \AddToHook{env/example/begin}{\crefalias{theorem}{example}}
    \AddToHook{env/notation/begin}{\crefalias{theorem}{notation}}
    \AddToHook{env/assumption/begin}{\crefalias{theorem}{assumption}}
    \AddToHook{env/remark/begin}{\crefalias{theorem}{remark}}
    \AddToHook{env/claim/begin}{\crefalias{theorem}{claim}}
    \AddToHook{env/conjecture/begin}{\crefalias{theorem}{conjecture}}
    \AddToHook{env/question/begin}{\crefalias{theorem}{question}}
    
    \numberwithin{equation}{section}
    \numberwithin{theorem}{section}
    
\begin{document}
    	
    \title{\textbf{Persistence of unknottedness of clean Lagrangian intersections}}\author{Johan Asplund
    	\and
    	Yin Li
    }
    \newcommand{\Addresses}{{
    		\bigskip
    		\footnotesize
    		
    		Johan Asplund, \textsc{Department of Mathematics, Stony Brook University, Stony Brook, NY 11794, United States}\par\nopagebreak
    		\textit{E-mail address}: \texttt{johan.asplund@stonybrook.edu}
    		
    		\medskip
    		
    		Yin Li, \textsc{Department of Mathematics, Uppsala University, 753 10 Uppsala, Sweden}\par\nopagebreak
    		\textit{E-mail address}: \texttt{yin.li@math.uu.se}
    		
    }}
    \date{}\maketitle
    \begin{abstract}
    Let $Q_0$ and $Q_1$ be two Lagrangian spheres in a $6$-dimensional symplectic manifold. Assume that $Q_0$ and $Q_1$ intersect cleanly along a circle that is unknotted in both $Q_0$ and $Q_1$. We prove that there is no nearby Hamiltonian isotopy of $Q_0$ and $Q_1$ to a pair of Lagrangian spheres meeting cleanly along a circle that is knotted in either component, answering a question of Smith. The proof is based on a classification of the spherical summands in the prime decomposition of an exact Lagrangian in the Stein neighborhood of the union $Q_0\cup Q_1$ and the deep result that lens space rational Dehn surgeries characterize the unknot.
    \end{abstract}
    
    \section{Introduction}
    
    \subsection{Summary of results}
    Let $(M,\omega)$ be a symplectic manifold, and $L \subset M$ a Lagrangian subset. A Hamiltonian isotopy in $M$ is said to be a \textit{nearby Hamiltonian isotopy} of $L$ if it is supported in a Stein neighborhood of $L$. A Hamiltonian diffeomorphism is said to be \textit{nearby} if it is the time-$1$ flow of a nearby Hamiltonian isotopy.
    
    The purpose of this paper is to prove the following result
    \begin{theorem}\label{theorem:main_intro}
        Let $(M,\omega)$ be a $6$-dimensional symplectic manifold, and let $Q_0, Q_1 \subset M$ be two Lagrangian spheres that intersect cleanly along a circle that is unknotted in both $Q_0$ and $Q_1$. There is no nearby Hamiltonian diffeomorphism $\varphi$ of $Q_0 \cup Q_1$ such that the intersection $Q_0 \cap \varphi(Q_1)$ is clean and is a knotted circle in either $Q_0$ or $Q_1$.
    \end{theorem}
    
    This result answers the following question of Ivan Smith in the negative, in the case where $Q_0\cap Q_1$ is an unknot and the Hamiltonian isotopies are nearby.
    \begin{question}[{\cite[Question 1.5]{gp}}]
    In a $6$-dimensional symplectic manifold $M$, if a pair of Lagrangian spheres $Q_0$ and $Q_1$ intersect cleanly along a circle, can the isotopy class of the knot $Q_0\cap Q_1$ change under (nearby) Hamiltonian isotopies in $Q_0$ or $Q_1$?
    \end{question}
    
    With the setting as in \cref{theorem:main_intro}, let $Z \coloneqq  Q_0 \cap Q_1$. Let $\nu_{Z/Q_i}$ denote the normal bundle of $Z \subset Q_i$ for $i\in \{0,1\}$. Fixing an identification $\eta\colon\nu_{Z/Q_0}\xrightarrow{\cong}\nu_{Z/Q_1}$ of the normal bundles of $Z$ in $Q_0$ and $Q_1$ respectively, we construct a Stein manifold $W_\eta$ by plumbing $T^\ast Q_0$ and $T^\ast Q_1$ along $Z$, see \cref{section:plumbings} for details. The possible choices of $\eta$ are identified with $\mathbb{Z}$ and we normalize so that the Lagrangian surgery of $Q_0$ and $Q_1$ along $Z$ in $W_k$ is diffeomorphic to $S^1\times S^2$ if $k = 0$ and the lens space $L(k,1)$ for $k\geq 1$. Since any Stein neighborhood of $Q_0\cup Q_1\subset M$ is symplectically equivalent to some $W_k$, our main result \cref{theorem:main_intro} follows from the following result.
    
    \begin{theorem}[{\cref{theorem:main}}]\label{theorem:unknot}
    Let $k\in\mathbb{Z}_{\geq0}$ be any non-negative integer. There is no Hamiltonian isotopy of the core spheres $Q_0,Q_1\subset W_k$ to a pair of exact Lagrangian spheres meeting cleanly in a circle that is knotted in either component.
    \end{theorem}
    \begin{remark}
        \begin{enumerate}
            \item \Cref{theorem:unknot} was proven in the case $k = 0$ by Smith--Wemyss \cite[Proposition 4.10]{sw} and in the case $k = 1$ by Ganatra--Pomerleano \cite[Proposition 6.29]{gp}.
            \item Okamoto \cite{yo} studies the rigidity of knot types under Hamiltonian isotopies using augmentation varieties associated to Legendrian knots and deduced closely related results.
        \end{enumerate}
    \end{remark}
    
    As we shall explain in \cref{section:outline_of_proof}, part of the proof of \cref{theorem:unknot} consists of a direct computation of the wrapped Fukaya category of $W_k$ for any $k\geq 0$. This computation also yields new instances of homological mirror symmetry that we now describe.
    
    Let $\mathbb{K}$ be any field. The mirror of $W_k$ is a crepant (partial) resolution $Y_k\rightarrow\mathit{Spec}(R_k)$ of the compound $A_2$ singularity
    \begin{equation}\label{eq:sing}
    R_k=\frac{\mathbb{K}[u,v,x,y]}{\left(uv-xy\left((x+1)^k+y-1\right)\right)}, \quad k\geq1
    \end{equation}
    at the origin, with the exceptional locus consisting of a pair of $(-1,-1)$ rational curves meeting at a single point, see \cref{section:contraction} for details. The resolution $Y_k$ is a normal algebraic variety which may be smooth or has isolated singularities, depending on $k$. Associated to the resolution $Y_k$ is a non-positively graded dg algebra $\mathcal{A}_k$, called the \textit{derived contraction algebra}, which prorepresents the derived deformations of the irreducible components of the reduced exceptional fiber of the contraction. Its zeroth cohomology $H^0(\mathcal{A}_k)$ recovers the contraction algebra introduced by Donovan--Wemyss \cite{dw}. See \cite{mb} for its construction and \cref{section:contraction} for a convenient model which works in our case. There is an equivalence
    \[
    D^\mathit{perf}(\mathcal{A}_k)\cong D^b\mathit{Coh}(Y_k)/\langle\mathcal{O}_{Y_k}\rangle
    \]
    between the derived category of perfect modules over $\mathcal{A}_k$ and the \textit{relative singularity category} studied by Kalck--Yang \cite{ky}. By identifying $\mathcal{A}_k$ with the Ginzburg dg algebra $\mathcal{G}_k$ in the statement of \cref{theorem:contraction}, we prove in \cref{section:contraction} the following
    
    \begin{theorem}[{\cref{theorem:MS}}]\label{theorem:MS_intro}
    Let $\mathbb{K}$ be any field and $k\geq1$. There is an equivalence
    \[
    D^\mathit{perf}\mathcal{W}(W_k;\mathbb{K})\cong D^b\mathit{Coh}(Y_k)/\langle\mathcal{O}_{Y_k}\rangle
    \]
    between the derived wrapped Fukaya category of $W_k$ and the relative singularity category of its mirror.
    \end{theorem}
    \begin{remark}
    \begin{enumerate}
    \item Let $p$ be a prime and $n\in \mathbb{Z}_{\geq 1}$. If $k=p^n$ and $\mathbb{K}$ is a field of characteristic $p$, or when $k=1$ and $\mathbb{K}$ is any field, the singularity \eqref{eq:sing} becomes
    \begin{equation}\label{eq:sing1}
    \frac{\mathbb{K}[u,v,x,y]}{\left(uv-xy(x^k+y)\right)},\quad k\geq1.
    \end{equation}
    In this case, \cref{theorem:MS_intro} reduces to \cite[Corollary 1.3]{sw}.
    \item By the work of Evans--Lekili \cite{evle}, \cref{theorem:MS_intro} also proves \cite[Conjecture E]{lese} in this particular case, see \cref{section:contraction}.
    \end{enumerate}
    \end{remark}
    
    Let $\mathit{PBr}_3$ denote the pure braid group on three strands, and let $\mathit{Symp}_\mathit{gr}(W_k)$ be the group of graded symplectomorphisms of $W_k$ (see \cite[Section 2.b]{sg}); the latter is a central $\mathbb Z$-extension of $\mathit{Symp}(W_k)$ and acts by autoequivalences on the compact Fukaya category. By \cite[Lemma 4.15]{sw} there is a natural representation 
    \[
        \rho_k \colon \mathit{PBr}_3\longrightarrow\pi_0\mathit{Symp}_\mathit{gr}(W_k)
    \]
    whose image is generated by $\tau_0$, $\tau_1$ and $\tau_0\tau_1^2 \tau_0^{-1}$, where $\tau_i$ is the Dehn twist on the core sphere $Q_i$ of $W_k$.
    
    \begin{corollary}
    For any $k\geq1$, the natural representation $\rho_k$ is faithful.
    \end{corollary}
    \begin{proof}
        The proof is the same as the one of \cite[Corollary 1.5]{sw}. The conclusion holds for any $k\geq 1$ since \cref{theorem:MS_intro} holds for any $k\geq 1$.
    \end{proof}
    
    \subsection{Outline of the proof of the main result}\label{section:outline_of_proof}
    
    We now give an outline of the proof of \cref{theorem:unknot}.
    
    Roughly speaking, both of the proofs in \cite{sw} and \cite{gp} rely on the existence of \textit{dilations}, which is a degree $1$ class $b\in\mathit{SH}^1(M;\mathbb{K})$ in the symplectic cohomology of some Liouville manifold $M$ whose image under the BV (Batalin--Vilkovisky) operator $\Delta\colon \mathit{SH}^\ast(M;\mathbb{K})\rightarrow\mathit{SH}^{\ast-1}(M;\mathbb{K})$ gives the identity $1\in\mathit{SH}^0(M;\mathbb{K})$, see \cite{ps2,ss} for details. The existence of a dilation imposes strong restrictions on the topology of Lagrangian submanifolds in $M$, which in particular excludes certain Dehn surgeries, therefore also knot types on which the surgery is performed. When the BV operator hits an arbitrary invertible element $h\in\mathit{SH}^0(M;\mathbb{K})^\times$ instead of the identity, we say that $M$ admits a \textit{quasi-dilation}. In \cite[Lemma 4.9]{sw}, Smith and Wemyss proved that $W_0$ admits a quasi-dilation over a field $\mathbb{K}$ of characteristic zero by embedding it into the total space of a Lefschetz fibration whose smooth fibers are $4$-dimensional $A_3$ Milnor fibers. For $W_1$, Ganatra and Pomerleano proved the existence of a dilation in $\mathit{SH}^1(W_1;\mathbb{F}_3)$, where $\mathbb{F}_3$ is a field of order $3$, see \cite[Section 6.4]{gp}. However, when $k\geq2$, it is not clear whether $W_k$ admits a dilation over some field $\mathbb{K}$ due to the lack of affine realizations. 
    
    It turns out that in order to study the rigidity of knot types of $Z\subset W_k$, we can use a much weaker notion---a \emph{cyclic quasi-dilation}, introduced by the second author in \cite{yle}. It is defined using the $S^1$-equivariant symplectic cohomology $\mathit{SH}_{S^1}^\ast(M;\mathbb{K})$, and the marking map $\mathbf{B}\colon\mathit{SH}_{S^1}^1(M;\mathbb{K})\rightarrow\mathit{SH}^0(M;\mathbb{K})$ as a stand-in for the BV operator, see \cref{section:dilations} for details.
    
    More precisely, we show that the Stein manifold $W_k$ for any $k\geq 0$ admits a cyclic quasi-dilation over \textit{any} field $\mathbb{K}$, from which the following classification result of exact Lagrangians in $W_k$ follows (cf.\@ \cref{lemma:aspherical,proposition:TOI,proposition:lensprism}) after a detailed analysis of the fundamental group rings of its prime summands.
    
    \begin{lemma}\label{lemma:classification}
    Let $k \in \mathbb{Z}_{\geq 1}$, and $L\subset W_k$ is a closed oriented exact Lagrangian submanifold. Then any summand in the prime decomposition of $L$ is diffeomorphic to either:
    \begin{enumerate}
    \item $S^1\times S^2$,
    \item a lens space such that if $p$ is a prime factor of $|\pi_1(L)|$, then $p$ also divides $k$, or
    \item a prism manifold such that $\pi_1(L)/Z(\pi_1(L))\cong D_{2n}$ for some even number $n$.
    \end{enumerate}
    \end{lemma}
    
    \begin{remark}
    When $k=1$ or $k>2$ is a prime number $p$, there is no exact Lagrangian $S^1\times S^2$ with vanishing Maslov class in $W_k$ by \cite[Corollary 1.10]{sw}. Their proof is purely algebraic and relies on the classification of fat-spherical objects in the compact Fukaya category $\mathcal{F}(W_k;\mathbb{F}_p)$. Intuitively, one may expect that any prime exact Lagrangian submanifold $L\subset W_k$ is diffeomorphic either to a sphere or a lens space $L(k,1)$, although this may be out of the reach with current techniques.
    \end{remark}
    
    Now, \cref{theorem:unknot} is deduced as a consequence of \cref{lemma:classification} and the proof of Gordon's conjecture by Kronheimer--Mrowka--Ozsv\'{a}th--Szab\'{o} \cite[Theorem 1.1]{kmos}, see \cref{section:proof} for details.
    
    The existence of a cyclic quasi-dilation for $W_k$ for any $k\in \mathbb{Z}_{\geq 1}$ is proved via \cite[Proposition 4]{yle}, by showing that the wrapped Fukaya category $\mathcal{W}(W_k;\mathbb{K})$ admits an exact Calabi--Yau structure. This is done by showing that the wrapped Fukaya category is equivalent to perfect modules over a Ginzburg dg algebra which is known to admit an exact Calabi--Yau structure \cite[Theorem 4.3.8]{bd}. The wrapped Fukaya category of $W_k$ is generated by the two Lagrangian cocores $L_0$ and $L_1$ (i.e., the cotangent fibers over $Q_0$ and $Q_1$) \cite{cdrgg,gps}. Thus, it suffices to compute the endomorphism $A_\infty$-algebra of their union $L_0 \cup L_1$ 
    \[
    \mathcal{W}_k\coloneqq \bigoplus_{i,j \in \{0,1\}}\mathit{CW}^\ast(L_i,L_j).
    \]
    Our particular choices of $L_0$ and $L_1$ will be made precise in \cref{section:setup}. The computation of $\mathcal W_k$ uses a Morse--Bott--Lefschetz fibration
    \[
    \pi \colon W_k\longrightarrow\mathbb{C},
    \]
    and (slight modifications of) techniques of Abouzaid--Auroux \cite{aa}. It is carried out in \cref{section:computation}. The outcome is the following result.
    
    \begin{theorem}[{\cref{theorem:ginzburg}}]\label{theorem:contraction}
    Let $\mathbb{K}$ be any field and $k\geq1$. The Fukaya $A_\infty$-algebra $\mathcal{W}_k$ over the semisimple ring $\Bbbk=\mathbb{K}e_0\oplus\mathbb{K}e_1$ is quasi-isomorphic to the Ginzburg dg algebra $\mathcal{G}_k$ associated to the $2$-cycle quiver
    \[
    \begin{tikzcd}
    \bullet\arrow[r,bend left,"e"] & \bullet \arrow[l,bend left,"f"]
    \end{tikzcd}
    \]
    with potential
    \[
    w_k=efe\left(1+(fe+1)+\cdots+(fe+1)^{k-1}\right)f.
    \]
    \end{theorem}
    
    \begin{remark}
    When $\mathrm{char}(\mathbb{K})=0$, since the lowest order term in the potential $w_k$ is some non-zero multiple of $(ef)^2$, the completion $\widehat{\mathcal{G}}_k$ of $\mathcal{G}_k$ is quasi-isomorphic to the complete Ginzburg dg algebra $\mathcal{G}_1\cong\widehat{\mathcal{G}}_1$ for any $k\geq1$. This is compatible with \cite[Theorem 1.1]{sw}, which shows that the Fukaya $A_\infty$-algebra $\mathcal{Q}_k$ of the two core spheres in the double bubble plumbing $W_k$ is always quasi-isomorphic to $\mathcal{Q}_1$ for $k\geq1$ and $\mathrm{char}(\mathbb{K})=0$. In fact, there is a quasi-isomorphism
    \[
    \mathcal{Q}_k\cong \Rhom_{\mathcal{W}_k}(\Bbbk,\Bbbk)
    \]
    due to Ekholm--Lekili \cite{ekle}.
    
    Also note that when $\mathbb K$ is a field of characteristic $p$ for some prime $p$, and $k=p^n$ for some $n \in \mathbb{Z}_{\geq 1}$, the potential $w_k$ becomes $(ef)^{k+1}$ over $\mathbb{K}$. In this case, our computation recovers the quiver with potential that appears in \cite[Section 2.3]{sw}.
    \end{remark}
    
    Identifying the wrapped Fukaya categories (or Chekanov--Eliashberg dg algebras) with Ginzburg dg algebras may be of some interest. Any (relative) Ginzburg dg algebra associated to a quiver with trivial potential is quasi-isomorphic to some Chekanov--Eliashberg dg algebra \cite[Theorem 1.1]{ar}, but very few such examples with \textit{non-trivial} potentials are known \cite[Theorem 1.2]{ylk}.
    
    \subsection*{Conventions}
    Throughout the paper we assume that $\mathbb K$ denotes an arbitrary field, although some constructions hold more generally over commutative rings.
    
    Since $c_1(W_k)=0$ for all $k\geq0$, all of our symplectic invariants, including the compact Fukaya category $\mathcal{F}(W_k;\mathbb{K})$, the wrapped Fukaya category $\mathcal{W}(W_k;\mathbb{K})$ and the symplectic cohomologies $\mathit{SH}^\ast(W_k;\mathbb{K})$ and $\mathit{SH}_{S^1}^\ast(W_k;\mathbb{K})$ considered in this paper will be $\mathbb{Z}$-graded. Moreover, all the Lagrangian submanifolds of $W_k$ considered in this paper will be exact, and cylindrical if they are non-compact, which allows us to define the compact and wrapped Fukaya categories over $\mathbb{Z}$, see for example \cite{gps1}. For an $A_\infty$-algebra $\mathcal{A}$, we shall write $b\cdot a$ or simply $ba$ for the product $\mu^2_\mathcal{A}(b,a)$.
    
    \subsection*{Outline}
        In \cref{section:computation} we compute the Fukaya $A_\infty$-algebra $\mathcal W_k$ and prove \cref{theorem:MS_intro}. In \cref{section:contraction_and_ginz} we relate $\mathcal W_k$ to a certain derived contraction algebra and a Ginzburg dg algebra. Finally, \cref{section:proof} is devoted to proving the classification result \cref{lemma:classification} and our main result \cref{theorem:main_intro}.
    
    \subsection*{Acknowledgments}
    
    The usages of dilations over fields of finite characteristics was asked by Mohammed Abouzaid during a conversation with the second author in 2022. He also benefited from discussions about central units in group rings on Mathematics Stack Exchange, and in particular indebted to Jeremy Rickard for \cref{remark:JR}. We thank Georgios Dimitroglou Rizell and Yank{\i} Lekili for many useful discussions and suggestions. We thank the anonymous referee who brought our attention to a gap in the earlier version of \cref{proposition:lensprism} and whose numerous comments improved the exposition of the paper.
    
    \section{The wrapped Fukaya category of double bubble plumbings}\label{section:computation}
    
    The goal of this section is to compute the wrapped Fukaya category of the double bubble plumbings $W_k$ for $k\geq 0$. We exploit the existence of a Morse--Bott--Lefschetz fibration $\pi \colon W_k\rightarrow\mathbb{C}$ whose associated fiberwise wrapped Fukaya category we compute following techniques introduced by Abouzaid--Auroux \cite{aa}. The wrapped Fukaya category of $W_k$ is obtained as a certain localization of the fiberwise wrapped Fukaya category.
    
    In \cref{section:plumbings} we introduce and give definitions of the double bubble plumbings $W_k$. In \cref{section:setup} we give a definition of the fiberwise wrapped Fukaya category for the Morse--Bott--Lefschetz fibration $\pi$. The computation of the fiberwise wrapped Fukaya category occupies \cref{section:single,section:product}, and in \cref{subsection:wrapping} we discuss the localization of the fiberwise wrapped Fukaya category of $W_k$.
    
    \subsection{Double bubble plumbings}\label{section:plumbings}
    Following \cite[Appendix A]{ap} and in particular \cite[Section 2]{sw} we now give the definition of the so-called double bubble plumbing.
    
    Generally, consider two smooth $n$-dimensional manifolds $Q_0$ and $Q_1$ and two embeddings $\kappa_i \colon Z \hookrightarrow Q_i$, $i\in \{0,1\}$. Let $\nu_{Z/Q_i}$ denote the normal bundle of the embedding $\kappa_i$, and fix an identification $\eta \colon \nu_{Z/Q_0} \overset{\cong}{\to} \nu_{Z/Q_1}$. The Weinstein neighborhood theorem implies that there is a tubular neighborhood $U(Z)_i \subset DT^\ast Q_i$ of $DT^\ast Z$ in $DT^\ast Q_i$ that is symplectomorphic to the pullback of the disk subbundle of $\iota_i^\ast(\nu_{Z/Q_i} \otimes \mathbb{C})$ where $\iota_i$ is the composition
    \[
    Z \overset{\kappa_i}{\hooklongrightarrow} Q_i \hooklongrightarrow DT^\ast Q_i,
    \]
    where the second map is the zero section. We now define the Stein domain
    \begin{equation}\label{eq:plumb}
        DT^\ast Q_0 \#_{Z,\eta} DT^\ast Q_1
    \end{equation}
    to be the result of gluing $DT^\ast Q_0$ and $DT^\ast Q_1$ using the identification
    \[
    \eta \otimes i \colon \iota_0^\ast(\nu_{Z/Q_0}\otimes \mathbb{C}) \overset{\cong}{\longrightarrow} \iota_1^\ast(\nu_{Z/Q_1}\otimes \mathbb{C}).
    \]
    We now define the plumbing $W_\eta(\kappa_0,\kappa_1)$ to be the completion of the Stein domain \eqref{eq:plumb}.
    
    \begin{definition}[Double bubble plumbing {\cite[Section 2.2]{sw}}]\label{def:double_bubble}
    Let $Q_0 = Q_1 = S^3$, $Z = S^1$ and let $\kappa_i \colon Z \hookrightarrow Q_i$ for $i\in \{0,1\}$ be two embeddings. For any identification $\eta \colon \nu_{Z/Q_0} \overset{\cong}{\to} \nu_{Z/Q_1}$ of the two normal bundles of $\kappa_0$ and $\kappa_1$, the Stein manifold $W_\eta(\kappa_0,\kappa_1)$ is called a \emph{double bubble plumbing}.
    
    When both $\kappa_0$ and $\kappa_1$ are unknotted, we denote the corresponding double bubble plumbing by $W_\eta$.
    \end{definition}
    
    Note that by construction, the two exact Lagrangians $Q_0,Q_1 \subset W_\eta$ intersect cleanly along the submanifold $Z$. Performing Lagrangian surgery of $Q_0$ and $Q_1$ along their clean intersection as described in \cite[Section 2.3]{mw} preserves exactness \cite[Lemma 6.2]{mw}. Since the normal bundle $\nu_{S^1/S^3} \cong S^1 \times D^2$ is trivial, the set of identifications $\eta \colon\nu_{S^1/S^3} \overset{\cong}{\to}\nu_{S^1/S^3}$ form a torsor for the integers. We fix an identification of such $\eta$ with $\mathbb{Z}$ in such a way that the Lagrangian surgery of $Q_0$ and $Q_1$ along $Z$ yields $S^1 \times S^2$ for $\eta = 0$, and the lens space $L(\eta,1)$ for $\eta \geq 1$. With this normalization, there is a symplectomorphism $W_{-k} \cong W_k$ for any $k\geq 1$ since the underlying smooth manifolds are diffeomorphic. Therefore it suffices to consider $W_k$ for $k\in \mathbb{Z}_{\geq 0}$.
    
    \subsection{The fiberwise wrapped Fukaya category}\label{section:setup}
    
    Abouzaid--Auroux \cite{aa} defined a partially wrapped Fukaya category associated to a toric Landau--Ginzburg model, called the \textit{fiberwise wrapped Fukaya category}. In this paper, we consider a similarly constructed partially wrapped Fukaya category, but associated to a Morse--Bott--Lefschetz fibration on $W_k$.
    \subsubsection{A Morse--Bott--Lefschetz fibration}\label{section:morse-bott-lefschetz}
        A Morse--Bott--Lefschetz fibration is a holomorphic map $p \colon X \to \mathbb{C}$ with transversely non-degenerate critical manifolds, meaning that its complex Hessian is pointwise non-degenerate as a complex bilinear form on each fiber of the complex normal bundle over the critical locus. Near any point in the critical locus, there are coordinates $(z_1,\ldots,z_n)$ so that $p$ is given by $(z_1,\ldots,z_n) \mapsto z_1^2 + \cdots + z_k^2$ for some $k \in \{1,\ldots,n\}$; if $k = n$ near every point in the critical locus, $p$ is a Lefschetz fibration. We say that a Morse--Bott--Lefschetz fibration $p$ is \emph{elementary} if there is a unique critical value of $p$.
    
        Following \cite[Section 4.1]{sw}, the Stein manifold $W_k$ for any $k\in \mathbb{Z}_{\geq 0}$ admits a Morse--Bott--Lefschetz fibration
        \[
        \pi \colon W_k \longrightarrow \mathbb{C}
        \]
        with generic fiber $(\mathbb{C}^\ast)^2$ and three singular fibers over the points $-1,0,1 \in \mathbb{C}$. The singular fibers are all isomorphic to $(\mathbb{C} \vee \mathbb{C}) \times \mathbb{C}^\ast$, and the vanishing cycles are $a$, $b$ and $a+kb$, respectively with respect to some basis $a,b$ of $H_1(T^2;\mathbb{Z})$, chosen consistently for the fibers of $\pi$, where $a$ is the meridian and $b$ is a longitude.
    
    \subsubsection{General setup and admissible Lagrangians}\label{section:general_setup}
    The goal of this section is to compute the fiberwise wrapped Fukaya category $\mathcal W(W_k,\pi)$. In order for this $A_\infty$-category to be well-defined, certain maximum principles for solutions of Floer equations in the construction of $\mathcal W(W_k,\pi)$ are employed. To be able to prove such maximum principles, we follow \cite[Section 3.1]{aa} and fix the following data:
    \begin{itemize}
    	\item an almost complex structure $J$ on $W_k$ compatible with the symplectic form such that $\pi$ is $J$-holomorphic outside of some small disks $\Delta_{-1},\Delta_0,\Delta_1\subset\mathbb{C}$ centered at $-1,0,1\in \mathbb{C}$,
    	\item a continuous weakly $J$-plurisubharmonic function $h\colon W_k\rightarrow[0,\infty)$,
    	\item a non-negative wrapping Hamiltonian $H\colon W_k\rightarrow[0,\infty)$,
    	\item a closed subset $W_k^\mathit{in}\subset W_k$, whose intersection with every fiber of $\pi$ is a sublevel set of $h$,
    \end{itemize}
    such that
    \begin{enumerate}[(i)]
     \item The restrictions of $h$ and $H$ to each fiber of $\pi$ are proper.
     \item The fibers of $\pi$ are preserved by the Hamiltonian flow of $H$, and outside of $W_k^\mathit{in}$, the level sets of $h$ are preserved by the Hamiltonian flow of $H$. Namely, $d\pi(X_H) = 0$, and $dh(X_H) = 0$ outside of $W_k^\mathit{in}$.
     \item Horizontal parallel transport preserves the function $H$ everywhere, and preserves the function $h$ outside of $W_k^\mathit{in}\cup\bigcup_{i\in\{-1,0,1\}}\pi^{-1}(\Delta_i)$. Namely, if $\xi^\#$ is the horizontal lift of some vector on $\mathbb{C}$, then
     $dH(\xi^\#)=0$, and $dh(\xi^\#)=0$ outside of $W_k^\mathit{in}\cup\bigcup_{i\in\{-1,0,1\}}\pi^{-1}(\Delta_i)$.
     \item The $1$-form $d^\mathbb{C}h$ vanishes on the symplectic orthogonal to the fibers of $\pi$ outside of $W_k^\mathit{in}\cup\bigcup_{i\in\{-1,0,1\}}\pi^{-1}(\Delta_i)$, i.e., $d^\mathbb{C}h(\xi^\#) = 0$ outside of $W_k^\mathit{in}\cup\bigcup_{i\in\{-1,0,1\}}\pi^{-1}(\Delta_i)$.
     \item The $1$-form $d^\mathbb{C}h$ is preserved by the Hamiltonian flow of $H$, and by horizontal parallel transport, i.e., $\mathcal{L}_{X_H}(d^\mathbb{C}h)=\mathcal{L}_{\xi^\#}(d^\mathbb{C}h)=0$.
     \item $d^\mathbb{C}h(X_H)\geq0$.
    \end{enumerate}
    In the above, (ii) and (iii) essentially mean that the functions $\pi$, $H$ and $h$ are Poisson commutative outside of $W_k^\mathit{in}\cup\bigcup_{i\in\{-1,0,1\}}\pi^{-1}(\Delta_i)$. (iv) and (v) are needed to ensure that the admissibility of the objects in the fiberwise wrapped Fukaya category (cf.\@ \cref{definition:admissible_lag}) is preserved under various Hamiltonian isotopies. Together with (vi), they also guarantee that the maximum principle holds for solutions of the perturbed Cauchy--Riemann equations relevant for the definition of the fiberwise wrapped Fukaya category. We note the difference between the above and the data used in \cite[Section 3.1]{aa}. In \cite{aa} the functions $h$ and $H$ are constructed for toric Landau--Ginzburg models, where the superpotential has a unique critical value. Our situation is different in the sense that $W_k$ is not toric and the Morse--Bott--Lefschetz fibration $\pi$ has multiple critical values.

    \begin{lemma}\label{lemma:setup}
    There exist functions $h,H\colon W_k\rightarrow[0,\infty)$ satisfying the conditions (i)--(vi) above.
    \end{lemma}
    \begin{proof}
    Consider the standard complex structure on $W_k$. To construct the required functions $h$ and $H$, note that $\pi\colon W_k\rightarrow\mathbb{C}$ is the fiber sum of three elementary Morse--Bott--Lefschetz fibrations $\pi_j\colon E_j\rightarrow\mathbb{C}$ for $j\in\{-1,0,1\}$, where $E_j\cong\mathbb{C}^2\times\mathbb{C}^\ast$ and $\pi_j$ has $j\in\mathbb{C}$ as its unique critical value. Note that when taking the boundary connected sum between two different elementary Morse--Bott--Lefschetz fibrations, fibers are identified via a biholomorphism of $(\mathbb{C}^\ast)^2$. We start with the function $h_{-1}\colon F_{-1}\rightarrow[0,\infty)$, defined on a smooth fiber $F_{-1} \subset E_{-1}$, which takes the form of the standard K\"{a}hler potential
    \[
    \phi_{-1}(z_1,z_2)\coloneqq|z_1|^2+|z_2|^2+\frac{1}{|z_1|^2}+\frac{1}{|z_2|^2},
    \]
    where $z_1,z_2$ are coordinates on $(\mathbb{C}^\ast)^2$. Extend $h_{-1}$ to all smooth fibers of $\pi_{-1}$ by parallel transport. Note that $\phi_{-1}$ extends smoothly over the unique singular fiber of $\pi_{-1}$ as the critical locus is given by collapsing the circle $|z_2|=1$ to a point. Hence we obtain a map $h_{-1}\colon E_{-1} \to [0,\infty)$. Using the fiberwise biholomorphisms on $E_{-1}\cap E_0$ and $E_0\cap E_1$ to patch together the locally defined plurisubharmonic functions $h_{-1}$, $h_0$ and $h_1$, we obtain a function $h\colon W_k\rightarrow[0,\infty)$. 
    
    Since $\phi_{-1}$ is plurisubharmonic, it follows that $h_{-1}$, and hence $h$, is plurisubharmonic. It is clear from the construction that the restriction of $h$ to each fiber of $\pi$ is proper. Let $W^\mathit{in}_k$ be the subset where $h\leq r(|\pi|)$, with $r$ being a non-decreasing function of $|\pi|$, which is constant over $[0,R_0]$ for some $R_0\gg1$. Since $h$ is defined locally by parallel transport, and the global parallel transport on $W_k$ is defined by gluing together the locally defined parallel transports on $\pi_j \colon E_j \to \mathbb{C}$ (cf. \cite{sw}, Lemma 4.1), we have $dh(\xi^\#)=0$. The fact that $d^\mathbb{C}h(\xi^\#)=0$ follows from \cite[Remark 3.1]{aa}. In our case, $d^\mathbb{C}h_{-1}$ gives the fiberwise Liouville $1$-form for $\pi_{-1}$, so the Hamiltonian vector field of $d^\mathbb{C}h$ is tangent to the fibers of $\pi$, from which it follows that $\mathcal{L}_{\xi^\#}(d^\mathbb{C}h)=0$.
    	
    To define the wrapping Hamiltonian $H$, note that the elementary Morse-Bott fibration $\pi_{-1}\colon E_{-1}\rightarrow\mathbb{C}$ can be identified with the product of the standard Lefschetz fibration $p_{-1}\colon \mathbb{C}^2\rightarrow\mathbb{C}$ with a unique singular fiber $\mathbb{C}^\ast$. We define a Hamiltonian $H_{\mathbb{C}^2}$ as follows: start with a Hamiltonian $\hbar_{-1}$ on the smooth fiber of $p_{-1}$ that vanishes near the zero section of $T^\ast S^1$ and is linear with positive slope on the cylindrical end. Then extend $\hbar_{-1}$ to the total space $\mathbb{C}^2$ of $p_{-1}$ by horizontal parallel transport. Recall that $E_{-1} \cong \mathbb{C}^2 \times \mathbb{C}^\ast$, and denote by $\pi_{\mathbb{C}^2}\colon E_{-1}\rightarrow\mathbb{C}^2$ and $\pi_{\mathbb{C}^\ast}\colon E_{-1}\rightarrow\mathbb{C}^\ast$ the natural projections. Then defining $H_{-1}\coloneqq \pi_{\mathbb{C}^2}^\ast\hbar_{-1}+\pi_{\mathbb{C}^\ast}^\ast\hbar_{-1}$ yields a Hamiltonian on $E_{-1}$. Then, we extend $H_{-1}$ to $W_k$ using the coordinate transformations between different $E_i$'s. Finally, we perturb the extension of $H_{-1}$ slightly to get our wrapping Hamiltonian $H\colon W_k\rightarrow[0,\infty)$ so that it has Morse critical points. Then it follows that $H$ is proper and $d\pi(X_H)=0$. By our choice of the function $h$, $X_H$ preserves the level sets of $h$ outside of $W_k^\mathit{in}$, and it follows that $dh(X_H)=0$. The fact that $\mathcal{L}_{X_H}(d^\mathbb{C}h)=0$ is also obvious from our choices of $h$ and $H$.
    \end{proof}
    
    \Cref{lemma:setup} ensures that the maximum principles proved in \cite[Propositions 3.10 and 3.11]{aa} can be applied to the solutions of Floer equations in the constructions of the Fukaya categories below.
    
    We now introduce a class of Lagrangians called \textit{admissible}, which will later be the objects of the $A_\infty$-category $\mathcal{W}(W_k,\pi)$.
    \begin{definition}[{\cite[Definition 3.4]{aa}}]
        An \emph{admissible arc} is a proper embedding $\gamma\colon [0,\infty)\rightarrow\mathbb{C}$ that is disjoint from the critical values of $\pi$ and the ray $(-\infty,-1)$, and along which the distance from the origin is strictly increasing outside of the large disk of radius $R_0$ centered at the origin.
    \end{definition}
    \begin{definition}[{\cite[Definition 3.5]{aa}}]\label{definition:admissible_lag}
        An \emph{admissible Lagrangian submanifold} is a properly embedded exact Lagrangian submanifold $L\subset W_k$ such that
        \begin{itemize}
        	\item $\pi(L)\subset\mathbb{C}$ agrees outside of a compact subset $\Delta\subset\mathbb{C}$ with a finite union of admissible arcs which do not reenter $\Delta$,
        	\item $d^\mathbb{C}h|_L$ vanishes outside of $W_k^\mathit{in}$.
        \end{itemize}
    \end{definition}
    Given an admissible Lagrangian $L\subset W_k$ and an isotopy $\rho^t$ of $\mathbb{C}$, preserving $\Delta\cup\mathit{Crit}(\pi)$ pointwise and the ray $(-\infty,-1)\subset \mathbb{C}$ setwise, there is a unique Lagrangian isotopy of $L$, which we denote by $\rho^t(L)$, such that
    \begin{itemize}
    	\item $\rho^t(L)=L$ in $\pi^{-1}(\Delta)$,
    	\item $\rho^t(L)$ fibers over the collection of arcs which is the image of $\pi(L)$ under $\rho^t$ outside of $\pi^{-1}(\Delta)$.
    \end{itemize}
    Suppose further that the isotopy $\rho^t$ maps radial lines to radial lines away from a compact subset, and moves all radial lines other than $(-\infty,-1) \subset \mathbb{C}$ in the counterclockwise direction. In particular, $\rho^t$ preserves the admissibility of the arcs. Define
    \begin{equation}\label{eq:L_t}
        L(t)\coloneqq \phi^t(\rho^t(L)),
    \end{equation}
    where $\phi^t$ is the flow of the wrapping Hamiltonian $H$ chosen above. It follows from the commutativity of $\phi^t$ and $\rho^t$ (\cite[Lemma 3.7]{aa}) that $L(t)(t')=L(t+t')$. We in particular consider admissible Lagrangians that fiber over admissible arcs in $\mathbb{C}$.
    
    As is standard in Floer theory, we study solutions to the perturbed Cauchy--Riemann equations using the same setup as in \cite[Section 3.2]{aa} that we now give a brief review of. Let $\Sigma$ denote the complement of finitely many boundary marked points on a compact Riemann surface with boundary. Suppose that $\overline{L}$ is a smooth family of admissible Lagrangian submanifolds in $W_k$ that is parametrized by $\partial \Sigma$ such that the family is constant near the ends of $\partial \Sigma$, and such that it varies along each component of $\partial \Sigma$ by a combination of
    \begin{itemize}
        \item the flow of $X_H \otimes \eta$ where $\eta$ is a one-form on $\partial \Sigma$ such that $\int_{\partial \Sigma}\eta \leq 0$, and
            \item the lift of an admissible isotopy of $\mathbb{C}$ with support away from $\Delta \cup \Delta'$, where $\Delta'$ is a small neighborhood of $\mathit{Crit}(\pi)$.
    \end{itemize}
    Denote by $\xi^t$ the vector field that generates the isotopy $\rho^t$, and let $(\xi^t)^\#$ denote its horizontal lift to $W_k$. Let $\tau \colon \Sigma \to \mathbb{R}$ denote a function that is constant near the boundary punctures of $\Sigma$ and satisfies that along each component of $\partial \Sigma$, the arcs $\rho^\tau(\pi(\overline{L}))$ vary by an admissible isotopy that moves in the clockwise direction outside of a compact set.
    
    Let $\alpha$ be a $1$-form on $\Sigma$ such that $d\alpha$ is compactly supported, $d\alpha \leq 0$ and $\alpha|_{\partial \Sigma} \geq \eta$ pointwise. A map $u \colon \Sigma \to W_k$ satisfies the \textit{Floer equation} if
    \begin{equation}\label{equation:floer}
        (du-X_H \otimes \alpha + (\xi^\tau)^\# \otimes d \tau)^{0,1} = 0.
    \end{equation}
    
    If $L_0$ and $L_1$ are two transversely intersecting Lagrangians that are admissible in the sense of \cref{definition:admissible_lag}, we define $\mathit{CF}^\ast(L_0,L_1)$ to be the chain complex over $\mathbb{K}$ freely generated by the set $L_0 \cap L_1$, and differential defined by counting solutions $u \colon \mathbb{R} \times [0,1] \to W_k$ to the Floer equation \eqref{equation:floer} that have Lagrangian boundary conditions on $L_0 \cup L_1$, i.e., $u(\mathbb{R} \times \{i\}) \subset L_i$ for $i\in \{0,1\}$, and are asymptotic to $x_{\pm} \in L_0 \cap L_1$ as $s \to \pm \infty$, where $s$ is a coordinate in the $\mathbb{R}$-factor in the domain of $u$.
    
    \subsubsection{The category, quasi-units and continuations}
    
    The fiberwise wrapped Fukaya category $\mathcal{W}(W_k,\pi)$ is defined as a certain localization of the Fukaya--Seidel category $\mathcal{F}(\pi)$.
    
    Let $\varepsilon > 0$ be such that for each pair of admissible Lagrangians $(L_0,L_1)$ and integers $m_0 \neq m_1$, the projections of $L_0(-\varepsilon m_0)$ and $L_1(-\varepsilon m_1)$ under $\pi$ are asymptotic to different radial lines in $\mathbb{C}$, and $L_0(-\varepsilon m_0)\cap L_1(-\varepsilon m_1)$ is compact (the existence is guaranteed by \cite[Lemma 3.16]{aa}). We use the notation
    \[
    L^t \coloneqq  L(-\varepsilon t).
    \]
    \begin{definition}[Fukaya--Seidel category]\label{definition:fukaya-seidel_cat}
        The Fukaya--Seidel category $\mathcal{F}(\pi)$ is defined as the directed $A_\infty$-category whose objects are admissible Lagrangian submanifolds equipped with $\mathit{Spin}$ structures and gradings. The morphisms in $\mathcal F(\pi)$ are given by
    \begin{align*}
        \hom_{\mathcal{F}(\pi)}(L_0^{m_0},L_1^{m_1}) = \begin{cases}
    	    \mathit{CF}^\ast(L_0^{m_0},L_1^{m_1}) & \text{if } m_0 < m_1, \\
            \mathbb{K} \cdot \mathit{id} & \text{if }m_0=m_1\text{ and }L_0=L_1\\
            0 & \textrm{otherwise}
    	\end{cases}.
    \end{align*}
    The $A_\infty$-structure on $\mathcal{F}(\pi)$ is defined by solutions to the Floer equation \eqref{equation:floer} with boundary on $L_0^{m_0}\cup \dots \cup L_r^{m_r}$ for $r \in \mathbb{Z}_{\geq 2}$, whose images are contained in a bounded subset of $W_k^\mathit{in}$ by \cite[Propositions 3.10 and 3.11]{aa}.
    \end{definition}
    
    A standard construction in Floer theory (see e.g.\@ \cite[Section 6.2]{sgs}, \cite[Section 3.3]{gps1}, and \cite[Section 3.4]{aa}) produces a quasi-unit $e_m\in\mathit{CF}^0(L^m,L^{m+1})$ for every $m \in \mathbb{Z}$. In short, they are defined by counting maps $u\colon \Sigma \to W_k$ that are solutions of the Floer equation \eqref{equation:floer} where $\Sigma$ a disk with a single boundary puncture regarded as an output and moving Lagrangian boundary conditions on $L^t = L(-\varepsilon t)$ for $t\in [m,m+1]$ (see \cite[Section (8k)]{sf}). Denote by $\mathcal{Z}$ the set of all quasi-units.
    
    \begin{definition}[Fiberwise wrapped Fukaya category]\label{definition:fiberwise_wfuk}
        The \emph{fiberwise wrapped Fukaya category} associated to the Morse--Bott--Lefschetz fibration $\pi \colon W_k \to \mathbb{C}$ is defined as the localization of the Fukaya--Seidel category $\mathcal F(\pi)$ by the quasi-units
        \[
        \mathcal{W}(W_k,\pi)\coloneqq \mathcal{Z}^{-1}\mathcal{F}(\pi),
        \]
        in the sense of Lyubashenko--Ovsienko \cite{lo}. We shall use the notation
        \[
        \mathit{CW}^\ast_{\pi}(L_0,L_1) \coloneqq  \hom_{\mathcal W(W_k,\pi)}(L_0,L_1).
        \]
    \end{definition}
    
    There are also continuation maps
    \[
    F_{L_0^k,L_1^j} \colon \mathit{CF}^\ast(L_0^k, L_1^j) \longrightarrow \mathit{CF}^\ast(L_0^{k+1}, L_1^{j+1}),
    \]
    that are defined by a count of maps $u \colon \mathbb{R} \times [0,1] \to W_k$ that are solutions to the Floer equation \eqref{equation:floer} and have moving Lagrangian boundary conditions on $L_0^t$ and $L_1^t$ for $t\in [k,k+1]$ along $\mathbb{R} \times \{0\}$ and $t\in [j,j+1]$ along $\mathbb{R} \times \{1\}$, respectively. The continuation maps $F_{L_0^k,L_1^j}$ yield morphisms in $\mathcal F(\pi)$ by extending them by the identity map.
    
    For the purpose of explicitly calculating certain endomorphism $A_\infty$-algebras in $\mathcal W(W_k,\pi)$ in \cref{section:single,section:product}, we use the description of the morphisms between two objects $L_0$ and $L_1$ in $\mathcal W(W_k,\pi)$ using homotopy colimits described in \cite[Section 3.5]{aa}. In particular, the following result is essentially obtained from the proof of \cite[Corollary 3.24]{aa}.
    
    \begin{lemma}\label{lemma:hocolim}
        For any $k\in \mathbb{Z}_{\geq 0}$ there is a quasi-isomorphism
        \[
        \mathit{CW}^\ast_\pi(L_0,L_1) \cong \hocolim_{j\to \infty} \hom_{\mathcal F(\pi)}(L_0^k,L_1^j),
        \]
        where the connecting maps in the homotopy colimit are given by composition with the quasi-units $e_j \in \mathit{CF}^0(L_0^j, L_0^{j+1})$.
    \end{lemma}
    \begin{proof}
        First, by \cite[Lemma 3.22]{aa}, we have a quasi-isomorphism
        \[
        \mathit{CW}^\ast_\pi(L_0,L_1) \cong \holim_{k\to\infty}\hocolim_{j\to \infty} \hom_{\mathcal F(\pi)}(L_0^k,L_1^j).
        \]
        The proof of \cite[Lemma 3.23]{aa} shows that the natural map
        \[
        \hocolim_{j\to \infty} \hom_{\mathcal F(\pi)}(L_0^{k+1},L_1^j) \longrightarrow \hocolim_{j\to \infty} \hom_{\mathcal F(\pi)}(L_0^k,L_1^j),
        \]
        is given by pre-composition with the quasi-unit $e_k \in \mathit{CF}^0(L_0^k,L_0^{k+1})$, and is in fact a quasi-isomorphism; it follows from the compatibilities between continuation maps $F_{L_0^k,L_1^j}$ and quasi-units proven in \cite[Lemma 3.19]{aa}. In particular, the induced map on cohomology
        \[
        \hocolim_{j\to \infty} H^\ast\hom_{\mathcal F(\pi)}(L_0^{k+1},L_1^j) \longrightarrow \hocolim_{j\to \infty} H^\ast\hom_{\mathcal F(\pi)}(L_0^k,L_1^j),
        \]
        is onto, and therefore the inverse system $\{ H^\ast\hom_{\mathcal F(\pi)}(L_0^{k},L_1^j)\}_k$ satisfies the Mittag-Leffler condition (see \cite[Definition 3.5.6]{weibel}). Hence the induced composition
        \[
        HW^\ast_\pi(L_0,L_1) \longrightarrow \holim_{k\to\infty}\hocolim_{j\to \infty} H^\ast\hom_{\mathcal F(\pi)}(L_0^k,L_1^j) \longrightarrow\hocolim_{j\to \infty} H^\ast\hom_{\mathcal F(\pi)}(L_0^k,L_1^j)
        \]
        is an isomorphism for any $k\geq 0$, finishing the proof.
    \end{proof}
    
    \subsubsection{The cotangent fibers in $W_k$}\label{section:cotangent_fibers}
    
    We now introduce the specific objects in the $A_\infty$-category $\mathcal{W}(W_k,\pi)$ that are relevant for us. Under the Morse--Bott--Lefschetz fibration $\pi\colon W_k\rightarrow\mathbb{C}$, the two exact Lagrangian spheres $Q_0$ and $Q_1$ in $W_k$ project to line segments $[-1,0]$ and $[0,1]$, respectively. Choose properly embedded curves $\gamma_0,\gamma_1\colon (-\infty,\infty)\rightarrow\mathbb{C}$ such that $\gamma_0$ intersects $(-1,0)$ transversely at $-\frac{1}{2}$, and $\gamma_1$ intersects $(0,1)$ transversely at $\frac{1}{2}$. We require $\gamma_0\cap[0,1]=\emptyset$ and $\gamma_1\cap[-1,0]=\emptyset$, and moreover that the arcs $\gamma_0$ and $\gamma_1$ are asymptotically radial outside of a disk containing the critical values $-1,0,1 \in \mathbb{C}$, with $\arg(\gamma_0(s))=\theta_0$ for $s\ll0$, $\arg(\gamma_0(s))=2\pi-\theta_0$ for $s\gg0$, $\arg(\gamma_1(s))=\theta_1$ for $s\ll0$, and $\arg(\gamma_1(s))=2\pi-\theta_1$ for $s\gg0$, where $0<\theta_0<\theta_1<\frac{\pi}{2}$, see \cref{fig:base}. Note that after removing the points $-\frac{1}{2}$ and $\frac{1}{2}$ from $\gamma_0$ and $\gamma_1$, respectively, each complement is a union of two admissible arcs.
    
    \begin{figure}[!htb]
    	\centering
    	\includegraphics{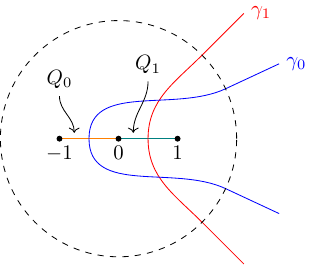}
    	\caption{The projections of $Q_0$, $Q_1$, and the arcs $\gamma_0$, $\gamma_1$.}\label{fig:base}
    \end{figure}
    
    Since the fibers $\pi^{-1}(-\frac{1}{2})$ and $\pi^{-1}(\frac{1}{2})$ are diffeomorphic to $(\mathbb{C}^\ast)^2$, we pick $(\mathbb{R}_+)^2$ as distinguished Lagrangians and denote them by $\ell_0\subset\pi^{-1}(-\frac{1}{2})$ and $\ell_1\subset\pi^{-1}(\frac{1}{2})$, respectively. The parallel transports of $\ell_0$ and $\ell_1$ along $\gamma_0$ and $\gamma_1$ define two Lagrangian submanifolds $L_0$ and $L_1$ in $W_k$. It is clear that both $L_0$ and $L_1$ are admissible Lagrangian submanifolds with respect to the Morse--Bott--Lefschetz fibration $\pi$ in the sense of \cref{definition:admissible_lag}. The compact subset $\Delta\subset\mathbb{C}$ for $L_0$ and $L_1$ are taken to be the points $\{-\frac{1}{2}\}$ and $\{\frac{1}{2}\}$, respectively. By construction $L_0$ and $L_1$ are Hamiltonian isotopic to the cotangent fibers over a point in $Q_0$ and $Q_1$, respectively. Therefore the wrapped Fukaya category $\mathcal{W}(W_k;\mathbb{K})$ is generated by $L_0$ and $L_1$ \cite{cdrgg,gps}.
    
    The goal of the rest of this section is to compute the endomorphism $A_\infty$-algebra 
    \[
    \mathcal{W}_k\coloneqq \bigoplus_{i,j \in \{0,1\}}\mathit{CW}^\ast(L_i,L_j).
    \]
    Recall from \cref{definition:fiberwise_wfuk} that the morphism space $\hom_{\mathcal W(W_k;\pi)}(L_0,L_1)$ is denoted by $\mathit{CW}^\ast_\pi(L_0,L_1)$, and define 
    \[
    \mathcal{W}_{k,\pi}\coloneqq \bigoplus_{i,j \in \{0,1\}}\mathit{CW}^\ast_\pi(L_i,L_j).
    \]
    As outlined at the beginning of the section, the Fukaya $A_\infty$-algebra $\mathcal W_k$ is a certain localization of $\mathcal W_{k,\pi}$. This will be discussed further in \cref{subsection:wrapping}. As a first step we compute $\mathcal W_{k,\pi}$ starting in \cref{section:single}. We end this subsection with a discussion about phase angles of the fiberwise Lagrangians constituting $L_0$ and $L_1$, which will be useful in the computations below.
    
    For $j \in \{0,1\}$, define the \textit{phase angles} of the fiberwise Lagrangian $\ell_{j,c}=\pi^{-1}(c)\cap L_j\subset(\mathbb{C}^\ast)^2$ at a point $c\in\gamma_j$ to be the set of angles
    \[
    \varphi_{j,c}\coloneqq \left\{\arg(x)=(\arg(x_1),\arg(x_2))\mid x=(x_1,x_2)\in\ell_{j,c}\setminus W_k^\mathit{in}\right\}.
    \]
    It follows from the definitions that $\varphi_{0,-\frac{1}{2}} = \varphi_{1,\frac{1}{2}} = \{(0,0)\}$. According to \cite[Corollary 4.19]{aa}, these angles are invariant under local parallel transports along the restrictions of the paths $\gamma_0$ and $\gamma_1$ to the bases of the elementary Morse--Bott--Lefschetz fibrations $\pi_i\colon E_i\rightarrow\mathbb{C}$, where $i \in \{-1,0,1\}$. It follows that for the Lagrangian $L_j\subset W_k$, there are at most three sets of different phase angles for the fiberwise Lagrangians $\ell_{j,c}$. More generally, the phase angles of the perturbed Lagrangian submanifolds $\ell_{j,c}(t)=\phi^t(\ell_{j,c})$ are denoted by $\varphi_{j,c}(t)$. In this case, the phase angles $\varphi_{j,c}(t)$ may depend on $c$, despite the fact that they still remain locally constant. These properties of our Lagrangian submanifolds $\ell_{j,c}(t)$ correspond to the condition of \textit{monomial admissibility} defined in \cite[Definition 4.2]{aa}.
    
    \subsection{The $A_\infty$-algebra of a single Lagrangian}\label{section:single}
    
    In this subsection, we compute the endomorphism $A_\infty$-algebras $\mathit{CW}^\ast_\pi(L_0,L_0)$ and $\mathit{CW}^\ast_\pi(L_1,L_1)$ in the fiberwise wrapped Fukaya category $\mathcal{W}(W_k,\pi)$, where $L_0$ and $L_1$ are the two Lagrangians defined in \cref{section:cotangent_fibers}.
    
    \subsubsection{The case of $L_1$}\label{sec:complex_L1}
    
    Fix $t<0$ and $t'>0$ and let $\gamma_{1,t} \coloneqq  \rho^t(\gamma_1)$. The Lagrangian submanifold $L_1(t)$ (see \eqref{eq:L_t}) is fibered over the path $\gamma_{1,t}$, i.e., it is obtained by parallel transporting of $\ell_{1,\frac 12} \subset (\mathbb{C}^\ast)^2$ along $\gamma_{1,t}$. Depending on the difference of $t'-t$, the paths $\gamma_{1,t}$ and $\gamma_{1,t'}$ may intersect at one or two points. More precisely, if $t_1$ is a value of $t$ such that $\rho^{t}$ pushes the ray $e^{-i\theta_1}\mathbb{R}_+$ past $e^{i\theta_1}\mathbb{R}_+$, then $\gamma_{1,t}\cap\gamma_{1,t'}=\{\frac{1}{2}\}$ for $t'-t<t_1$ and $\gamma_{1,t}\cap\gamma_{1,t'}=\{\frac{1}{2},c_{t',t}\}$ for $t'-t>t_1$. The latter case  is depicted in \cref{fig:L1}. We choose $t$ and $t'$ generically so that $L_1(t)\cap L_1(t')$ is transverse and contained in the compact subset $W_k^\mathit{in}\subset W_k$. Thus the intersections $L_1(t)\cap L_1(t')$ are concentrated in the fibers of $\pi$ over $\gamma_{1,t} \cap \gamma_{1,t'}$. Let $\ell_{1,-}(t'),\ell_{1,+}(t)\subset\pi^{-1}(c_{t',t})$ denote the parallel transports of $\ell_1(t')=\phi^{t'}(\ell_1)$ and $\ell_1(t)=\phi^t(\ell_1)$ from $\frac{1}{2}$ to $c_{t',t}$ along $\gamma_{1,t'}$ and $\gamma_{1,t}$, respectively. By \cite[Proposition 5.2]{aa} it follows that
    \begin{equation}\label{eq:single}
        \mathit{CF}^\ast(L_1(t'),L_1(t)) = \begin{cases}
        \mathit{CF}^\ast(\ell_1(t'),\ell_1(t))\oplus\mathit{CF}^\ast(\ell_{1,-}(t'),\ell_{1,+}(t))[1] & t'-t>t_1, \\
        \mathit{CF}^\ast(\ell_1(t'),\ell_1(t)) & 0<t'-t<t_1.
        \end{cases}
    \end{equation}
    In the above, the gradings on $L_1(t)$ and $L_1(t')$ are determined by the choices of gradings on the Lagrangians $\ell_1(t),\ell_1(t')\subset(\mathbb{C}^\ast)^2$ in the fiber and on the paths $\gamma_{1,t},\gamma_{1,t'}\subset\mathbb{C}$. In particular, since the phase angles of $\gamma_{1,t}$ and $\gamma_{1,t'}$ differ by an amount in $(-\pi,0)$ at $c_{t',t}$, the fiberwise generators above this intersection point have degree $-1$ in the Floer complex $\mathit{CF}^\ast(L_1(t'),L_1(t))$.
    
    \begin{figure}[!htb]
    	\centering
    	\includegraphics{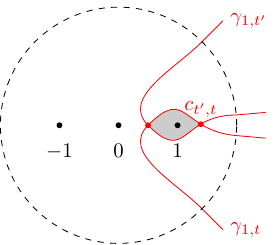}
    	\caption{The arcs $\gamma_{1,t}$ and $\gamma_{1,t'}$ with a shaded bigon contributing to the Floer differential.}\label{fig:L1}
    \end{figure}
    
    We take the almost complex structure $J$ on $W_k$ so that it coincides with the standard one outside of $W_k^\mathit{in}\cap\bigcup_{i\in\{-1,0,1\}}\pi^{-1}(\Delta_i)$. Recall from \cref{section:general_setup} that $\pi$ is $J$-holomorphic away from the disks $\Delta_i \subset \mathbb{C}$ for $i\in \{-1,0,1\}$. Thus, it follows from the open mapping theorem that solutions to the Floer equation \eqref{equation:floer} with boundary components on a union of fibered Lagrangian submanifolds (disjoint from the $\pi^{-1}(\Delta_i)$'s) are either contained in the fibers of $\pi$, or behave (away from the singular fibers) as (multi)sections of $\pi$ over regions of $\mathbb{C}$ bounded by the arcs over which the Lagrangians fiber. 
    
    We consider the case $t'-t>t_1$, where $\mathit{CF}^\ast(L_1(t'),L_1(t))$ is isomorphic to the mapping cone
    \begin{equation}\label{eq:cone}
    \mathit{CF}^\ast(\ell_1(t'),\ell_1(t))\oplus\mathit{CF}^\ast(\ell_{1,-}(t'),\ell_{1,+}(t))[1], \quad \partial=\begin{pmatrix}
    \partial_0 & \mathfrak{s} \\ 0 & \partial_1
    \end{pmatrix},
    \end{equation}
    where the diagonal entries are the Floer differentials on the fiberwise Floer complexes, and the off-diagonal term
    \[
    \mathfrak{s}=\mathfrak{s}_{\ell_1,t',t}\colon \mathit{CF}^\ast(\ell_{1,-}(t'),\ell_{1,+}(t))\longrightarrow\mathit{CF}^{\ast}(\ell_1(t'),\ell_1(t))
    \]
    is a chain map defined by a signed count of $J$-holomorphic sections of $\pi\colon W_k\rightarrow\mathbb{C}$ over the region of $\mathbb{C}$ bounded by $\gamma_{1,t}$ and $\gamma_{1,t'}$ (the shaded bigon in \cref{fig:L1}).
    
    The fiberwise Floer complexes $\mathit{CF}^\ast(\ell_1(t'),\ell_1(t))$ in \eqref{eq:cone} depend on the choice of the fiberwise wrapping Hamiltonian $H$. Since $\ell_1\subset\pi^{-1}(\frac{1}{2})$ has constant phase angles at infinity, one can take $H$ to be a convex Hamiltonian whose slope at infinity in the fiber $\pi^{-1}(\frac{1}{2})$ is chosen so that the fiberwise Lagrangian intersections $\varphi^{t'}(\ell_1(t')) \cap \varphi^t(\ell_1(t))$ lie in a compact subset, where $\varphi^t$ denotes the time-$t$ flow of $X_H$. Next, \cite[Propositions 5.6 and 5.7]{aa} imply that
    \begin{itemize}
    	\item For $t'-t$ generic, there is a square $P_1(t'-t)\subset\mathbb{R}^2$ (centered at the origin) such that the complex $\mathit{CF}^\ast(\ell_1(t'),\ell_1(t))$ is supported in degree $0$ with vanishing differential, and
    	\begin{equation}\label{eq:add}
    	\mathit{CF}^0(\ell_1(t'),\ell_1(t))\cong\bigoplus_{p\in P_1(t'-t)\cap(2\pi\mathbb{Z})^2}\mathbb{K}\cdot\vartheta_p^{t'\rightarrow t},
    	\end{equation}
    	where $\vartheta_p^{t'\rightarrow t}$ corresponds to intersections of $\ell_1(t')$ and $\ell_1(t)$ inside $\pi^{-1}(\frac{1}{2})$.
    	\item For generic choices of $t''>t'>t$, the Floer product is given by
    	\begin{equation}\label{eq:prod}
    	\vartheta_p^{t'\rightarrow t}\cdot\vartheta_{p'}^{t''\rightarrow t'}=\vartheta_{p+p'}^{t''\rightarrow t},
    	\end{equation}
    	where $p\in P_1(t'-t)\cap(2\pi\mathbb{Z})^2$ and $p'\in P_1(t''-t')\cap(2\pi\mathbb{Z})^2$.
    \end{itemize}
    
    For the complex $\mathit{CF}^\ast(\ell_{1,-}(t'),\ell_{1,+}(t))$, since the isotopy $\rho^t$ for $t<0$ does not act as a convex Hamiltonian isotopy, we need to pick $t_1'\geq t_1$ large enough so that there is a convex Hamiltonian isotopy from $\ell_{1,+}(t)$ to $\ell_{1,-}(t')$ in the fiber $\pi^{-1}(c_{t',t})$ if $t'-t>t_1'$. We can then apply \cite[Propositions 5.6 and 5.14]{aa}, which imply that with $t'-t,t''-t',t''-t$ generic, we have
    \begin{itemize}
    	\item For $t'-t>t_1'$, the fiberwise Floer complex $\mathit{CF}^\ast(\ell_{1,-}(t'),\ell_{1,+}(t))$ is supported in degree $0$ with vanishing differential, and there is a square $Q_1(t'-t)\subset\mathbb{R}^2$ such that
    	\begin{equation}\label{eq:add1}
    	\mathit{CF}^0(\ell_{1,-}(t'),\ell_{1,+}(t))\cong\bigoplus_{q\in Q_1(t'-t)\cap(2\pi\mathbb{Z})^2}\mathbb{K}\cdot\zeta_q^{t'\rightarrow t},
    	\end{equation}
    	where $\zeta_q^{t'\rightarrow t}$ corresponds to intersections of $\ell_{1,-}(t')$ and $\ell_{1,+}(t)$ in $\pi^{-1}(c_{t',t})$.
    	\item For $t'-t>t_1'$, the Floer product is given by
    	\begin{equation}\label{eq:p1}
    	\zeta_q^{t'\rightarrow t}\cdot\vartheta_p^{t''\rightarrow t'}=\zeta_{p+q}^{t''\rightarrow t},
    	\end{equation}
    	where $q\in Q_1(t'-t)\cap(2\pi\mathbb{Z})^2$, $p\in P_1(t''-t')\cap(2\pi\mathbb{Z})^2$. Similarly, for $t''-t'>t_1'$,
    	\begin{equation}\label{eq:p2}
    	\vartheta_p^{t'\rightarrow t}\cdot\zeta_q^{t''\rightarrow t'}=\zeta_{p+q}^{t''\rightarrow t},
    	\end{equation}
    	where $p\in P_1(t''-t')\cap(2\pi\mathbb{Z})^2$, $q\in Q_1(t'-t)\cap(2\pi\mathbb{Z})^2$.
    	\item For $t'-t>t_1'$ and $t''-t'>t_1'$,
    	\begin{equation}\label{eq:p3}
    	\zeta_p^{t\rightarrow t'}\cdot\zeta_q^{t''\rightarrow t'}=0
    	\end{equation}
    	for all $p,q\in Q_1(t'-t)\cap(2\pi\mathbb{Z})^2$.
    \end{itemize}
    
    \begin{remark}\label{remark:weight}
    As opposed to \cite[Proposition 5.14]{aa}, no additional constants can appear as coefficients for the products in \eqref{eq:p1} and \eqref{eq:p2} since we do not weight the counts by symplectic areas when working in the exact category. 
    \end{remark}
    
    In particular, for $t'-t>t_1'$ generic, the Floer complex $\mathit{CF}^\ast(L_1(t'),L_1(t))$ is given by the mapping cone
    \begin{equation}\label{eq:vc}
    \left\{\bigoplus_{q\in Q_1(t'-t)\cap(2\pi\mathbb{Z})^2}\mathbb{K}\cdot\zeta_q^{t'\rightarrow t}[1]\overset{\mathfrak{s}}{\longrightarrow}\bigoplus_{p\in P_1(t'-t)\cap(2\pi\mathbb{Z})^2}\mathbb{K}\cdot\vartheta_p^{t'\rightarrow t}\right\},
    \end{equation}
    where the shift $[1]$ is due to the fact that the degree $0$ generator $\zeta_q^{t'\rightarrow t}$ in the fiber has degree $-1$ when regarded as a generator of $\mathit{CF}^\ast(L_1(t'),L_1(t))$. It follows that
    \[
    \mathfrak{s}_{\ell_1,t',t}(\zeta_0^{t'\rightarrow t})=\sum_{p\in P_1(t'-t)\cap(2\pi\mathbb{Z})^2}c_{\bar{p}}\vartheta_p^{t'\rightarrow t}
    \]
    for some $c_{\bar{p}}\in\mathbb{K}$, where $\bar{p}=\frac{p}{2\pi}$. It follows from \eqref{eq:p1}, \eqref{eq:p2} and the Leibniz rule that for any $q\in Q_1(t'-t)\cap(2\pi\mathbb{Z})^2$,
    \[
    \mathfrak{s}_{\ell_1,t',t}(\zeta_q^{t'\rightarrow t})=\sum_{p\in P_1(t'-t)\cap(2\pi\mathbb{Z})^2}c_{\bar{p}}\vartheta_{p+q}^{t'\rightarrow t}.
    \]
    For the statement of the next proposition, we introduce the Laurent polynomial
    \[
    g_2(x)=\sum_{p\in P_1(t'-t)\cap(2\pi\mathbb{Z})^2}c_{\bar{p}}x^{\bar{p}}\in\mathbb{K}[x_1^{\pm1},x_2^{\pm1}].
    \]
    
    \begin{proposition}\label{proposition:diff-b}
    The fiberwise wrapped Fukaya $A_\infty$-algebra $\mathit{CW}^\ast_\pi(L_1,L_1)$ is quasi-isomorphic to the dg algebra with underlying associative algebra $\mathbb{K}[x_1^{\pm1},x_2^{\pm1},\beta]/(\beta^2)$, grading
    \[
    |x_1|=|x_2|=0,\quad |\beta|=-1,
    \]
    and differentials
    \[
    dx_1=dx_2=0, \quad d\beta=g_2(x).
    \]
    \end{proposition}
    \begin{proof}
    By \cref{lemma:hocolim}, $\mathit{CW}^\ast_\pi(L_1,L_1)$ is quasi-isomorphic to
    \[
    \hocolim_{t\to \infty}\mathit{CF}^\ast(L_1(t'),L_1(t)),
    \]
    for $t'-t$ generic. Let $t'-t>t_1'$ and identify $\vartheta_p^{t'\rightarrow t}$ with $x^{\bar{p}}$ for all $p\in P_1(t'-t)\cap(2\pi\mathbb{Z})^2$, and $\zeta_q^{t'\rightarrow t}$ with $x^{\bar{q}}$ for all $q\in Q_1(t'-t)\cap(2\pi\mathbb{Z})^2$, where $\bar p = \frac{p}{2 \pi}$ and $\bar q = \frac{q}{2\pi}$. We may therefore identify $\mathit{CF}^\ast(L_1(t'),L_1(t))$ with a subcomplex of the mapping cone
    \begin{equation}\label{eq:lc}
    \left\{\mathbb{K}[x_1^{\pm1},x_2^{\pm1}]\overset{g_2}{\longrightarrow}\mathbb{K}[x_1^{\pm1},x_2^{\pm1}]\right\},
    \end{equation}
    consisting of those Laurent polynomials whose Newton polytopes are contained in the squares $\frac{1}{2\pi}P_1(t'-t)$ and $\frac{1}{2\pi}Q_1(t'-t)$.
    	
    It follows from \eqref{eq:p1} and \eqref{eq:p2} that under the above identification, the product operations on these Floer complexes are given by multiplications of Laurent polynomials. By \cite[Proposition 5.16]{aa}, the connecting maps in the homotopy colimit are given by composition with $\vartheta_0^{t'\rightarrow t}$. Therefore, it follows that the connecting maps in the homotopy colimit for $t'-t$ sufficiently large are given by inclusion maps, and the homotopy colimit of the complexes \eqref{eq:vc} is therefore quasi-isomorphic to the mapping cone \eqref{eq:lc}. After identifying the image of $\zeta_0^{t'\rightarrow t}$ in the homotopy colimit with $\beta$, we see that the complex $\mathit{CW}^\ast_\pi(L_1,L_1)$ is generated by $x_1^{\pm1},x_2^{\pm1}$ and $\beta$ with the prescribed differentials. The product structures are determined by \eqref{eq:prod}, \eqref{eq:p1}, \eqref{eq:p2} and \eqref{eq:p3}.
    	
    Since the complex $\mathit{CW}^\ast_\pi(L_1,L_1)$ is supported in degrees $-1$ and $0$, the $A_\infty$-operations $\mu^k$ on $\mathit{CW}^\ast_\pi(L_1,L_1)$ must vanish for any $k\geq4$, and $\mu^3$ can be non-zero only when applied to three entries in degree $0$. Since the only degree $0$ generators of $\mathit{CW}^\ast_\pi(L_1,L_1)$ lie in the fiber $\pi^{-1}(\frac{1}{2})$, which has trivial $A_\infty$-operations, we see that $\mu^3$ also vanishes.
    \end{proof}

    \begin{remark}
    Since $\mathit{HW}^\ast_\pi(L_1,L_1)\cong\mathbb{K}[x_1^{\pm1},x_2^{\pm_1}]/(g_2(x))$ is supported in degree $0$, it is not hard to see that the $A_\infty$-algebra $\mathit{CW}_\pi^\ast(L_1,L_1)$ is actually formal.
    \end{remark}
    
    In order to compute $\mathit{CW}^\ast_\pi(L_1,L_1)$, it remains to identify the Laurent polynomial $g_2(x)$ appearing in \cref{proposition:diff-b}. To do this, we appeal to a trick of Abouzaid--Auroux \cite{aa}, where $L_1(t)$ and $L_1(t')$ are replaced with Lagrangian cylinders. For each $x=(x_1,x_2)\in(\mathbb{K}^\ast)^2$, let $\tau_x \subset\pi^{-1}(\frac{1}{2})$ be a product torus in $(\mathbb{C}^\ast)^2$, equipped with a rank $1$ local system over $\mathbb{K}$ whose holonomy around the $i$-th $S^1$-factor is $x_i^{-1}$ for $i\in\{1,2\}$. Note that $\tau_x$ (including its local system) is invariant under parallel transport between the fibers of $\pi$, and we can choose the wrapping Hamiltonian $H$ in \cref{lemma:setup} so that $\tau_x$ is invariant under its flow. For any $t\in \mathbb{R}$, let $T_x(t)$ be the admissible Lagrangian submanifold obtained by parallel transporting $\tau_x\subset\pi^{-1}(\frac{1}{2})$ over the admissible arc $\gamma_{1,t}$.
    
    \begin{remark}
    In \cite{aa}, the authors work in the case where $\mathbb{K}=\mathbb{C}$ or a Novikov field, so they consider product tori $\mathfrak{t}_x\subset(\mathbb{C}^\ast)^n$ with fixed moment map coordinates. Since we do not weight our curve counting by symplectic areas, the position of $\tau_x\subset(\mathbb{C}^\ast)^n$ is not important for us. Also, their torus $\mathfrak{t}_x$ is equipped with a unitary local system since one can always rescale by the weights of symplectic areas to recover any point of $(\mathbb{K}^\ast)^n$, but we consider local systems with holonomies in $(\mathbb{K}^\ast)^2$ due to the lack of rescaling. Despite these minor differences, most of the arguments in \cite{aa} can be applied without essential changes to our case, and when this is the situation, we shall directly cite the results from \cite{aa} below instead of rerunning the same arguments.
    \end{remark}
    
    We now consider the Floer complex $\mathit{CF}^\ast(T_x(t'),T_x(t))$ for $t-t'>t_1$ generic and fixed $x\in(\mathbb{K}^\ast)^2$. Similarly to \eqref{eq:single}, this Floer complex is quasi-isomorphic to the mapping cone of a chain map
    \[
    \mathit{CF}^\ast(\tau_x,\tau_x)\longrightarrow \mathit{CF}^\ast(\tau_x,\tau_x),
    \]
    which is defined by counting $J$-holomorphic sections of $\pi\colon W_k\rightarrow\mathbb{C}$ over the region bounded by $\gamma_{1,t}$ and $\gamma_{1,t'}$.
    
    \begin{proposition}\label{proposition:Tx}
    When $t'-t>t_1$, the complex $\mathit{CF}^\ast(T_x(t'),T_x(t))$ is given by the following mapping cone
    \begin{equation}\label{eq:sx}
    \left\{H^\ast(T^2;\mathbb{K})[1]\overset{\mathfrak{s}_x}{\longrightarrow}H^\ast(T^2;\mathbb{K})\right\},
    \end{equation}
    where the connecting differential $\mathfrak{s}_x$, defined by counting $J$-holomorphic sections of $\pi\colon W_k\rightarrow\mathbb{C}$ over the region bounded by $\gamma_{1,t}$ and $\gamma_{1,t'}$, with coincidence conditions on cycles in $\tau_x$ at $-1$ and $c_{t',t}$, is given by the multiplication by $g_2(x)\in\mathbb{K}$. 
    \end{proposition}
    \begin{proof}
    One can choose the fiberwise wrapping Hamiltonian $H$ so that it is a perfect Morse function when restricted to the Lagrangian tori $\tau_x$ in the fibers $\pi^{-1}(\frac{1}{2})$ and $\pi^{-1}(c_{t',t})$, so that each Floer complex $\mathit{CF}^\ast(\tau_x,\tau_x)$ can be identified with a copy of $H^\ast(T^2;\mathbb{K})$ (possibly up to a shift). The rest is the same as the proof of \cite[Proposition 5.21]{aa}. Note that \cref{remark:weight} applies here as well, which explains why there is no additional constant before $g_2(x)$ in our statement.
    \end{proof}
    
    Since the differential $\mathfrak{s}_x$ in \eqref{eq:sx} is multiplication by some element of $\mathbb{K}$, it is enough to determine the image of the generator of $H^0(T^2;\mathbb{K})$ under $\mathfrak{s}_x$, which amounts to counting $J$-holomorphic sections of $\pi\colon W_k\rightarrow\mathbb{C}$ whose boundary passes through some prescribed input point in the fiber $\pi^{-1}(\frac{1}{2})$.
    
    \begin{proposition}\label{proposition:disk}
    For $t'-t>t_1$ and $x\in(\mathbb{K}^\ast)^2$, there are two homotopy classes of $J$-holomorphic sections of $\pi\colon W_k\rightarrow\mathbb{C}$ with boundary on $T_x(t')\cup T_x(t)$, corresponding to the two irreducible components of the singular fiber $\pi^{-1}(1)$. For each such class, the moduli space of $J$-holomorphic sections consists of a single orbit under the fiberwise $T^2$-action, and the count of these sections passing through any given point of $\tau_x\subset\pi^{-1}(\frac{1}{2})$ is equal to $1$.
    \end{proposition}
    \begin{proof}
    Let $S\subset\mathbb{C}$ be the region bounded by $\gamma_{1,t}$ and $\gamma_{1,t'}$. Since a $J$-holomorphic section of $\pi\colon W_k\rightarrow\mathbb{C}$ over $S$ has intersection number $1$ with $\pi^{-1}(1)=F^+\cup F^-$, where $F^+\cong F^-\cong\mathbb{C}^\ast\times\mathbb{C}$, it must intersect exactly one of $F^+$ and $F^-$, and be disjoint from the other irreducible component of $\pi^{-1}(1)$, On the other hand, the section has intersection number $0$ with $\pi^{-1}(-1)$ and $\pi^{-1}(0)$, so it is disjoint from both. It thus suffices to consider $J$-holomorphic disks contained in
    \[
    W_k^\pm\coloneqq W_k\setminus\left(\pi^{-1}(-1)\cup\pi^{-1}(0)\cup F^\pm\right),
    \]
    which is biholomorphic to $(\mathbb{C}\setminus\{-1,0\})\times(\mathbb{C}^\ast)^2$. This enables us to reduce the problem to that of finding $J$-holomorphic maps $S\rightarrow(\mathbb{C}^\ast)^2$ satisfying appropriate boundary conditions. The rest of the argument is the same as \cite[Proposition 5.24]{aa}.
    \end{proof}
    
    \begin{proposition}\label{proposition:g}
    The Laurent polynomial $g_2(x)$ in \cref{proposition:diff-b} is given by $g_2(x)=x_1^k+x_2$.
    \end{proposition}
    \begin{proof}
    The count of $J$-holomorphic sections of $\pi\colon W_k\rightarrow\mathbb{C}$ contributing to the chain map $\mathfrak{s}_x$ in \eqref{eq:sx} is weighted by their holonomies of the local systems on $T_x(t)$ and $T_x(t')$ (obtained by parallel transporting the local system on the fiberwise Lagrangian torus $\tau_x$) along the boundaries. We first deform $S\subset\mathbb{C}$ (the region bounded by $\gamma_{1,t}$ and $\gamma_{1,t'}$, see \cref{fig:L1}) to a disk centered at $1\in\mathbb{C}$ by deforming the (non-smooth) Lagrangian $\bigcup_{s\in\partial S}\tau_x$ to a product torus such that the holonomy of the local system along the deformed curve (i.e., the circle in the base surrounding the critical value $1\in\mathbb{C}$) in the deformed boundary is trivial. Let $[D_{\pm}]\in H_2(W_k,T_x(t')\cup T_x(t))\cong\mathbb{Z}^2$ be the two relative homology classes of the sections of $\pi\colon W_k\rightarrow\mathbb{C}$ contributing to $\mathfrak{s}_x$ obtained from \cref{proposition:disk}. By convention $[D_+]$ is the homology class whose intersection number with $F^+$ is $1$. Following the proof of \cite[Proposition 5.26]{aa}, we define a $2$-chain $D_{+-}$ on $W_k$ relative to $T_x(t')\cup T_x(t)$ as follows. Pick a point $s\in\partial S$ and a path $\gamma$ connecting $s$ to $1$. The chain $D_{+-}$ fibers over $\gamma$, with the fiber over each point $z\in\gamma$ being an orbit of the $S^1$-action rotating the disk fibers in the tubular neighborhood of $\{x_1+x_2^k=1\}\subset(\mathbb{C}^\ast)^2$, where the coordinates $(x_1,x_2)$ on the fiber $\pi^{-1}(z)$ are given by parallel transport of the chosen coordinates on $\pi^{-1}(\frac{1}{2})$. Note that when approaching the singular fiber over $1$ along the path $\gamma$, these $S^1$-orbits collapse to a point on $F^+\cap F^-\cong\mathbb{C}^\ast$. For a suitable choice of the orientation, the chain $D_{+-}$ intersects $F^-$ once positively, and $F^+$ once negatively, so we have $[D_{+-}]=[D_-]-[D_+]$ in $H_2(W_k,T_x(t')\cup T_x(t))$. Since the restriction of the Lagrangian $T_x(t')\cup T_x(t)$ to $\partial S$ is isotopic to the product torus $\tau_x\times S^1\cong T^3$ in $W_k$, and $\partial D_{\pm}$ represent the class $(-k,1,1)\in H_1(T^3)\cong\mathbb{Z}^3$, by our choice of the local system on $\tau_x$, the holonomy along the circle $\partial D_{+-}$ is $x_1^kx_2^{-1}$. Combining this with \cref{proposition:disk}, this is precisely the contribution to the Floer differential of $\beta$.
    \end{proof}
    
    \begin{remark}\label{remark:Z}
    When working over $\mathbb{Z}$ instead of a field $\mathbb{K}$, in order to apply the above trick for the computation of the Laurent polynomial $g_2(x)\in\mathbb{Z}[x_1^{\pm1},x_2^{\pm1}]$, one needs to consider the fiberwise tori $\tau_x\subset(\mathbb{C}^\ast)^2$ for $x\in(\mathbb{Z}\setminus\{0\})^2$. Since the holonomies along the circle factors are assumed to be $x_i^{-1}$ for $i \in \{1,2\}$, this causes a priori problems. In order to make sense of the above argument, we need to work temporarily over $\mathbb{Q}$, where $g_2(x)$ for any $x\in(\mathbb{Z}\setminus\{0\})^2$ belongs. Equip the tori $\tau_x$ with $\mathbb{Q}$-local systems and study the Floer complex given by the mapping cone
    \[
    \left\{H^\ast(T^2;\mathbb{Q})[1]\overset{g_2(x)}{\longrightarrow}H^\ast(T^2;\mathbb{Q})\right\}
    \]
    for $x\in(\mathbb{Z}\setminus\{0\})^2$. However, since our computation shows that $g_2(x)$ only involves positive powers of $x_1$ and $x_2$, we actually have $g_2(x)\in\mathbb{Z}$ for a particular $x$. The same remark also applies to the undetermined Laurent polynomials $f(x)$ and $g_1(x)$ appearing in the computations of the Floer differential and the Floer product below, where we always have $f(x),g_1(x)\in\mathbb{Z}[x_1,x_2]$. Since the field $\mathbb{Q}$ is only used for auxiliary purposes, our computations here make sense over $\mathbb{Z}$.
    \end{remark}
    
    \subsubsection{The case of $L_0$}\label{sec:complex_L0}
    The computation of the fiberwise wrapped Fukaya $A_\infty$-algebra $\mathit{CW}_\pi^\ast(L_0,L_0)$ is very similar to that of $L_1$, with some additional complexities arising in counting $J$-holomorphic sections that will be mentioned below. For later purposes, we shall mention here the facts parallel to \eqref{eq:add}, \eqref{eq:prod} and \eqref{eq:add1}. We perturb the Lagrangian submanifold $L_0$ using the fiberwise Hamiltonian vector field $X_H$ and the isotopy $\rho^t$ on the base, and consider the Floer complex $\mathit{CF}^\ast(L_0(t'),L_0(t))$, where $t'-t>t_0$ is generic and sufficiently large in the sense that $\rho^{t_0}$ pushes the ray $e^{-i\theta_0}\mathbb{R}_+$ past $e^{i\theta_0}\mathbb{R}_+$, see \cref{fig:L0}.
    \begin{figure}[!htb]
    	\centering
    	\includegraphics{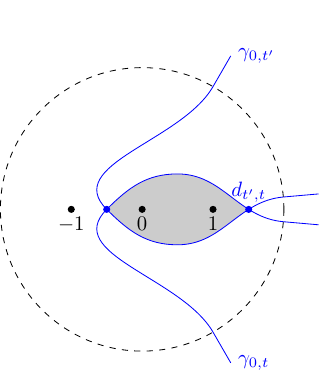}
    	\caption{The arcs $\gamma_{0,t}$ and $\gamma_{0,t'}$ with a shaded bigon contributing to the Floer differential.}\label{fig:L0}
    \end{figure}
    Note that $\gamma_{0,t}\cap\gamma_{0,t'}=\{-\frac{1}{2},d_{t',t}\}$. Define $\ell_0(t)=\phi^t(\ell_0)$ and define $\ell_{0,-}(t')$ and $\ell_{0,+}(t)$ to be the parallel transports of $\ell_0(t')$ and $\ell_0(t)$, respectively, from $-\frac{1}{2}$ to $d_{t',t}$. As in the case of $\ell_1$, we can choose the wrapping Hamiltonian $H$ so that when $t'-t$, $t''-t'$ and $t''-t$ are generic and $t_0'\geq t_0$ is sufficiently large, the following statements are true:
    
    \begin{itemize}
    	\item The fiberwise Floer complex $\mathit{CF}^\ast(\ell_0(t'),\ell_0(t))$ is supported in degree $0$ with vanishing differential, and
    	\[
        \mathit{CF}^0(\ell_0(t'),\ell_0(t))\cong\bigoplus_{p\in P_0(t'-t)\cap(2\pi\mathbb{Z})^2}\mathbb{K}\cdot\delta_p^{t'\rightarrow t}
        \]
    	for some square $P_0(t'-t)\subset\mathbb{R}^2$.
    	\item We have the fiberwise Floer product
    	\[
        \delta_p^{t'\rightarrow t}\cdot\delta_p^{t''\rightarrow t'}=\delta_{p+p'}^{t''\rightarrow t}
        \]
    	for $p\in P_0(t'-t)\cap(2\pi\mathbb{Z})^2$ and $p'\in P_0(t''-t')\cap(2\pi\mathbb{Z})^2$.
    	\item For $t'-t>t_0'$, the fiberwise Floer complex $\mathit{CF}^\ast(\ell_{0,-}(t'),\ell_{0,+}(t))$ is supported in degree $0$ with vanishing differential, and
    	\[
        \mathit{CF}^0(\ell_{0,-}(t'),\ell_{0,+}(t))\cong\bigoplus_{q\in Q_0(t'-t)\cap(2\pi\mathbb{Z})^2}\mathbb{K}\cdot\eta_q^{t'\rightarrow t}
        \]
    	for some square $Q_0(t'-t)\subset\mathbb{R}^2$.
    \end{itemize}
    
    \begin{proposition}\label{proposition:f}
    The $A_\infty$-algebra $\mathit{CW}^\ast_\pi(L_0,L_0)$ is quasi-isomorphic to the dg algebra with underlying associative algebra $\mathbb{K}[x_1^{\pm1},x_2^{\pm1},\chi]/(\chi^2)$, grading
    \[
    |x_1|=|x_2|=0, \quad |\chi|=-1,
    \]
    and differentials
    \[
    dx_1=dx_2=0, \quad d\chi=f(x),
    \]
    where $f(x)\in\mathbb{K}[x_1^{\pm1},x_2^{\pm1}]$.
    \end{proposition}
    \begin{proof}
    As in the proof of \cref{proposition:diff-b}, the generators $x_i^\pm$ come from $\delta_p^{t'\rightarrow t}$ after passing to the homotopy colimit as $t\rightarrow\infty$, while the generator $\chi$ comes from $\eta_q^{t'\rightarrow t}$. As in the proof of \cref{proposition:diff-b}, the differential
    \[
    \mathfrak{s}=\mathfrak{s}_{\ell_0,t',t}\colon \mathit{CF}^0(\ell_{0,-}(t'),\ell_{0,+}(t))\longrightarrow\mathit{CF}^0(\ell_0(t'),\ell_0(t)),
    \]
    which is given by counting $J$-holomorphic sections of $\pi\colon W_k\rightarrow\mathbb{C}$ over the region bounded by the arcs $\gamma_{0,t}$ and $\gamma_{0,t'}$ (the shaded bigon in \cref{fig:L0}), can be identified with the multiplication by a Laurent polynomial $f(x)\in\mathbb{K}[x_1^{\pm1},x_2^{\pm1}]$ after passing to the homotopy colimit as $t\rightarrow\infty$.
    \end{proof}

    The computation of the Laurent polynomial $f(x)$ is slightly more complicated than the computation of $g_2(x)$ in \cref{proposition:g}, as the bigon formed by $\gamma_{0,t}$ and $\gamma_{0,t'}$ (depicted in \cref{fig:L0}) contains two critical values as opposed to only a single one. This will be postponed to the next subsection, see \cref{corollary:f=ug}.
    
    \subsection{The product structure}\label{section:product}
    
    In this subsection, we compute the Floer products that involve both of the admissible Lagrangians $L_0$ and $L_1$ defined in \cref{section:setup} and whose Floer complexes are computed in \cref{sec:complex_L1,sec:complex_L0}. To compute the products we first perturb the Lagrangians $L_0$ and $L_1$ using the isotopy $\rho^t$ on the base and the fiberwise Hamiltonian flow $X_H$ to obtain transversely intersecting admissible Lagrangian submanifolds $L_0(t)$, $L_1(t')$ and $L_0(t'')$ with $t''>t'>t$ and $t''-t>t_0$, which project to the arcs $\gamma_{0,t}$, $\gamma_{1,t'}$ and $\gamma_{0,t''}$, see \cref{fig:product}. Let $\gamma_{1,t'} \cap \gamma_{0,t} = \{a_{t',t},b_{t',t}\}$ and $\gamma_{1,t'} \cap \gamma_{0,t''} = \{a_{t'',t'},b_{t'',t'}\}$. Let $\ell_{0,-}(t'')\subset\pi^{-1}(b_{t'',t})$ be the parallel transport of $\ell_0(t'')=\phi^{t''}(\ell_0)$ from $-\frac{1}{2}$ to $b_{t'',t'}$ along $\gamma_{0,t''}$, $\ell_{1,-}(t') \subset\pi^{-1}(b_{t'',t'})$ be the parallel transport of $\ell_1(t')$ from $\frac{1}{2}$ to $b_{t'',t'}$, $\ell_{1,+}(t') \subset\pi^{-1}(b_{t'',t'})$ be the parallel transport of $\ell_1(t)$ from $\frac{1}{2}$ to $b_{t',t}$, and $\ell_{0,+}(t) \subset\pi^{-1}(b_{t',t})$ be the parallel transport of $\ell_0(t)$ from $-\frac{1}{2}$ to $b_{t',t}$. It is clear that the Floer cochain complex $\mathit{CF}^\ast(L_1(t'),L_0(t))$ is isomorphic to the complex
    \[
    \left(\mathit{CF}^\ast(\ell_{1,+}(t'),\ell_{0,+}(t))\oplus\mathit{CF}^0(\ell_{1,-}(t'),\ell_{0,+}(t))[1],\begin{pmatrix}\partial_{0} & \mathfrak{s}  \\ 0 & \partial_{1}\end{pmatrix}\right)
    \]
    where $\mathfrak{s}$ is a chain map
    \[
    \mathfrak{s}=\mathfrak{s}_{\ell_1,\ell_0,t',t}\colon\mathit{CF}^0(\ell_{1,-}(t'),\ell_{0,+}(t))[1]\longrightarrow\mathit{CF}^0(\ell_{1,+}(t'),\ell_{0,+}(t))
    \]
    that is defined by a count of $J$-holomorphic sections of $\pi \colon W_k\rightarrow\mathbb{C}$ over the region bounded by $\gamma_{0,t}$ and $\gamma_{1,t'}$ (the shaded bigon in \cref{fig:product_A} with vertices $b_{t',t}$ and $a_{t',t}$ surrounding the critical value $1$). The Floer complex $\mathit{CF}^\ast(L_0(t''),L_1(t'))$ has a similar description, with the fiberwise generators above the intersection point $b_{t'',t'}$ lying in degree $0$ and the generators above $a_{t'',t'}$ lying in degree $-1$. The same argument used when dealing with the case of $\mathit{CF}^\ast(\ell_{1,-}(t'),\ell_{1,+}(t))$ implies that when $t'-t$, $t''-t'$ and $t''-t$ generic and the real numbers $t_{01}^\pm,t_{10}^\pm$ are sufficiently large, we can choose the convex wrapping Hamiltonian $H$ such that the following is true:
    \begin{itemize}
    \item For $t''-t'>t_{01}^-$, the fiberwise Floer complex $\mathit{CF}^\ast(\ell_{0,-}(t''),\ell_{1,-}(t'))$ is supported in degree $0$ with vanishing differential, and there is a square $Q_{01}^-(t''-t')\subset\mathbb{R}^2$ such that
        \[
        \mathit{CF}^0(\ell_{0,-}(t''),\ell_{1,-}(t'))\cong\bigoplus_{q\in Q_{01}^-(t''-t')\cap(2\pi\mathbb{Z})^2}\mathbb{K}\cdot\kappa_q^{t''\rightarrow t'},
        \]
        where $\kappa_q^{t''\rightarrow t'}$ corresponds to an intersection point of $\ell_{0,-}(t'')$ and $\ell_{1,-}(t')$ in the fiber $\pi^{-1}(b_{t'',t'})$.
    	\item For $t''-t'>t_{01}^+$, the fiberwise Floer complex $\mathit{CF}^\ast(\ell_{0,-}(t''),\ell_{1,+}(t'))$ is supported in degree $0$ with vanishing differential, and there is a square $Q_{01}^+(t''-t')\subset\mathbb{R}^2$ such that
    	\[
        \mathit{CF}^0(\ell_{0,-}(t''),\ell_{1,+}(t'))\cong\bigoplus_{q\in Q_{01}^+(t''-t')\cap(2\pi\mathbb{Z})^2}\mathbb{K}\cdot\lambda_q^{t''\rightarrow t'},
        \]
        where $\lambda_q^{t''\rightarrow t'}$ corresponds to an intersection point of $\ell_{0,-}(t'')$ and $\ell_{1,+}(t')$ in the fiber $\pi^{-1}(a_{t'',t'})$.
    	\item For $t'-t>t_{10}^+$, the fiberwise Floer complex $\mathit{CF}^\ast(\ell_{1,+}(t'),\ell_{0,+}(t))$ is supported in degree $0$ with vanishing differential, and there is a square $Q_{10}^+(t'-t)\subset\mathbb{R}^2$ such that
        \begin{equation}\label{eq:mu}
        \mathit{CF}^0(\ell_{1,+}(t'),\ell_{0,+}(t))\cong\bigoplus_{q\in Q_{10}^+(t'-t)\cap(2\pi\mathbb{Z})^2}\mathbb{K}\cdot\mu_q^{t'\rightarrow t},
        \end{equation}
    	where $\mu_q^{t'\rightarrow t}$ corresponds to an intersection point of $\ell_{1,+}(t')$ and $\ell_{0,+}(t)$ in the fiber $\pi^{-1}(b_{t',t})$.
    	\item For $t'-t>t_{10}^-$, the fiberwise Floer complex $\mathit{CF}^\ast(\ell_{1,-}(t'),\ell_{0,+}(t'))$ is supported in degree $0$ with vanishing differential, and there is a square $Q_{10}^-(t'-t)\subset\mathbb{R}^2$ such that
    	\[
        \mathit{CF}^0(\ell_{1,-}(t'),\ell_{0,+}(t))\cong\bigoplus_{q\in Q_{10}^-(t'-t)\cap(2\pi\mathbb{Z})^2}\mathbb{K}\cdot\nu_q^{t'\rightarrow t},
        \]
    	where $\nu_q^{t'\rightarrow t}$ corresponds to an intersection point of $\ell_{1,-}(t')$ and $\ell_{0,+}(t)$ in the fiber $\pi^{-1}(a_{t',t})$.
    \end{itemize}
    
    \begin{figure}[!htb]
    	\centering
        \begin{subfigure}{0.45\textwidth}
            \includegraphics{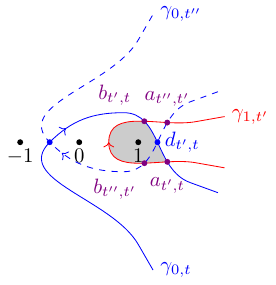}
            \caption{}\label{fig:product_A}
        \end{subfigure}
    	\begin{subfigure}{0.45\textwidth}
    	    \includegraphics{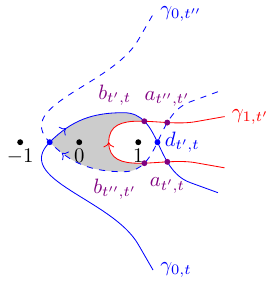}
            \caption{}\label{fig:product_B}
    	\end{subfigure}
    	\caption{The arcs $\gamma_{0,t}$, $\gamma_{1,t'}$ and $\gamma_{0,t''}$, and shaded triangles contributing to the Floer product \eqref{eq:tp1}.}\label{fig:product}
    \end{figure}
    
    In this case, there is a non-trivial triangle product, namely
    \begin{equation}\label{eq:tp1}
    \mathit{CF}^0(L_1(t'),L_0(t))\otimes\mathit{CF}^0(L_0(t''),L_1(t'))\longrightarrow\mathit{CF}^0(L_0(t''),L_0(t)),
    \end{equation}
    which is determined by counting $J$-holomorphic sections of $\pi\colon W_k\rightarrow\mathbb{C}$ above the triangle with vertices $b_{t'',t'},b_{t',t}$ and $-\frac{1}{2}$, which is shaded in \cref{fig:product_B}. In terms of the fiberwise Floer complexes, this is given by a map
    \[
    \mathfrak{p}\colon\bigoplus_{q\in Q_{10}^+(t'-t)\cap(2\pi\mathbb{Z})^2}\mathbb{K}\cdot\mu_q^{t'\rightarrow t}\otimes\bigoplus_{q'\in Q_{01}^-(t''-t')\cap(2\pi\mathbb{Z})^2}\mathbb{K}\cdot\kappa_{q'}^{t''\rightarrow t'}\longrightarrow\bigoplus_{p\in P_0(t''-t)\cap(2\pi\mathbb{Z})^2}\mathbb{K}\cdot\delta_p^{t''\rightarrow t}.
    \]
    After passing to the homotopy colimits as $t,t'\rightarrow\infty$ and identifying the generators $\mu_q^{t'\rightarrow t}$ with $x^{\bar{q}}$, $\kappa_{q'}^{t''\rightarrow t'}$ with $x^{\bar{q}'}$ and $\delta_p^{t''\rightarrow t}$ with $x^{\bar{p}}$ (where $q=2\pi\bar{p}$, $q'=2\pi\bar{q}'$ and $p=2\pi\bar{p}$ as before), the map $\mathfrak{p}$ is identified with the composition
    \[
    \mathbb{K}[x_1^{\pm1},x_2^{\pm1}]\otimes\mathbb{K}[x_1^{\pm1},x_2^{\pm1}]\longrightarrow\mathbb{K}[x_1^{\pm1},x_2^{\pm1}]\xrightarrow{g_1}\mathbb{K}[x_1^{\pm1},x_2^{\pm1}],
    \]
    where the first map is given by the product of Laurent polynomials and the second map is multiplication with $g_1(x)\in\mathbb{K}[x_1^{\pm1},x_2^{\pm1}]$ to be determined below.
    
    To find $g_1(x)$, one needs to count $J$-holomorphic sections of $\pi\colon W_k\rightarrow\mathbb{C}$ with boundary conditions on $L_0(t)\cup L_1(t')\cup L_0(t'')$. The same trick of replacing $L_0(t)$, $L_1(t')$ and $L_0(t'')$ with the corresponding Lagrangian cylinders $T_x(t)$, $T_x(t')$ and $T_x(t'')$ (obtained by parallel transporting the product tori $\tau_x\subset\pi^{-1}(-\frac{1}{2})$ and $\tau_x\subset\pi^{-1}(\frac{1}{2})$ over the paths $\gamma_{0,t}$, $\gamma_{1,t'}$ and $\gamma_{0,t''}$, respectively) used in the computation of $g_2(x)$ in the proof of \cref{proposition:g} still works here, thanks to the following analog of \cref{proposition:Tx}.
    
    \begin{lemma}\label{lemma:px}
    For $t'-t>t_{10}^+$ and $t''-t'>t_{01}^-$, the Floer product
    \begin{equation}\label{eq:TTT}
    \mathit{CF}^\ast(T_x(t'),T_x(t))\otimes\mathit{CF}^\ast(T_x(t''),T_x(t'))\longrightarrow\mathit{CF}^\ast(T_x(t''),T_x(t))
    \end{equation}
    is given by
    \[
    H^\ast(T^2;\mathbb{K})\otimes H^\ast(T^2;\mathbb{K})\overset{\mathfrak{p}_x}{\longrightarrow}H^\ast(T^2;\mathbb{K}),
    \]
    where $\mathfrak{p}_x$ is the cup product composed with multiplication by $g_1(x)\in\mathbb{K}$.
    \end{lemma}
    \begin{proof}
    For $t'-t>t_0$, the Lagrangian submanifolds $L_0(t')$ and $T_x(t)$ intersect transversely at a unique point in the fibers $\pi^{-1}(-\frac{1}{2})$ and $\pi^{-1}(d_{t',t})$, denote by $\varepsilon_{x,0}^{t'\rightarrow t}\in\mathit{CF}^0(\ell_0(t'),\tau_x)$ and $\sigma_{x,0}^{t'\rightarrow t}\in\mathit{CF}^0(\ell_{0,-}(t'),\tau_x)$ the corresponding Floer generators. The Floer complex $\mathit{CF}^\ast(L_0(t'),T_x(t))$ is generated by $\varepsilon_{x,0}^{t'\rightarrow t}$ (in degree $0$) and $\sigma_{x,0}^{t'\rightarrow t}$ (in degree $-1$), and the differential is given by counting $J$-holomorphic sections of $\pi\colon W_k\rightarrow\mathbb{C}$ over the region bounded by the curves $\gamma_{0,t}$ and $\gamma_{0,t'}$. Similarly, the complex $\mathit{CF}^\ast(L_1(t''),T_x(t))$ is generated by $\varepsilon_{x,1}^{t''\rightarrow t}\in\mathit{CF}^0(\ell_{1,+}(t''),\tau_x)$ and $\sigma_{x,1}^{t''\rightarrow t}\in\mathit{CF}^0(\ell_{1,-}(t''),\tau_x)$, which correspond to intersection points in the fibers $\pi^{-1}(b_{t'',t})$ and $\pi^{-1}(a_{t'',t})$, respectively, see \cref{fig:product2}. 
    \begin{figure}[!htb]
    	\centering
        \includegraphics{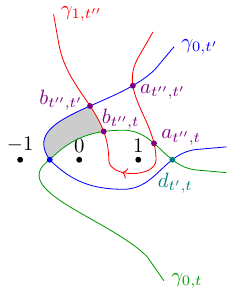}
    	\caption{The arcs $\gamma_{0,t}$, $\gamma_{0,t'}$ and $\gamma_{1,t''}$, and a shaded triangle contributing to the Floer product \eqref{eq:LLT}.}\label{fig:product2}
    \end{figure}
    Consider the triangle product
    \begin{equation}\label{eq:LLT}
    \mathit{CF}^\ast(L_0(t'),T_x(t))\otimes\mathit{CF}^\ast(L_1(t''),L_0(t'))\longrightarrow\mathit{CF}^\ast(L_1(t''),T_x(t)),
    \end{equation}
    for $t''>t'>t$, which is obtained by counting $J$-holomorphic sections of $\pi\colon W_k\rightarrow\mathbb{C}$ over the triangle formed by the curves $\gamma_{0,t}$, $\gamma_{0,t'}$ and $\gamma_{1,t''}$, with vertices labeled by $-\frac{1}{2}$, $b_{t'',t'}$ and $b_{t'',t}$, depicted in \cref{fig:product2}. Since its interior contains no critical values of $\pi$, we can shrink the triangle to a single point via compactly supported isotopies of the arcs $\gamma_{0,t}$, $\gamma_{0,t'}$ and $\gamma_{1,t''}$, and reduce the count to that inside a fiber, see for example the proof of \cite[Proposition 5.14]{aa}. This implies that for $q=2\pi\bar{q}\in Q_{10}^+(t''-t')\cap(2\pi\mathbb{Z})^2$, we have
    \begin{equation}\label{eq:id1}
    \varepsilon_{x,0}^{t'\rightarrow t}\cdot\mu_q^{t''\rightarrow t'}=x^{\bar{q}}\varepsilon_{x,1}^{t''\rightarrow t},
    \end{equation}
    where $\mu_q^{t''\rightarrow t'}$ are generators in the fiber $\pi^{-1}(b_{t'',t'})$, see \eqref{eq:mu}. On the other hand, it follows from \cite[Proposition 5.15]{aa} that for $p=2\pi\bar{p}\in P_0(t''-t')\cap(2\pi\mathbb{Z})^2$, we have the identity
    \begin{equation}\label{eq:id2}
    \varepsilon_{x,0}^{t'\rightarrow t}\cdot\delta_p^{t''\rightarrow t'}=x^{\bar{p}}\varepsilon_{x,0}^{t''\rightarrow t}
    \end{equation}
    computing the Floer product in the fiber $\pi^{-1}(-\frac{1}{2})$.
    
    By our assumption for the Floer product \eqref{eq:tp1}, with $t''-t'>t_{10}^+$ and $t'''-t''>t_{01}^-$, which are both assumed to be generic, we have
    \[
    \mu_{q}^{t''\rightarrow t'}\cdot\kappa_{q'}^{t'''\rightarrow t''}=g_1(x)\delta_{q+q'}^{t'''\rightarrow t'},
    \]
    where $q\in Q_{10}^+(t''-t)\cap(2\pi\mathbb{Z})^2$ and $q'\in Q_{01}^-(t'''\rightarrow t'')\cap(2\pi\mathbb{Z})^2$. Multiplying by $\varepsilon_{x,0}^{t'\rightarrow t}$ on the left and use the identities \eqref{eq:id1} and \eqref{eq:id2}, we obtain
    \begin{equation}\label{eq:ek}
    \varepsilon_{x,1}^{t'\rightarrow t}\cdot\kappa_{q'}^{t''\rightarrow t'}=x^{\bar{q}'}g_1(x)\varepsilon_{x,0}^{t''\rightarrow t},
    \end{equation}
    where we have renamed $t''$ and $t'''$ to $t'$ and $t''$, respectively. This computes the Floer product
    \begin{equation}\label{eq:LLT1}
    \mathit{CF}^0(L_1(t'),T_x(t))\otimes\mathit{CF}^0(L_0(t''),L_1(t'))\longrightarrow\mathit{CF}^0(L_0(t''),T_x(t)).
    \end{equation}
    
    In a similar fashion, we can further replace the $L_0(t'')$ in \eqref{eq:LLT1} by multiplying the right hand side of \eqref{eq:ek} by the generator $\psi_{x,0}^{t'''\rightarrow t''}\in\mathit{CF}^0(\tau_x,\ell_0(t'''))$. Using a shrinking argument similar to the one described above, one can prove the identities
    \[
    \kappa_{q'}^{t''\rightarrow t'}\cdot\psi_{x,0}^{t'''\rightarrow t''}=x^{\bar{q}'}\psi_{x,1}^{t'''\rightarrow t'}
    \]
    for $q'\in Q_{01}^-(t'''\rightarrow t'')\cap(2\pi\mathbb{Z})^2$, where $\psi_{x,1}^{t'''\rightarrow t'}\in\mathit{CF}^0(\tau_x,\ell_{1,+}(t'''))$ is the unique generator, and
    \begin{equation}\label{eq:id3}
    \varepsilon_{x,0}^{t''\rightarrow t}\cdot\psi_{x,0}^{t'''\rightarrow t''}=e_x^{t'''\rightarrow t},
    \end{equation}
    where $e_x^{t'''\rightarrow t}\in\mathit{CF}^0(\tau_x,\tau_x)$ corresponds to the identity in the fiber above the intersection point $(\gamma_{0,t}\cap\gamma_{0,t'''})\setminus\{-\frac{1}{2}\}$. It follows that
    \begin{equation}\label{eq:id4}
    \varepsilon_{x,1}^{t'\rightarrow t}\cdot\psi_{x,1}^{t''\rightarrow t'}=g_1(x)e_x^{t''\rightarrow t}.
    \end{equation}
    Again, we have renamed $t'''$ to $t''$ in \eqref{eq:id4}. Note that this determines the product
    \[
    \mathit{CF}^0(L_1(t'),T_x(t))\otimes\mathit{CF}^0(T_x(t''),L_1(t'))\longrightarrow\mathit{CF}^0(T_x(t''),T_x(t)).
    \]
    To finish the proof, we pick generic $t<t'<t''<t'''<t''''$ and compute the products
    \begin{equation}\label{eq:id5}
    \begin{split}
    e_x^{t''\rightarrow t}\cdot g_1(x)e_x^{t''''\rightarrow t''}&=\varepsilon_{x,0}^{t'\rightarrow  t}\cdot\psi_{x,0}^{t''\rightarrow t'}\cdot\varepsilon_{x,1}^{t'''\rightarrow t''}\cdot\psi_{x,1}^{t''''\rightarrow t'''} \\
    &=\varepsilon_{x,0}^{t'\rightarrow t}\cdot g_1(x)\mu_0^{t'''\rightarrow t'}\cdot\psi_{x,1}^{t''''\rightarrow t'''} \\
    &=g_1(x)\varepsilon_{x,1}^{t'''\rightarrow t}\cdot\psi_{x,1}^{t''''\rightarrow t'''} \\
    &=g_1(x)^2e_x^{t''''\rightarrow t},
    \end{split}
    \end{equation}
    where the first equality follows from \eqref{eq:id3} and \eqref{eq:id4}, the second equality follows from a domain shrinking argument as used in the derivation of \eqref{eq:id1}, the third equality is an application of \eqref{eq:id1}, and the last equality is a consequence of \eqref{eq:id4}. Note that \eqref{eq:id5} determines the triangle product \eqref{eq:TTT} in degree $0$. The claim in higher degrees follows from a direct application of \cite[Lemma 5.22]{aa}.
    \end{proof}
    
    \begin{proposition}\label{proposition:u}
    We have $g_1(x)=x_1-1$.
    \end{proposition}
    \begin{proof}
    Consider the admissible Lagrangian cylinders $T_x(t),T_x(t'),T_x(t'')\subset W_k$ defined by parallel transporting the Lagrangian brane $\tau_x\subset\pi^{-1}(\pm\frac{1}{2})$ over the arcs $\gamma_{0,t},\gamma_{1,t'}$ and $\gamma_{0,t''}$ in \cref{fig:product_B}. By \cref{lemma:px}, in order to determine $g_1(x)$, we need to compute the product \eqref{eq:TTT}, which is defined by counting $J$-holomorphic sections of $\pi\colon W_k\rightarrow\mathbb{C}$ with boundaries on $T_x(t)\cup T_x(t')\cup T_x(t'')$ over the triangle with vertices $b_{t'',t'}$, $b_{t',t}$ and $-\frac{1}{2}$ (the shaded triangle in \cref{fig:product_B}, but remember now that the fibers over the arcs are $\tau_x$). There are two relative homotopy classes of such $J$-holomorphic sections, corresponding to the two irreducible components of the singular fiber $\pi^{-1}(0)$. The same argument as in \cite[Proposition 5.24]{aa} implies that the count of $J$-holomorphic sections in each of these homotopy classes is $1$. To identify their contributions to $g_1(x)$, we deform $T_x(t)\cup T_x(t')\cup T_x(t'')$ to a product torus as in the proof of \cref{proposition:g} and consider the holonomies along the boundaries of the two disks projecting to the circle surrounding the unique critical value $0$. Applying the same argument as in the proof of \cref{proposition:g}, it is not hard to see that the holonomy weighted count is $x_1+1$. However, note that the orientation on the arc $\gamma_{1,t'}$ from $b_{t'',t'}$ to $b_{t',t}$ reverses the one on the product torus, see \cref{fig:product}. As a consequence, the orientation on the boundary $\partial D_{+-}$ of the chain appeared in the proof of \cref{proposition:g} also needs to be reversed. This gives $g_1(x)=x_1-1$.
    \end{proof}

    Let $a\in\mathit{CW}^0_\pi(L_1,L_0)$ and $b\in\mathit{CW}^0_\pi(L_0,L_1)$ denote the generators corresponding to $\mu_0^{t'\rightarrow t}\in\mathit{CF}^0(\ell_{1,+}(t'),\ell_{0,+}(t))$ and $\kappa_0^{t''\rightarrow t'}\in\mathit{CF}^0(\ell_{0,-}(t''),\ell_{1,-}(t))$ after passing to the homotopy colimits as $t,t'\rightarrow\infty$, respectively. Also introduce the idempotents $e_0$ and $e_1$ for the objects $L_0$ and $L_1$ in the fiberwise wrapped Fukaya category $\mathcal{W}(W_k,\pi)$, respectively. It follows from \cref{lemma:px} and \cref{proposition:u} that in the endomorphism $A_\infty$-algebra $\mathcal{W}_{k,\pi}$ of $L_0$ and $L_1$ in $\mathcal{W}(W_k,\pi)$, now regarded as an $A_\infty$-algebra over the semisimple ring $\Bbbk=\mathbb{K}e_0\oplus\mathbb{K}e_1$, we have the relation
    \[
    a\cdot b=(x_1-1)e_0.
    \]
    
    We also perturb the Lagrangians $L_0$ and $L_1$ to $L_1(t)$, $L_0(t')$ and $L_0(t'')$ where $t<t'<t''$ are generic, see \cref{fig:p2}. In this case, the non-trivial triangle product
    \[
    \mathit{CF}^\ast(L_0(t'),L_1(t))\otimes\mathit{CF}^\ast(L_1(t''),L_0(t'))\longrightarrow\mathit{CF}^\ast(L_1(t''),L_1(t))
    \]
    can be computed in a similar way as \eqref{eq:p1}. The outcome is that after passing to the homotopy colimits as $t,t'\rightarrow\infty$, we have the relation
    \[
    b\cdot a=(x_1-1)e_1
    \]
    in the fiberwise wrapped Fukaya $A_\infty$-algebra $\mathcal{W}_{k,\pi}$.
    
    \begin{figure}[!htb]
    	\centering
    	\includegraphics{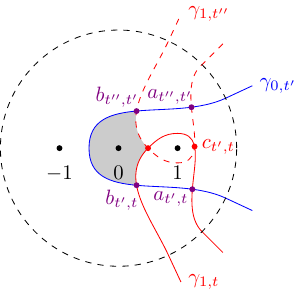}
    	\caption{The arcs $\gamma_{1,t}$, $\gamma_{0,t'}$ and $\gamma_{1,t''}$ together with a shaded triangle contributing to a Floer product.}\label{fig:p2}
    \end{figure}

    We now determine the Laurent polynomial $f(x)\in\mathbb{K}[x_1^{\pm1},x_2^{\pm1}]$ in \cref{proposition:f}, which follows from the following relation between different generators. Recall the description of $\mathit{CW}^\ast_\pi(L_1,L_1)$ in \cref{proposition:diff-b}.
    
    \begin{proposition}\label{proposition:relation}
    We have $\chi=a\cdot\beta\cdot b$ in $\mathcal{W}_{k,\pi}$.
    \end{proposition}
    \begin{proof}
    We choose arcs $\gamma_{0,t},\gamma_{1,t'},\gamma_{1,t''}\subset\mathbb{C}$ as in \cref{fig:relation}, which define the Lagrangian submanifolds $L_0(t),L_1(t')$ and $L_0(t'')$. Assume that $t'-t>t_{10}^+$ and $t''-t'>t_1'$, and consider the Floer product
    \[
    \mathit{CF}^\ast(L_1(t'),L_0(t))\otimes\mathit{CF}^\ast(L_1(t''),L_1(t'))\longrightarrow\mathit{CF}^\ast(L_1(t''),L_0(t)).
    \]
    Is is given by counting $J$-holomorphic sections of $\pi\colon W_k\rightarrow\mathbb{C}$ over the triangle with vertices $b_{t',t}$, $c_{t'',t'}$ and $a_{t'',t}$, see \cref{fig:relation_A}. Since this triangle does not contain any critical value of $\pi$, we can shrink it to a single point, which reduces the count to that inside a fiber, as in the proof of \cref{lemma:px}. In particular, in terms of fiberwise generators, we have (there is a more general formula for $\mu_q^{t'\rightarrow t}\cdot\zeta_{q'}^{t''\rightarrow t'}$, but only the $q=q'=0$ case is relevant for us)
    \[
    \mu_0^{t'\rightarrow t}\cdot\zeta_0^{t''\rightarrow t'}=\nu_0^{t''\rightarrow t}.
    \]
    Now pick the arc $\gamma_{0,t'''}\subset\mathbb{C}$ as shown in \cref{fig:relation} and consider the Floer product
    \[
    \mathit{CF}^\ast(L_1(t''),L_0(t))\otimes\mathit{CF}^\ast(L_0(t'''),L_1(t''))\longrightarrow\mathit{CF}^\ast(L_0(t'''),L_0(t))
    \]
    for $t''-t>t_{10}^-$ and $t'''-t''>t_{01}^-$. Again, this is determined by the count of $J$-holomorphic sections of $\pi\colon W_k\rightarrow\mathbb{C}$ over the triangle with vertices $a_{t'',t}$, $b_{t''',t''}$ and $d_{t''',t}$ as shown in \cref{fig:relation_B}, which contains no critical values of $\pi$. By the same shrinking argument as above, we obtain
    \[
    \nu_0^{t''\rightarrow t}\cdot\kappa_0^{t'''\rightarrow t''}=\eta_0^{t'''\rightarrow t}.
    \]
    Passing to the homotopy colimits as $t,t''\rightarrow\infty$ completes the proof.
    \end{proof}
    
    \begin{figure}[!htb]
    	\centering
        \begin{subfigure}{0.45\textwidth}
        \includegraphics{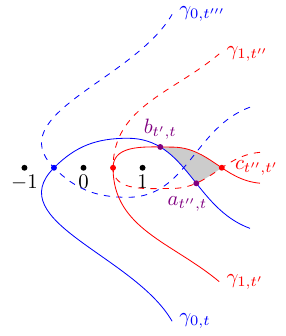}
        \caption{}\label{fig:relation_A}
        \end{subfigure}
        \begin{subfigure}{0.45\textwidth}
        \includegraphics{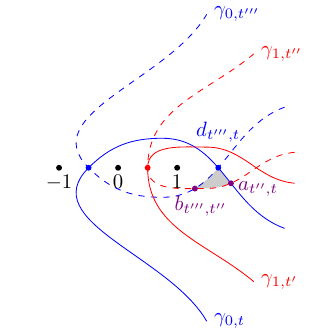}
        \caption{}\label{fig:relation_B}
        \end{subfigure}
    	
    \caption{The arcs $\gamma_{0,t}$, $\gamma_{1,t'}$, $\gamma_{1,t''}$ and $\gamma_{0,t'''}$ and two triangles contributing to the Floer product.}\label{fig:relation}
    \end{figure}
    
    \begin{corollary}\label{corollary:f=ug}
    We have $f(x)=g_1(x)g_2(x)=(x_1-1)(x_1^k+x_2)$.
    \end{corollary}
    \begin{proof}
    Since $|a|=|b|=0$, and the generators of $\mathcal{W}_{k,\pi}$ are supported in non-positive degrees, we have $da=db=0$. Thus
    \[
    d\chi=a\cdot d\beta\cdot b=a\cdot g_2(x)e_1\cdot b=g_2(x)a\cdot be_0=g_2(x)g_1(x)e_0.
    \]
    \end{proof}
    
    \subsection{Wrapping and localization}\label{subsection:wrapping}
    
    In this subsection, we finish our computation of the wrapped Fukaya category $\mathcal{W}(W_k;\mathbb{K})$ by localizing the fiberwise wrapped Fukaya category $\mathcal{W}(W_k,\pi)$. We start with the following result concerning the $A_\infty$-structure on $\mathcal{W}_{k,\pi}$.
    
    \begin{lemma}\label{lemma:dg}
    The $A_\infty$-operations $\mu^{\geq3}$ vanish in the fiberwise wrapped Fukaya $A_\infty$-algebra $\mathcal{W}_{k,\pi}$.
    \end{lemma}
    \begin{proof}
    Recall that the gradings on the admissible Lagrangians $L_0,L_1\subset W_k$ are induced from the choices of gradings on the fiberwise Lagrangians $\ell_0\subset\pi^{-1}(-\frac{1}{2})$ and $\ell_1\subset\pi^{-1}(\frac{1}{2})$ and orientations on the arcs $\gamma_0,\gamma_1\subset\mathbb{C}$. The same is true for the isotoped Lagrangians $L_0(t),L_1(t')$, with the fiberwise Lagrangians replaced with $\ell_0(t),\ell_1(t')$ and the base arcs replaced with $\gamma_{0,t},\gamma_{1,t'}$. Now if we choose the standard grading on the fiberwise Lagrangian $\ell_0(t)$, the phase angles of $\gamma_{0,t}$ and $\gamma_{1,t'}$ will determine the gradings of the generators of the Floer cochain complex $\mathit{CF}^\ast(L_1(t'),L_0(t))$. If the projections in $\mathbb{C}$ of $L_0(t)$ and $L_1(t')$ have phase angles that differ by an amount in $(\pi,2\pi)$ (resp.\@ $(0,\pi)$) at their intersection point $c_{t',t}$, for $t'-t$ sufficiently large, then by our choice of the fiberwise gradings, the generators of $\mathit{CF}^\ast(L_1(t'),L_0(t))$ must lie in degree $-1$ (resp.\@ $0$).
    
    Since the isotopy $\rho^t$ in the base is stopped at the ray $(-\infty,-1) \subset \mathbb{C}$ and the fiberwise generators of $\mathcal{W}_{k,\pi}$ all have degree $0$, it follows that the generators of the fiberwise wrapped Fukaya $A_\infty$-algebra $\mathcal{W}_{k,\pi}$ are supported in degrees $0$ and $-1$. This forces all the $A_\infty$-operations $\mu^k$ with $k\geq4$ to vanish for degree reasons, and $\mu^3\neq0$ only when it is applied to three generators of degree $0$. Now suppose that $u\colon\Sigma\rightarrow W_k$ is a solution of the Floer equation \eqref{equation:floer} contributing to $\mu^3$, with boundary on some Hamiltonian perturbations of the Lagrangian submanifolds $L_0$ or $L_1$. Since it cannot be the case that all the asymptotes of $u$ lie inside a fiber, by the open mapping theorem, $u$ must be a $J$-holomorphic section of $\pi\colon W_k\rightarrow\mathbb{C}$. Suppose that $\mu^3$ is applied to three degree $0$ generators projecting to $c_1,c_2,c_3\in\mathbb{C}$ under $\pi$, see \cref{fig:crossings}, which depicts the projection of $u\colon\Sigma\rightarrow W_k$. By our above discussion regarding gradings and arc orientations, the only two options for the orientations of the arcs are depicted in \cref{fig:crossings1,fig:crossings2}. In either case, the generator corresponding to $c_4$ is forced to have degree $0$, which means that such $J$-holomorphic section can not exist, since $\mu^3$ has degree $-1$.
    \end{proof}
    
    \begin{figure}[!htb]
    	\centering
        \begin{subfigure}{0.45\textwidth}
        \centering
        \includegraphics{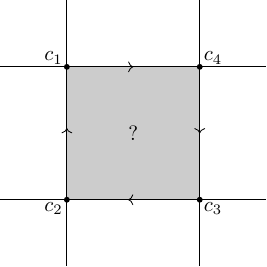}
        \caption{}\label{fig:crossings1}
        \end{subfigure}
        \begin{subfigure}{0.45\textwidth}
        \centering
        \includegraphics{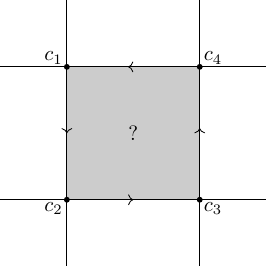}
        \caption{}\label{fig:crossings2}
        \end{subfigure}
    	\caption{Image under $\pi$ of a putative solution to the Floer equation contributing to $\mu^3$; the question mark indicates that one or more critical values can appear in the interior of the square.}\label{fig:crossings}
    \end{figure}
    
    \begin{corollary}
    The $A_\infty$-algebra $\mathcal{W}_{k,\pi}$ is quasi-isomorphic to the dg algebra over $\Bbbk=\mathbb{K}e_0\oplus\mathbb{K}e_1$ with generators $x_1^{\pm1},x_2^{\pm1},a,b,\beta$, which have gradings
    \[
    |x_1|=|x_2|=|a|=|b|=0,\quad |\beta|=-1,
    \]
    satisfy the relations
    \[
    a b=(x_1-1)e_0,\quad b a=(x_1-1)e_1,\quad \beta a=b\beta=0,\quad \beta^2=0,
    \]
    and have differentials
    \[
    da=db=0,\quad d\beta=(x_1^k+x_2)e_1.
    \]
    \end{corollary}
    \begin{proof}
    This follows essentially from our computations in \cref{section:single,section:product}. The generators of $\mathit{CW}^{-1}_\pi(L_0,L_1)$ and $\mathit{CW}^{-1}_\pi(L_1,L_0)$, which correspond to intersection points in the fibers above the points $a_{t',t}$ and $a_{t'',t'}$ in \cref{fig:product,fig:p2}, do not appear in the above list because they can be identified with $a\beta$, $\beta b$, or their multiplications by $x_1-1$. We have the relations $\beta a=b\beta=0$ because the Floer complexes do not compose in that order.
    \end{proof}
    
    In order to compute the fully wrapped Fukaya category $\mathcal{W}(W_k;\mathbb{K})$, we need to remove the stop corresponding to the ray $(-\infty,-1) \subset \mathbb{C}$ (i.e., a fiber $\pi^{-1}(c)$ for some $c\in(-\infty,-1)$). According to Abouzaid--Seidel \cite{asl}, there is an acceleration functor
    \begin{equation}\label{eq:acc}
    \mathcal{W}(W_k,\pi)\longrightarrow\mathcal{W}(W_k;\mathbb{K})
    \end{equation}
    associated to the stop removal, which realizes the wrapped Fukaya category $\mathcal{W}(W_k;\mathbb{K})$ as the localization along the natural transformation from the identity functor to the $A_\infty$-functor
    \[
    \mu\colon\mathcal{W}(W_k,\pi)\longrightarrow\mathcal{W}(W_k,\pi)
    \]
    induced by the monodromy of $\pi$ along a large circle, see also \cite[Theorem 1.2]{zso} and \cite[Theorem 1.20]{gps}. The monodromy functor $\mu$ preserves admissibility of Lagrangians, and for an admissible Lagrangian submanifold $L$, its image $\mu(L)$ is called the \textit{once wrapping} of $L$. For example, when $L=L_0$, its once wrapping $\mu(L_0)$ is defined by parallel transporting $\ell_0\subset\pi^{-1}(-\frac{1}{2})$ along the arc $\mu(\gamma_0)$ depicted in \cref{fig:wrapping_unper}.
    \begin{figure}[!htb]
        \centering
        \includegraphics{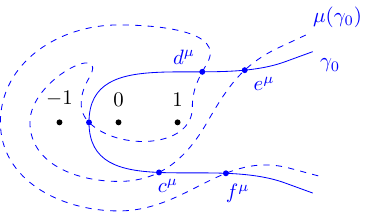}
        \caption{Projections of $L_0$ and its once wrapping $\mu(L_0)$.}\label{fig:wrapping_unper}
    \end{figure}
    The natural transformation $\mathit{id}\Rightarrow\mu$ is determined by a continuation element
    \[
    w_L\in\mathit{CW}_\pi^{-2}(\mu(L),L)
    \]
    for any admissible Lagrangian $L$.
    \begin{figure}[!htb]
    	\centering
        \begin{subfigure}{0.45\textwidth}
        \includegraphics{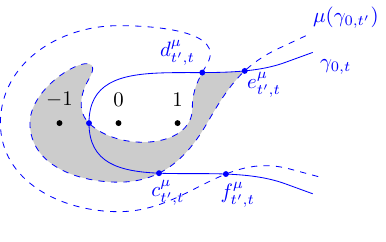}
        \caption{}\label{fig:wrapping_A}
        \end{subfigure}
        \begin{subfigure}{0.45\textwidth}
        \includegraphics{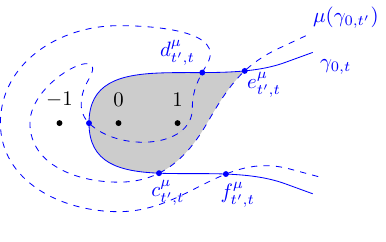}
        \caption{}\label{fig:wrapping_B}
        \end{subfigure}
    	\caption{Projections of $L_0(t)$ and its once wrapping $\mu(L_0(t'))$ with two bigons contributing to the Floer differential.}\label{fig:wrapping}
    \end{figure}
    
    Since the admissible Lagrangian submanifold $L_0$ has two ends, the continuation element $w_{L_0}$ is also the sum of two parts, lying in the fibers $\pi^{-1}(e^\mu)$ and $\pi^{-1}(f^\mu)$, respectively. More precisely, for $t'-t$ large enough and generic, the Floer complex $\mathit{CF}^\ast(\mu(L(t')),L(t))$ is isomorphic to the direct sum
    \begin{equation}\label{eq:w1cpx}
    \begin{split}
    \mathit{CF}^\ast(\ell_0(t'),\ell_0(t))&\oplus\mathit{CF}^\ast(\ell_{0,+}^\mu(t'),\ell_{0,-}(t))[1]\oplus\mathit{CF}^\ast(\ell_{0,-}^\mu(t'),\ell_{0,+}(t))[1] \\
    &\oplus\mathit{CF}^\ast(\ell^\mu_{0,++}(t'),\ell_{0,++}(t))[2]\oplus\mathit{CF}^\ast(\ell^\mu_{0,--}(t'),\ell_{0,--}(t))[2],
    \end{split}
    \end{equation}
    of the fiberwise Floer complexes over the points $-\frac{1}{2}$, $c_{t',t}^\mu$, $d_{t',t}^\mu$, $e_{t',t}^\mu$ and $f_{t',t}^\mu$, respectively (the perturbed equivalent of the points $c^\mu,d^\mu,e^\mu$ and $f^\mu$ in \cref{fig:wrapping_unper}). Now define the following Lagrangian submanifolds:
    \begin{itemize}
        \item Define $\ell_{0,+}^\mu(t')$ to be the parallel transport of $\ell_0(t')$ from $-\frac{1}{2}$ to $c_{t',t}^\mu$ along $\mu(\gamma_{0,t'})$.
        \item Define $\ell_{0,-}^\mu(t')$ to be the parallel transport of $\ell_0(t')$ from $-\frac{1}{2}$ to $d_{t',t}^\mu$ along $\mu(\gamma_{0,t'})$.
        \item Define $\ell^\mu_{0,++}(t')$ to be the parallel transport of $\ell_0(t')$ from $-\frac{1}{2}$ to $e_{t',t}^\mu$ along $\mu(\gamma_{0,t'})$.
        \item Define $\ell_{0,++}(t)$ to be the parallel transport of $\ell_0(t)$ from $-\frac{1}{2}$ to $e_{t',t}^\mu$ along $\gamma_{0,t}$.
        \item Define $\ell^\mu_{0,--}(t')$ to be the parallel transport of $\ell_0(t')$ from $-\frac{1}{2}$ to $f_{t',t}^\mu$ along $\mu(\gamma_{0,t'})$.
        \item Define $\ell_{0,--}(t)$ to be the parallel transport of $\ell_0(t)$ from $-\frac{1}{2}$ to $f_{t',t}^\mu$ along $\gamma_{0,t}$.
    \end{itemize}
    The differential $\partial\colon \mathit{CF}^\ast(\mu(L(t')),L(t))\rightarrow\mathit{CF}^{\ast+1}(\mu(L(t')),L(t))$ is defined by counting $J$-holomorphic sections of $\pi\colon W_k\rightarrow\mathbb{C}$ over the bigons bounded by the paths $\gamma_{0,t}$ and $\mu(\gamma_{0,t'})$. For example, for the generators over $e^\mu_{t',t}$, there are two bigons in the base with outputs at $c^\mu_{t',t}$ and $d^\mu_{t',t}$, respectively, see \cref{fig:wrapping_A,fig:wrapping_B}. These counts contribute to the two components of the map
    \[
    \mathit{CF}^\ast(\ell^\mu_{0,++}(t'),\ell_{0,++}(t))\longrightarrow\mathit{CF}^\ast(\ell_{0,+}^\mu(t'),\ell_{0,-}(t))\oplus\mathit{CF}^\ast(\ell_{0,-}^\mu(t'),\ell_{0,+}(t)),
    \]
    which is part of the differential of \eqref{eq:w1cpx}. There is a continuation map
    \[
    K_{L_0,L_0}^{t',t,\mu}\colon\mathit{CF}^\ast(L_0(t'),L_0(t))\longrightarrow\mathit{CF}^\ast(\mu(L_0(t')),L_0(t)),
    \]
    under which the summand $\mathit{CF}^\ast(\ell_0(t'),\ell_0(t))$ in the mapping cone description of $\mathit{CF}^\ast(L_0(t'),L_0(t))$, that is similar to \eqref{eq:cone}, is mapped to the direct sum
    \[
    \mathit{CF}^\ast(\ell^\mu_{0,++}(t'),\ell_{0,++}(t))[2]\oplus\mathit{CF}^\ast(\ell^\mu_{0,--}(t'),\ell_{0,--}(t))[2],
    \]
    in \eqref{eq:w1cpx}. Each component of this map is defined by counting $J$-holomorphic sections of $\pi\colon W_k\rightarrow\mathbb{C}$ over bigons with moving Lagrangian boundary conditions. Near the input, the Lagrangian boundary conditions are given by $L_0(t')$ and $L_0(t)$, while near the output, they are replaced with $\mu(L_0(t'))$ and $L_0(t)$. Since the total monodromy of $\pi$ is trivial, it follows that $w_{L_0}^{t'\rightarrow t}$ (which gives $w_{L_0}$ in the homotopy colimit as $t\rightarrow\infty$) coincides with the image of $\vartheta_0^{t'\rightarrow t}\in\mathit{CF}^0(\ell_0(t'),\ell_0(t))$, i.e.,
    \[
    K_{L_0,L_0}^{t',t,\mu}(\vartheta_0^{t'\rightarrow t})=w_{L_0}^{t'\rightarrow t}\in\mathit{CF}^{-2}(\mu(L_0(t')),L_0(t)).
    \]

    Similarly, one can define the Lagrangian $\mu(L_1)$ and the continuation element $w_{L_1}$ corresponding to $L_1$. The localization \eqref{eq:acc} forces all the morphisms $w_L \colon L\rightarrow\mu(L)[2]$ to be isomorphisms.
    
    \begin{theorem}\label{theorem:dga}
    The wrapped Fukaya $A_\infty$-algebra $\mathcal{W}_k$ is quasi-isomorphic to the $\mathbb{K}[x_1,x_2]$-linear path algebra of the quiver
    \begin{equation}\label{eq:quiver}
    \begin{tikzcd}
    \bullet \arrow[loop left,"\alpha"] \arrow[r,bend left,"a"] & \bullet \arrow[loop right,"\beta"] \arrow[l,bend left,"b"]
    \end{tikzcd}
    \end{equation} 
    with gradings
    \[
    |x_1|=|x_2|=|a|=|b|=0,\quad |\alpha|=|\beta|=-1,
    \]
    relations
    \[
    a b=(x_1-1)e_0,\quad b a=(x_1-1)e_1,\quad \alpha^2=\beta^2=0,
    \]
    and differentials
    \begin{equation}\label{eq:diff-ab}
    da=db=0,\quad d\alpha=(x_2+1)e_0,\quad d\beta=(x_1^k+x_2)e_1.
    \end{equation}
    \end{theorem}
    \begin{proof}
    We first explain the origin of the generator $\alpha$ and its differential. Consider the fiberwise Floer complex $\mathit{CF}^\ast(\ell_{0,+}^\mu(t'),\ell_{0,-}(t))$ in \eqref{eq:w1cpx}. For $t'-t$ large enough and generic, there is a square $Q_0^\mu(t'-t)\subset\mathbb{R}^2$ such that $\mathit{CF}^\ast(\ell_{0,+}^\mu(t'),\ell_{0,-}(t))$ is supported in degree $0$ and
    \[
    \mathit{CF}^0(\ell_{0,+}^\mu(t'),\ell_{0,-}(t))\cong\bigoplus_{q\in Q_0^\mu(t'-t)\cap(2\pi\mathbb{Z})^2}\mathbb{K}\cdot\varsigma_q^{t'\rightarrow t},
    \]
    where each $\varsigma_q$ corresponds to an intersection point in the fiber $\pi^{-1}(c_{t',t}^\mu)$. The chain map
    \[
    \mathfrak{s}_{\ell_0,t',t}^\mu\colon \mathit{CF}^0(\ell_{0,+}^\mu(t'),\ell_{0,-}(t))\longrightarrow\mathit{CF}^0(\ell_0(t'),\ell_0(t))[1],
    \]
    which is part of the Floer differential of the mapping cone description \eqref{eq:w1cpx} of $\mathit{CF}^\ast(\mu(L_0(t')),L_0(t))$, is given by counting $J$-holomorphic sections of $\pi\colon W_k\rightarrow\mathbb{C}$ over the region bounded by the arcs $\gamma_{0,t}$ and $\mu(\gamma_{0,t'})$, see \cref{fig:wrapping4}.
    \begin{figure}[!htb]
        \centering
        \includegraphics{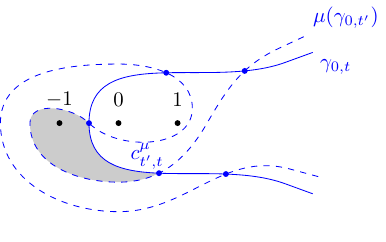}
        \caption{Projections of $L_0(t)$ and its once wrapping $\mu(L_0(t'))$ with a shaded bigon contributing to a Floer differential.}
        \label{fig:wrapping4}
    \end{figure}
    Since the region contains the critical value $-1$, the same argument as in the proof of \cref{proposition:diff-b,proposition:g} shows that after passing to the homotopy colimit as $t\rightarrow\infty$, the map $\mathfrak{s}_{\ell_0,t',t}^\mu$ is the multiplication by the polynomial $x_2+1$. Thus, denoting by $\alpha$ the generator $\varsigma_0^{t'\rightarrow t}$ in the homotopy colimit as $t\rightarrow\infty$, we get $d\alpha=(x_2+1)e_0$.
    
    By abuse of notation, we still use $a$, $b$ and $\beta$ to denote the images of the same generators under the continuation maps
    \begin{equation}\label{eq:K}
    K_{L_i,L_j}^{\mu,\mu}\colon\mathit{CW}^\ast_\pi(L_i,L_j)\longrightarrow\mathit{CW}_\pi^\ast(\mu(L_i),\mu(L_j))
    \end{equation}
    for $i,j\in \{0,1\}$. The next step of the proof is to identify the continuation elements $w_{L_0}$ and $w_{L_1}$ with
    \[
    a\beta b\alpha+\alpha a\beta b\in\mathit{CW}_\pi^{-2}(\mu(L_0),L_0)
    \]
    and
    \[
    b\alpha a\beta+\beta b\alpha a\in\mathit{CW}_\pi^{-2}(\mu(L_1),L_1),
    \]
    respectively. Here, we identify $\alpha a\beta b$ with the part of $w_{L_0}$ in the fiber $\pi^{-1}(f^\mu)$. The cases of other products are similar. First, by \cref{proposition:relation}, the product $a\beta b$ already makes sense in the Floer complex $\mathit{CF}^\ast(L_0(t'''),L_0(t))$ for $t'''-t$ generic and sufficiently large. In terms of the fiberwise generators, we have
    \[
    \mu_0^{t'\rightarrow t}\cdot\zeta_0^{t''\rightarrow t'}\cdot\kappa_0^{t'''\rightarrow t''}=\eta_0^{t'''\rightarrow t}.
    \]
    To further multiply with the generator $\alpha$ identified above, we need to consider the once wrapping $\mu(L_0(t''''))$ of $L_0(t'''')$ and replace $L_0(t''')$ with $\mu(L_0(t'''))$. We form the product
    \[
    \mathit{CF}^\ast(\mu(L_0(t''')),L_0(t))\otimes\mathit{CF}^\ast(\mu(L_0(t'''')),\mu(L_0(t''')))\longrightarrow\mathit{CF}^\ast(\mu(L_0(t'''')),L_0(t))
    \]
    for $t''''-t'''$ large enough and generic, see \cref{fig:3p}, where the projections of these Lagrangian submanifolds under $\pi$ are depicted.
    \begin{figure}[!htb]
    	\centering
        \includegraphics{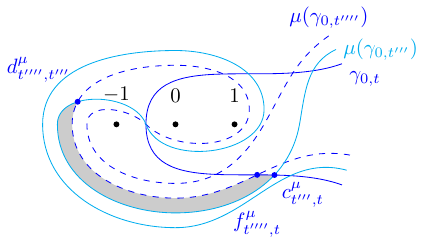}
        \caption{The arcs $\gamma_{0,t}$, $\mu(\gamma_{0,t'''})$ and $\mu(\gamma_{0,t''''})$ with a shaded triangle contributing to a Floer product.}\label{fig:3p}
    \end{figure}
    Note that the generators over $d_{t'''',t'''}^\mu$ and $c_{t''',t}^\mu$ have degree $-1$, by \cref{proposition:relation} and our argument above, they correspond respectively to the images of (Laurent polynomial multiplications of) the generators $a\beta b$ under the continuation map \eqref{eq:K} and $\alpha$, while the generators over $f_{t'''',t}^\mu$ have degree $-2$. By counting $J$-holomorphic sections of $\pi\colon W_k\rightarrow\mathbb{C}$ over the triangle formed by $\gamma_{0,t}$, $\mu(\gamma_{0,t'''})$ and $\mu(\gamma_{0,t''''})$ with vertices $d_{t'''',t'''}^\mu$, $c_{t''',t}^\mu$ and $f_{t'''',t}^\mu$ (see \cref{fig:3p}), we see that the product $\mu_0^{t'\rightarrow t}\cdot\zeta_0^{t''\rightarrow t'}\cdot\kappa_0^{t'''\rightarrow t''}\cdot\varsigma_0^{t''''\rightarrow t'''}$ survives in
    \[
    \mathit{CF}^0(\ell_{0,--}^\mu(t''''),\ell_{0,--}(t))[2]\subset\mathit{CF}^{-2}(\mu(L_0(t'''')),L_0(t)).
    \]
    Recall that the part of $w_{L_0}^{t''''\rightarrow t}$ in the Floer complex $\mathit{CF}^\ast(\ell_{0,--}^\mu(t''''),\ell_{0,--}(t))$ is by definition the continuation image of $\vartheta_0^{t''''\rightarrow t}$ under the map
    \[
    \mathit{CF}^\ast(\ell_0(t''''),\ell_0(t))\longrightarrow\mathit{CF}^\ast(\ell_{0,--}^\mu(t''''),\ell_{0,--}(t))
    \]
    induced by the continuation map $K_{L_0,L_0}^{t'''',t,\mu}$. Since the shaded triangle in \cref{fig:3p} does not contain any critical value of $\pi$, the shrinking argument used previously implies that
    \[
    w_{L_0}^{t''''\rightarrow t}|_{\mathit{CF}^\ast(\ell_{0,--}^\mu(t''''),\ell_{0,--}(t))}=\mu_0^{t'\rightarrow t}\cdot\zeta_0^{t''\rightarrow t'}\cdot\kappa_0^{t'''\rightarrow t''}\cdot\varsigma_0^{t''''\rightarrow t'''}.
    \]
    Passing to the homotopy colimits as $t,t',t'',t'''\rightarrow\infty$, we have verified that $\alpha a\beta b$ coincides with the part of $w_{L_0}$ above the fiber $\pi^{-1}(f^\mu)$ in the fiberwise wrapped Fukaya category $\mathcal{W}(W_k,\pi)$.
    
    Finally, we consider the localization of $\mathcal{W}_k(W_k,\pi)$ along the natural transformation $\mathit{id}\Rightarrow\mu$. In the localization, multiplication by the continuation elements $w_{L_i}$ iteratively induces isomorphisms
    \[
    L_i\xrightarrow{w_{L_i}}\mu(L_i)\xrightarrow{w_{L_i}}\mu^2(L_i)\xrightarrow{w_{L_i}}\cdots
    \]
    in $\mathcal W(W_k;\mathbb{K})$ for $i \in \{0,1\}$. In particular, there are quasi-isomorphisms
    \begin{equation}\label{eq:CW}
    \mathit{CW}^\ast(L_i,L)\cong\hocolim_{n\rightarrow\infty}\mathit{CW}^\ast_\pi(\mu^n(L_i),L)
    \end{equation}
    for any object $L$ of $\mathcal{W}(W_k,\pi)$. Since the continuation elements are given by $a\beta b\alpha+\alpha a\beta b$ and $b\alpha a\beta+\beta b\alpha a$ for $L_0$ and $L_1$, respectively, this allows us to compute the endomorphism $A_\infty$-algebra $\mathcal{W}_k$ of the objects $L_0$ and $L_1$ explicitly using \eqref{eq:CW}. Since we have proved in \cref{lemma:dg} that $\mathcal{W}_{k,\pi}$ is a dg algebra, its localization $\mathcal{W}_k$ is also a dg algebra, which is generated by $x_1^{\pm1},x_2^{\pm1},a,b,\alpha$ and $\beta$ with the relations and differentials as stated. It follows from \eqref{eq:diff-ab} that the negative powers $x_1^{-1}$ and $x_2^{-1}$ do not contribute to the cohomology, so throwing them away does not affect $\mathcal{W}_k$ up to quasi-isomorphism.
    \end{proof}
    
    The following is a straightforward consequence of \cref{theorem:dga}.
    
    \begin{corollary}\label{corollary:H0}
    After taking a quotient identifying all the idempotents, the $\mathbb{K}$-algebra $H^0(\mathcal{W}_k)$ is generated by $a$ and $b$ with the relations $a^2=b^2=0$, $ab=ba$ and $(ab+1)^k=1$. In particular, $H^0(\mathcal{W}_k)$ is finite-dimensional and $H^0(\mathcal{W}_k)^\times\cong\mathbb{Z}_k$.
    \end{corollary}
    \begin{proof}
    To see the finite-dimensionality of $H^0(\mathcal{W}_k)$, note that the only non-zero generators are $(ab)^n$ for $n\geq 0$. On the other hand, the relation $(ab+1)^k=1$ shows that $(ab)^k$ can be expressed as linear combinations over $\mathbb{K}$ of $1,\ldots,(ab)^{k-1}$.
    \end{proof}
    
    \section{Derived contraction algebras and Ginzburg dg algebras}\label{section:contraction_and_ginz}
    In \cref{section:contraction}, we relate the wrapped Fukaya $A_\infty$-algebra of the double bubble plumbings $W_k$ to the derived contraction algebra associated to the crepant partial resolution of a compound $A_2$ singularity. This allows us to conclude homological mirror symmetry for the double bubble plumbings $W_k$. In \cref{section:ginzburg}, we relate the wrapped Fukaya $A_\infty$-algebra of $W_k$ to the Ginzburg dg algebra of the $2$-cycle quiver with a certain potential.
    
    \subsection{Relative wrapped Fukaya category and mirror symmetry}\label{section:contraction}
    
    Let $\mathbb{K}$ be any field. Our description of the wrapped Fukaya $A_\infty$-algebra $\mathcal{W}_k$ in \cref{theorem:dga} is closely related to the work of Evans--Lekili \cite{evle}, which interprets the derived contraction algebras associated to crepant partial resolutions of compound $A_n$ singularities as relative Fukaya categories of disks with marked points. In particular, one can interpret the derived contraction algebras $\mathcal{A}_k$ associated to the crepant (partial) resolutions $Y_k\rightarrow\mathit{Spec}(R_k)$ of the singularities
    \begin{equation}\label{eq:ca2_sing}
        R_k=\frac{\mathbb{K}[u,v,x,y]}{\left(uv-xy\left((x+1)^k+y-1\right)\right)}, \quad k\geq1,
    \end{equation}
    as the wrapped Fukaya categories of a disk relative to three interior points. More precisely, for $k\geq1$, consider the quiver 
    \[
    	\begin{tikzcd}
    	\bullet \arrow[loop left,"\alpha"] \arrow[r,bend left,"a"] & \bullet \arrow[loop right,"\beta"] \arrow[l,bend left,"b"]
    	\end{tikzcd}
    \]
    in the statement of \cref{theorem:dga}. By \cite[Theorem 6.3]{evle} $\mathcal{A}_k$ is quasi isomorphic to the $\mathbb{K}[x,y]$-linear path algebra associated to this quiver, with gradings $|a|=|b|=0$, $|\alpha|=|\beta|=-1$, and relations
    \[
    ab=xe_0,\quad ba=xe_1,\quad \alpha^2=\beta^2=0.
    \]
    The differential is specified on the generators by
    \[
    da=db=0,\quad d\alpha=ye_0,\quad d\beta=\left((x+1)^k+y-1\right)e_1
    \]
    and extended to the whole path algebra by the Leibniz rule. Here we have abused the notation intentionally to emphasize the relation with the dg algebra model of $\mathcal{W}_k$ in \cref{theorem:dga}.
    
    \begin{figure}[!htb]
    	\centering
    	\includegraphics{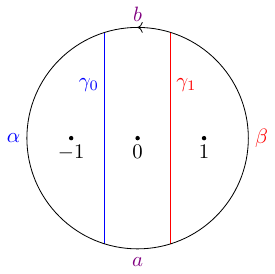}
    	\caption{The relative wrapped Fukaya category of $(\mathbb{C},\{-1,0,1\})$.}\label{fig:base1}
    \end{figure}
    
    The above model of the derived contraction algebra $\mathcal{A}_k$ computes the endomorphism $A_\infty$-algebra of the objects $\gamma_0,\gamma_1\subset\mathbb{C}\setminus\{-1,0,1\}$ in a (base changed) relative wrapped Fukaya category, see \cref{fig:base1} that we now describe. The generators $a,b,\alpha$ and $\beta$ in the dg algebra above correspond to Reeb chords on the boundary of the disk from $\gamma_0$ to $\gamma_1$, from $\gamma_1$ to $\gamma_0$, from $\gamma_1$ to itself and from $\gamma_0$ to itself, respectively. Note that the wrapped Fukaya category of $\mathbb{C}$ relative to the three points $\{-1,0,1\}$, which we denote by $\mathcal{W}(\mathbb{C},\{-1,0,1\})$, is an $A_\infty$-category linear over the polynomial ring $\mathbb{K}[t_0,t_1,t_2]$, where the variables $t_0,t_1,t_2$ correspond to the points $-1,0,1$ which partially compactifies $\mathbb{C}\setminus\{-1,0,1\}$ to $\mathbb{C}$. Now consider the base change
    \begin{equation}\label{eq:W1}
    \mathcal{W}(\mathbb{C},\{-1,0,1\})\otimes_{\mathbb{K}[t_0,t_1,t_2]}\mathbb{K}[x,y],
    \end{equation}
    where $\mathbb{K}[x,y]$ is regarded as a $\mathbb{K}[t_0,t_1,t_2]$-algebra via the homomorphism given by
    \[
    t_i\mapsto f_i(x,y)\textrm{ for }i \in \{0,1,2\}
    \]
    where in this case
    \[
    f_0(x,y)=y,\quad f_1(x,y)=x,\quad f_2(x,y)=(x+1)^k+y-1.
    \]
    Now the dg algebra $\mathcal{A}_k$ is quasi-isomorphic to the endomorphism $A_\infty$-algebra of the objects $\gamma_0$ and $\gamma_1$ of the Fukaya category \eqref{eq:W1}.
    
    In the above, the smoothness of the curves $\{f_i(x,y)=0\}\subset\mathbb{K}^2$ for $i\in \{0,1,2\}$ ensures that the compound $A_2$ singularity of $\mathit{Spec}(R_k)$ at the origin admits a small resolution, and the order of the polynomials $f_0,f_1,f_2$ determines the specific resolution. As in \cite[Lemma 3.1]{sw}, one can show that the exceptional locus of $Y_k\rightarrow\mathit{Spec}(R_k)$ consists of two floppable $(-1,-1)$ curves meeting at one point.
    
    Our dg algebra model of the wrapped Fukaya $A_\infty$-algebra $\mathcal{W}_k$ obtained in \cref{theorem:dga} is related to the above model of $\mathcal{A}_k$ by the change of variables
    \[
    x_1=x+1,\quad x_2=y-1,
    \]
    so the polynomials $f_0,f_1$ and $f_2$ in the base change are now replaced with
    \[
    g_0(x_1,x_2)=x_2+1,\quad g_1(x_1,x_2)=x_1-1,\quad g_2(x_1,x_2)=x_1^k+x_2.
    \]
    We have therefore proved the following lemma.
    
    \begin{lemma}\label{lemma:wfuk_contraction}
    There is a quasi-isomorphism of $A_\infty$-algebras $\mathcal{A}_k\cong\mathcal{W}_k$ over $\Bbbk$.
    \qed
    \end{lemma}
    
    The existence of a $1$-dimensional model of the wrapped Fukaya category $\mathcal{W}(W_k;\mathbb{K})$ is compatible with a conjecture by Lekili--Segal \cite[Conjecture E]{lese} relating the wrapped Fukaya category of the total space of an algebraic torus fibration with the wrapped Fukaya category of its base relative to the discriminant locus. In our case, the algebraic torus fibration is given by the Morse--Bott--Lefschetz fibration $\pi\colon W_k\rightarrow\mathbb{C}$, with discriminant locus $\{-1,0,1\}$. Note that \cref{lemma:wfuk_contraction} gives a fully faithful embedding
    \[
    \mathcal{W}(W_k;\mathbb{K})\hookrightarrow\mathcal{W}(\mathbb{C},\{-1,0,1\})\otimes_{\mathbb{K}[t_0,t_1,t_2]}\mathbb{K}[x,y].
    \]
    
    One consequence of realizing $\mathcal{W}(W_k;\mathbb{K})$ as a subcategory of the relative Fukaya category \eqref{eq:W1} is that the continuation elements $w_{L_0}$ and $w_{L_1}$ (or their linear combination $w_{L_0}e_0+w_{L_1}e_1$) used to localize $\mathcal{W}(W_k,\pi)$ in \cref{subsection:wrapping} can be geometrically interpreted as a simple Reeb orbit in the boundary of a disk $D\subset\mathbb{C}$ with radius $>1$, see \cref{fig:base1}. With respect to a suitably chosen Hamiltonian, this simple Reeb orbit corresponds to a class $w\in\mathit{SH}^{-2}(W_k;\mathbb{K})$, which endows $\mathcal{W}(W_k;\mathbb{K})$ with the structure of an $A_\infty$-category over $\mathbb{K}[w]$, see \cite[Remark 2.5]{evle} and \cite[Section 6.2]{hk}.
    
    We end this subsection by proving the following theorem, which establishes mirror symmetry between the Weinstein manifold $W_k$ and the $3$-dimensional Calabi--Yau variety $Y_k$.
    
    \begin{theorem}\label{theorem:MS}
        Let $\mathbb{K}$ be any field and $k\geq1$. There is an equivalence of categories
        \[
        D^\mathit{perf}\mathcal{W}(W_k;\mathbb{K})\cong D^b\mathit{Coh}(Y_k)/\langle\mathcal{O}_{Y_k}\rangle.
        \]
    \end{theorem}
    \begin{proof}
    By \cref{lemma:wfuk_contraction} it is enough to establish the equivalence $D^\mathit{perf}(\mathcal{A}_k)\cong D^b\mathit{Coh}(Y_k)/\langle\mathcal{O}_{Y_k}\rangle$. Recall that there is a tilting bundle $\mathcal{P}\rightarrow Y_k$ constructed by Van den Bergh \cite{mv3}, with respect to which $\mathcal{A}_k$ is given by the Drinfeld localization of the endomorphism algebra $\mathit{End}(\mathcal{P})$ with respect to the idempotent $\mathit{id}_{\mathcal{O}_Y}$. Thus it is enough to show that $D^b\mathit{Coh}(Y_k)\cong D^\mathit{perf}\left(\mathit{End}(\mathcal{P})\right)$, which holds as long as $Y_k$ is a normal algebraic variety by \cite[Proposition 3.2.10]{mv3}. Since the singularities of $\mathit{Spec}(R_k)$ are isolated hypersurface singularities, it follows from Serre's criterion that $Y_k$ is normal over any field $\mathbb{K}$.
    \end{proof}
    
    \begin{remark}
    Note that when $k\geq1$ is even, $(u,v,x,y)=(0,0,-2,0)$ is a singularity of the affine variety $\mathit{Spec}(R_k)$, therefore $Y_k$ is singular. On the other hand, when $k$ is odd, $Y_k$ can be either singular or smooth. The mirrors of $W_k$ considered by Smith--Wemyss in \cite{sw}, i.e., crepant resolutions of the singularities \eqref{eq:sing1}, are always smooth. One should think of the Smith--Wemyss mirrors of $W_k$ as approximations of the (singular) varieties $Y_k$ considered here.
    \end{remark}
    
    \subsection{Relation to a Ginzburg dg algebra}\label{section:ginzburg}
    
    Let $Q$ be a quiver and let $W \in \mathbb{K}Q/[\mathbb{K}Q,\mathbb{K}Q]$ be an element called a \textit{potential}. Let $\overline Q$ denote the graded quiver (sometimes called the enhanced quiver) with vertex set $Q_0$ and arrow set $Q_1$ consisting of the following
    \begin{itemize}
        \item an arrow $g\colon v \to w$ in degree $0$ for each $g\colon v\to w$ in $Q$,
        \item an arrow $g^\ast \colon w\to v$ in degree $-1$ for each $g\colon v\to w$ in $Q$,
        \item an arrow $h_v\colon v \to v$ in degree $-2$ for each $v \in Q_0$.
    \end{itemize}
    Given an arrow $g$ in $Q$, we define the cyclic differential of $W$ with respect to $g$ as
    \[
    \partial^\circ_gW \coloneqq  \sum_{W = pg q}qp.
    \]
    The sum is taken over all ways of writing the potential $W$ as $W = pgq$, e.g., if $W = ghgh'$, then $\partial^\circ_gW = hgh' + h'gh$.
    \begin{definition}[Ginzburg dg algebra]
        Let $(Q,W)$ be a quiver with potential $W$. The \emph{Ginzburg dg algebra} $\mathcal G(Q,W)$ is defined as the dg algebra whose underlying vector space is the path algebra $\mathbb{K}\overline Q$, and with differential defined on generators by
        \[
        dg = 0, \quad dg^\ast = \partial^\circ_gW, \quad dh_v = \sum_{\substack{\text{arrows $g$} \\ \text{starting at $v$}}} gg^\ast - \sum_{\substack{\text{arrows $g$} \\ \text{ending at $v$}}} g^\ast g,
        \]
        and extended to the entire path algebra by Leibniz rule and linearity.
    \end{definition}
    
    \begin{theorem}\label{theorem:ginzburg}
    Let $\mathbb{K}$ be any field and $k\geq1$. The Fukaya $A_\infty$-algebra $\mathcal{W}_k$ over $\mathbb{K}$ is quasi-isomorphic to the Ginzburg dg algebra $\mathcal{G}_k$ associated to the $2$-cycle quiver
    \[
    \begin{tikzcd}
    0\arrow[r,bend left,"e"] & 1 \arrow[l,bend left,"f"] 
    \end{tikzcd}
    \]
    with potential
    \[
    w_k=efe\left(1+(fe+1)+\cdots+(fe+1)^{k-1}\right)f.
    \]
    \end{theorem}
    \begin{proof}
    First we give an explicit description of $\mathcal G_k$ in our situation. We have generators $e$ and $f$ in degree $0$, generators $e^\ast$ and $f^\ast$ in degree $-1$ and generators $h_1$ and $h_2$ in degree $-2$. The differentials of the generators are given by
        \begin{align*}
            de &= df = 0 \\
            de^\ast &= fe(1+(fe+1)+\cdots+(fe+1)^{k-1})f \\
            df^\ast &= efe(1+(fe+1)+\cdots+(fe+1)^{k-1}) \\
            dh_0 &= ee^\ast - f^\ast f \\
            dh_1 &= ff^\ast - e^\ast e.
        \end{align*}
    
    By \cref{lemma:wfuk_contraction}, it suffices to identify the derived contraction algebra $\mathcal{A}_k$ with the Ginzburg dg algebra $\mathcal G_k$. To do this, define a map $\phi\colon \mathcal{G}_k\rightarrow\mathcal{A}_k$ by
    \[
    \begin{split}
    &e\mapsto a,\quad f\mapsto b,\quad e^\ast\mapsto \beta b-b\alpha,\quad f^\ast\mapsto a\beta-\alpha a,\\
    &h_1\mapsto 0,\quad h_2\mapsto 0.
    \end{split}
    \]
    It is straightforward to check that $\phi$ is a map of dg algebras. To show that it is a quasi-isomorphism, first notice that as a $\mathbb{K}$-algebra, $H^0(\mathcal{G}_k)$ is generated by the elements $u_0,u_1,e,f$, where $u_0$ and $u_1$ are two idempotents, together with relations
    \[
    \left((fe+1)^k-1\right)u_1f=0,\quad \left((ef+1)^k-1\right)u_0e=0
    \]
    coming from $de^\ast$ and $df^\ast$, where $u_0$ and $u_1$ are idempotents. Note that one can use the identity
    \[
    efe\left(1+(fe+1)+\cdots+(fe+1)^{k-1}\right)=ef\left(1+(ef+1)+\cdots+(ef+1)^{k-1}\right)e
    \]
    to get the second relation above from $de^\ast$. On the other hand, as a $\mathbb{K}$-algebra $H^0(\mathcal{A}_k)$ is generated by $e_0,e_1,a,b$ together with the relations
    \begin{align*}
        \left((ab+1)^k-1\right)e_0a&=\left((x+1)^k-1\right)e_0a=\left((x+1)^k-1\right)ae_1=-yae_1=-ye_0a=0, \\
        \left((ba+1)^k-1\right)e_1b&=\left((x+1)^k-1\right)e_1b=-ye_1b=-bye_0=0.
    \end{align*}
    It is clear that the map $\phi$ matches the generators and is compatible with the relations, therefore induces an isomorphism of $\mathbb{K}$-algebras $H^0(\mathcal{G}_k)\cong H^0(\mathcal{A}_k)$. 
    
    Fix $j>0$. Suppose $\phi$ induces an isomorphism $H^i(\mathcal{G}_k)\cong H^i(\mathcal{A}_k)$ between finite-dimensional vector spaces for all $-j\leq i\leq 0$. Since $\mathcal{G}_k$ is homologically smooth and supported in non-positive degrees, it follows from \cite[Proposition 2.5]{ky} that there exists a perfect module $\mathcal{M}$ over $\mathcal{G}_k$ built by taking direct sums and direct summands of the diagonal bimodule such that $H^{-j-1}(\mathcal{G}_k)\cong H^0(\mathcal{M})$. Since the morphism $\phi$ induces a functor $\Phi\colon D^\mathit{perf}(\mathcal{G}_k)\rightarrow D^\mathit{perf}(\mathcal{A}_k)$ between derived categories of perfect modules and is compatible with taking direct sums and direct summands, there is a commutative diagram
    \[
    \begin{tikzcd}
    	H^{-j-1}(\mathcal{G}_k) \arrow[d,"\cong"'] \arrow[r,"\phi"] &H^{-j-1}(\mathcal{A}_k) \arrow[d,"\cong"] \\
    	H^0(\mathcal{M}) \arrow[r,"H^0(\Phi)"] &H^0\left(\Phi(\mathcal{M})\right)
    \end{tikzcd}
    \]
    In the above, the map $H^0(\Phi)$ is an isomorphism since we have proved that $\phi$ induces an isomorphism on $H^0$ and since $\mathcal{M}$ is obtained by taking direct sums and direct summands of the diagonal bimodule $\mathcal{G}_k$. By commutativity, $\phi$ also induces an isomorphism on $H^{-j-1}$. The result now follows from induction.
    \end{proof}
    
    \begin{remark}
    It is known that the cohomology $H^\ast(\mathcal{A}_k)$ is $2$-periodic for $\ast\leq0$, see \cite[Section 5.2]{mb}. Since $H^{-1}(\mathcal{A}_k)=0$, we actually have $H^\ast(\mathcal{A}_k)\cong H^0(\mathcal{A}_k)[w]$ for some periodicity element $w\in H^{-2}(\mathcal{A}_k)$. From the proof of \cref{theorem:dga}, it is not hard to see that $w$ corresponds to the continuation element $w_{L_0}e_0+w_{L_1}e_1$ in the wrapped Fukaya $A_\infty$-algebra $\mathcal{W}_k$.
    \end{remark}
    
    \begin{remark}
        The Weinstein manifold $W_k$ admits a description as a Weinstein connected sum of the Weinstein pairs $(T^\ast S^3,\Lambda_U)$ and $(T^\ast S^1 \times B^4,S^1 \times U)$, where $U \hookrightarrow S^3$ is the unknot, and $\Lambda_U$ denotes the unit conormal bundle, where the gluing map in the connected sum depends on $k \geq 0$. Equivalently $W_k$ is the result of attaching a ``generalized Weinstein handle'' $T^\ast(S^1 \times D^2)$ (where $D^2$ is the closed disk) to $T^\ast S^3$ along the unit conormal bundle $\Lambda_U \subset ST^\ast S^3$ in such a way that the gluing map realizes the identifications $\lambda \mapsto \mu$ and $\mu \mapsto \lambda + k\mu$ where $\lambda$ and $\mu$ are choices of longitudes and meridians of the boundaries of the two core solid tori in the gluing. This description allows us to compute the Chekanov--Eliashberg dg algebra of the Legendrian attaching spheres of $W_k$ using the gluing techniques developed in \cite{ae,ad}. Such a computation would offer an alternative proof of \cref{theorem:ginzburg} via the Bourgeois--Ekholm--Eliashberg surgery formula \cite{bee}. This is compatible with the result \cite[Corollary 4.16]{ad} for $k = 0$, since there is a Weinstein homotopy taking $W_0$ to the plumbing of two copies of $T^\ast S^3$ at two points. Such a Weinstein homotopy can be constructed by perturbing the Weinstein data by a smooth function $f \colon W_0 \to \mathbb{R}$ that is constant outside of a neighborhood of $Q_0 \cap Q_1$, and such that $f|_{Q_0 \cap Q_1}$ is a small Morse function with two critical points.
    \end{remark}
    
    \section{Unknottedness of clean Lagrangian intersections}\label{section:proof}
    
    Let us now state our main theorem.
    
    \begin{theorem}\label{theorem:main}
    Let $k\in\mathbb{Z}_{\geq0}$ be any non-negative integer. There is no Hamiltonian isotopy of the core spheres $Q_0,Q_1\subset W_k$ to a pair of exact Lagrangian spheres meeting cleanly in a circle that is knotted in either component.
    \end{theorem}
    
    This section is devoted to the proof of \cref{theorem:main}. In \cref{section:dilations} we recall the definition of the notion of a cyclic quasi-dilation introduced by the second author in \cite{yle}. In \cref{section:classification} we prove a classification result of geometrically spherical exact Lagrangian submanifolds in the double bubble plumbings $W_k$. Finally, the proof of \cref{theorem:main} appears in \cref{section:proof_main_result}.
    
    \subsection{Cyclic dilations}\label{section:dilations} 
    
    Let $\mathit{SH}_{S^1}^\ast(M;\mathbb{K})$ be the $S^1$-equivariant symplectic cohomology, which is the cohomology of the complex
    \[
    \left(\mathit{SC}^\ast(M;\mathbb{K})\otimes_\mathbb{K}\mathbb{K}(\!(u)\!)/u\mathbb{K}[\![u]\!],\partial+u\delta_1+u^2\delta_2+\cdots\right),
    \]
    where $\partial$ is the usual Floer differential on the complex $\mathit{SC}^\ast(M;\mathbb{K})$ defining the (non-equivariant) symplectic cohomology, $\delta_1\colon \mathit{SC}^\ast(M;\mathbb{K})\rightarrow\mathit{SC}^{\ast-1}(M;\mathbb{K})$ is the cochain level BV operator, and the operations $\delta_j\colon \mathit{SC}^\ast(M;\mathbb{K})\rightarrow\mathit{SC}^{\ast+1-2j}(M;\mathbb{K})$ satisfy the structure equations $\sum_{j=0}^d\delta_j\delta_{d-j}=0$ of an $S^1$-complex for each $d\geq0$. 
    
    Let $M$ be a Liouville manifold with $c_1(M)=0$, and let $\mathbb{K}$ be any field. There is a marking map
    \begin{equation}\label{eq:mark}
    \mathbf{B}\colon\mathit{SH}_{S^1}^\ast(M;\mathbb{K})\longrightarrow\mathit{SH}^{\ast-1}(M;\mathbb{K})
    \end{equation}
    relating the symplectic cohomology and its $S^1$-equivariant version. We refer the reader to \cite[Proposition 2.9]{bo} for the detailed (especially the chain level) definition of $\mathbf{B}$.
    
    \begin{definition}\label{definition:cyclic}
    A \emph{cyclic quasi-dilation} is a pair $(\tilde{b},h)\in\mathit{SH}_{S^1}^1(M;\mathbb{K})\times\mathit{SH}^0(M;\mathbb{K})^\times$ such that the image of $\tilde{b}$ under the marking map \eqref{eq:mark} is the invertible element $h\in\mathit{SH}^0(M;\mathbb{K})^\times$. A \emph{cyclic dilation} is a cyclic quasi-dilation with $h = 1$.
    \end{definition}
    
    \begin{remark}
    In \cite{yle}, the class $\tilde{b}$ in \cref{definition:cyclic} is called a ``cyclic dilation" instead of a ``cyclic quasi-dilation." Here, we instead follow the convention of \cite{yla} and reserve the terminology ``cyclic dilation" for the situation $h=1$. Closely related notions have also been introduced by Zhou \cite{zz} in the study of symplectic fillings of asymptotically dynamically convex contact manifolds.
    \end{remark}
    
    Since the BV operator $\Delta\colon\mathit{SH}^\ast(M;\mathbb{K})\rightarrow\mathit{SH}^{\ast-1}(M;\mathbb{K})$ is the composition of the erasing map
    \[
    \mathbf{I}\colon\mathit{SH}^\ast(M;\mathbb{K})\longrightarrow\mathit{SH}^\ast_{S^1}(M;\mathbb{K})
    \]
    and the marking map $\mathbf{B}$, it is clear from the definition that any Liouville manifold $M$ admitting a quasi-dilation also admits a cyclic quasi-dilation. However, the converse is not true \cite{yle,zz}.
    
    Let $Q$ be a closed oriented manifold, we say that $Q$ admits a cyclic quasi-dilation if its cotangent bundle $T^\ast Q$ admits a cyclic quasi-dilation. The existence of a cyclic quasi-dilation imposes restrictions on the topology of $Q$. As a first example, we have the following result that generalizes \cite[Theorem 5.14]{gp}. 
    
    We say a closed, oriented manifold $Q$ is \textit{dominated} by another closed, oriented manifold $Q'$ if there is a map $Q'\rightarrow Q$ of non-zero degree.
    
    \begin{proposition}\label{proposition:cover}
    Let $Q$ be a closed, oriented $3$-manifold which admits a cyclic quasi-dilation over a field $\mathbb{K}$ of characteristic $0$. Then one of the following two statements holds.
    \begin{itemize}
    	\item[(i)] $Q$ is finitely covered by $S^1\times\Sigma_g$ for some closed oriented surface of genus $g\geq1$.
    	\item[(ii)] $Q$ is finitely covered by a connected sum $\#_nS^1\times S^2$ for some $n\geq0$ (as a convention, $\#_0S^1\times S^2=S^3$).
    \end{itemize}
    \end{proposition}
    \begin{proof}
    By definition, the cyclic quasi-dilation $(\tilde{b},h)$ on $Q$ gives a class $\tilde{b}\in H_2^{S^1}(\mathcal{L}Q;\mathbb{K})$ whose image under the composition
    \[
    H_2^{S^1}(\mathcal{L}Q;\mathbb{K})\xrightarrow{\mathbf{B}_\mathcal{L}} H_3(\mathcal{L}Q;\mathbb{K})\xrightarrow{\mathit{ev}_\ast}H^0(Q;\mathbb{K})\cong\mathbb{K}
    \]
    lies in $\mathbb{K}^\times$, where $\mathbf{B}_\mathcal{L}$ is the marking map in string topology introduced by Chas--Sullivan \cite[Section 6]{cs} (cf.\@ \cite[Proposition 2.9]{bo}) and $\mathit{ev} \colon \mathcal LQ \to Q$ is the map $\gamma \mapsto \gamma(1)$. To see this, simply note that since the image of $\tilde{b}$ under $\mathbf{B}_\mathcal{L}$ defines an invertible element $h\in H_3(\mathcal{L}Q;\mathbb{K})$ (with respect to the loop product), so after composing with $\mathit{ev}_\ast$, which repects the product structures, we get a non-zero element in $\mathbb{K}$.
    
    Without loss of generality, we may assume that $\tilde{b}\in H_2^{S^1}(\mathcal{L}_cQ;\mathbb{K})$ for some connected component $\mathcal{L}_cQ\subset\mathcal{L}Q$. Since $\tilde{b}$ can be equivalently regarded as a class in $H_2(S^\infty\times_{S^1}\mathcal{L}_cQ;\mathbb{K})$, we can represent it by a map $\sigma\colon\Sigma\rightarrow S^\infty\times_{S^1}\mathcal{L}_cQ$, where $\Sigma$ is an orientable, connected surface. Since $S^\infty\times_{S^1}\mathcal{L}_cQ$ is an orbit space, we obtain a map $\tilde{\sigma}\colon S^1\times\Sigma\rightarrow S^\infty\times\mathcal{L}_cQ$ representing the lift of $\tilde{b}$ in $H_3(S^\infty\times\mathcal{L}_cQ;\mathbb{K})$. Consider the composition
    \[
    \mathit{ev}\circ p\circ\tilde{\sigma}\colon S^1\times\Sigma\longrightarrow Q,
    \]
    where $p\colon S^\infty\times\mathcal{L}_cQ\rightarrow\mathcal{L}_cQ$ is the natural projection. The map $\mathit{ev}\circ p\circ\tilde{\sigma}$ has non-zero degree by the assumption that $\tilde{b}$ is a cyclic quasi-dilation, it follows that $Q$ is dominated by $S^1\times\Sigma$. The result then follows from \cite[Theorems 1 and 3]{kn}.
    \end{proof}
    
    The advantage of the notion of a cyclic quasi-dilation is that it has an algebraic counterpart, using which one can establish the existence of cyclic quasi-dilations in various geometric situations, see e.g.\@ \cite[Section 6.1]{yle}. To explain this, recall that a \textit{(weak) smooth Calabi--Yau structure} (of dimension $n$) on an $A_\infty$-category $\mathcal{A}$ is a Hochschild cocycle $\zeta\in\mathit{CC}_{-n}(\mathcal{A})$\footnote{We are using cohomological grading here.} that induces a quasi-isomorphism
    \[
    \mathcal{A}^![-n]\cong\mathcal{A}
    \]
    between the diagonal bimodule $\mathcal{A}$ and a shift of its inverse dualizing bimodule
    \[
    \mathcal{A}^!\coloneqq \Rhom_{\mathcal{A}\otimes\mathcal{A}^\mathit{op}}(\mathcal{A},\mathcal{A}\otimes\mathcal{A}^\mathit{op}).
    \]
    An \textit{exact Calabi--Yau structure} on $\mathcal{A}$ is a lift of the homology class $[\zeta]\in\mathit{HH}_{-n}(\mathcal{A})$ to a class in cyclic homology $\mathit{HC}_{-n+1}(\mathcal{A})$ through Connes' map $\mathit{HC}_{-n+1}(\mathcal{A})\rightarrow\mathit{HH}_{-n}(\mathcal{A})$.
    
    \begin{remark}
    The notion of an exact Calabi--Yau structure is strictly stronger than that of a strong smooth Calabi--Yau structure in the sense of Kontsevich--Soibelman \cite{ks} (equivalently a left Calabi--Yau structure in the sense of Brav--Dyckerhoff \cite{bdcy}).
    \end{remark}
    
    \begin{lemma}\label{lemma:cyclic-d}
    The Weinstein manifold $W_k$ admits a cyclic quasi-dilation over any field $\mathbb{K}$.
    \end{lemma}
    \begin{proof}
    It is well-known that any Ginzburg dg algebra associated to quiver with potential is exact Calabi--Yau. For a proof, see \cite[Theorem 4.3.8]{bd}. Note that the field $\mathbb{K}$ is assumed to have characteristic $0$ or $2$ there, however, it is easy to see that the same proof holds over any field $\mathbb{K}$.
    	
    On the other hand, it follows from \cite[Corollary 37]{yle} that $W_k$ admits a cyclic quasi-dilation over $\mathbb{K}$ if and only if its wrapped Fukaya category $\mathcal{W}(W_k;\mathbb{K})$ admits an exact Calabi--Yau structure (which may differ from the smooth Calabi--Yau structure arising from the symplectic geometry of $W_k$). Again, although the field $\mathbb{K}$ is taken to have characteristic $0$ there, since the argument is based on the works of Ganatra \cite{sgs,sgc}, which works over an arbitrary field, the same argument extends to the case of an arbitrary $\mathbb{K}$. The lemma now follows from \cref{theorem:contraction} and the property that having an exact Calabi--Yau structure is invariant under quasi-isomorphisms.
    \end{proof}
    
    \subsection{Nonexistence of aspherical summands}
    
    A crucial step in the proof of \cref{theorem:main} is to exclude the aspherical summands in the prime decomposition of oriented closed exact Lagrangian submanifolds $L\subset W_k$. When $k=0$, this follows from the existence of a dilation in $\mathit{SH}^1(W_0;\mathbb{K})$ for a field of characteristic $0$ \cite[Corollary 6.3]{ss}, see \cite[Lemma 4.9]{sw} for details. When $k\geq1$, we use an argument similar to that of \cite[Lemma 4.1]{yln}.
    
    \begin{lemma}\label{lemma:aspherical}
    Let $L\subset W_k$ be a closed oriented exact Lagrangian submanifold, then $L$ admits a prime decomposition such that no prime summand is a $K(\pi,1)$.
    \end{lemma}
    \begin{proof}
    In this proof, we work over a field $\mathbb{K}$ of characteristic $0$. By \cref{lemma:cyclic-d}, $W_k$ admits a cyclic quasi-dilation $(\tilde{b},h)\in\mathit{SH}_{S^1}^1(W_k;\mathbb{K})\times\mathit{SH}^0(W_k;\mathbb{K})^\times$. It is proved in \cite[Corollary 40]{yle} that if the cyclic quasi-dilation $(\tilde{b},h)$ is actually a cyclic dilation, then $W_k$ cannot contain an exact Lagrangian $K(\pi,1)$ space. Thus we can assume from now on that $k\geq1$ and $h\in\mathit{SH}^0(W_k;\mathbb{K})^\times$ is a unit which is non-trivial in the sense that if we write
    \[
    h=\alpha\cdot1+t,
    \]
    where $\alpha\in\mathbb{K}$ and $t\in\mathit{SH}_+^0(W_k;\mathbb{K})$, then $t\neq0$. Here, $\mathit{SH}_+^\ast(W_k;\mathbb{K})$ is the positive symplectic cohomology generated by non-constant Hamiltonian orbits.
    
    Now suppose there is a closed oriented exact Lagrangian submanifold $L\subset W_k$ whose prime decomposition contains a component $L_i$ that is $K(\pi,1)$, then there is a degree $1$ map $L\rightarrow L_i$ which induces a map $H_\ast(\mathcal{L}L;\mathbb{K})\rightarrow H_\ast(\mathcal{L}L_i;\mathbb{K})$ preserving the unit defined by the inclusion of constant loops, the BV operator, and the product structure. Denote by $Z(A)$ the center of an algebra $A$. Consider the composition of maps
    \begin{equation}\label{eq:map}
    \begin{split}
    \mathit{SH}^0(W_k;\mathbb{K})\rightarrow\mathit{SH}^0(T^\ast L;\mathbb{K})\xrightarrow{\cong}H_3(\mathcal{L}L;\mathbb{K})\rightarrow H_3(\mathcal{L}L_i;\mathbb{K})\xrightarrow{\cong}Z(\mathbb{K}[\pi_1(L_i)]),
    \end{split}
    \end{equation}
    where the first map is the Viterbo restriction map, the second map is the Viterbo transfer, which is an isomorphism for cotangent bundles, the third map is the one induced from the projection $L\rightarrow L_i$ mentioned above, and the last map factors through the isomorphism $H_3(\mathcal{L}L_i;\mathbb{K})\cong\mathit{HH}^0\left(C_{-\ast}(\Omega_pL_i;\mathbb{K})\right)$, the latter is isomorphic to $Z(\mathbb{K}[\pi_1(L_i)])$ since $L_i$ is a $K(\pi,1)$. Under \eqref{eq:map}, the element $h\in\mathit{SH}^0(W_k;\mathbb{K})^\times$ is mapped to a central unit in the fundamental group algebra $\mathbb{K}[\pi_1(L_i)]$. Since $\pi_1(L_i)$ is torsion-free, it is a result due to \"{O}inert that any such central unit must be trivial \cite[Theorem 6.2]{jo}, meaning that it is a scalar multiple of some element in $\pi_1(L_i)$ (see also \cite[Appendix A]{yln} for a proof in the special case that is sufficient for our purpose here). This implies $\alpha=0$ in the expression of $h$. Note that the positive powers of $h$, $\langle h^n\rangle_{n\in\mathbb{Z}}$ must span an infinite cyclic subgroup of $\mathit{SH}^0(W_k;\mathbb{K})^\times$, otherwise their images in $\mathbb{K}[\pi_1(L_i)]$ would contradict \cite[Theorem 6.2]{jo} and the torsion-freeness of $\pi_1(L_i)$ mentioned above.
    
    The $\mathbb{K}$-algebra map $\mathit{SH}^0(W_k;\mathbb{K})\rightarrow Z(\mathbb{K}[\pi_1(L_i)])$ in \eqref{eq:map} can also be viewed as a consequence of the open-string Viterbo functoriality due to Abouzaid--Seidel \cite[Section 4.11]{ase}, which in our case can be interpreted as an $A_\infty$-algebra morphism (in particular, it identifies the idempotents $e_0$ and $e_1$ of $\Bbbk$)
    \[
    \psi\colon\mathcal{W}_k\longrightarrow\mathbb{K}[\pi_1(L_i)],
    \]
    where $\mathbb{K}[\pi_1(L_i)]$ is regarded as an $A_\infty$-algebra supported in degree $0$ with trivial $A_\infty$-operations. By abuse of notation, we use the same notation to denote the invertible element $h$ under the closed-open isomorphism
    \[
    \mathit{CO}\colon\mathit{SH}^0(W_k;\mathbb{K})\cong\mathit{HH}^0(\mathcal{W}(W_k;\mathbb{K})).
    \]
    Since $\mathcal{W}_k$ is supported in non-positive degrees, the only non-vanishing component of $\psi$ is its degree zero part, so it must be the case that the invertible element $h$ lies in the subgroup $H^0(\mathcal{W}_k)^\times\subset\mathit{HH}^0(\mathcal{W}(W_k;\mathbb{K}))^\times$ (here we use the fact that $\mathcal{W}_k$ is smooth Calabi--Yau and $H^0(\mathcal{W}_k)$ is commutative). Finally, since the powers of $h$ span an infinite cyclic subgroup $\langle h^n\rangle_{n\in\mathbb{Z}}\subset H^0(\mathcal{W}_k)^\times$, which contradicts the conclusion of \cref{corollary:H0}.
    \end{proof}
    
    Alternatively, since the positive powers of $h$ in the proof above are linearly independent with each other in $\mathit{SH}^0(W_k;\mathbb{K})$ (cf.\@ the proof of \cite[Lemma 4.1]{yln}), \cref{lemma:aspherical} can be proved by showing that $\mathit{SH}^0(W_k;\mathbb{K})$ is finite-dimensional. We include a proof of this fact in the following lemma since it may be of independent interest.
    
    \begin{lemma}
    Let $\mathbb{K}$ be a field of characteristic $0$. $\mathit{SH}^0(W_k;\mathbb{K})$ is finite-dimensional, and its dimension equals the Tjurina number of the singularity $R_k$.
    \end{lemma}
    \begin{proof}
    Recall that $\mathcal A_k$ denotes the derived contraction algebra associated to the crepant resolutions of the singularities $R_k$ defined in \eqref{eq:ca2_sing}. There are isomorphisms between $\mathbb{K}$-vector spaces
    \[
    \mathit{SH}^0(W_k;\mathbb{K})\cong\mathit{HH}^0(\mathcal{W}(W_k;\mathbb{K}))\cong\mathit{HH}^0(\mathcal{A}_k),
    \]
    where the first isomorphism follows from the work of Ganatra \cite[Theorem 1.1]{sgs}, and the second isomorphism is a consequence of \cref{theorem:MS}.
    
    Let $D_\mathit{Sg}^\mathit{dg}(R_k)$ be the dg enhancement of the triangulated category of singularities. When $k\geq1$, $R_k$ is an isolated hypersurface singularity, it follows from \cite[Proposition 5.2.5]{mb} and \cite[Theorem 5.9]{hk} that there are isomorphisms
    \[
    \mathit{HH}^0(\mathcal{A}_k)\cong\mathit{HH}^0(D_\mathit{Sg}^\mathit{dg}(\widehat R_k))\cong \widehat{\mathcal{T}}_k
    \]
    as $\mathbb{K}$-algebras, where $\mathcal{T}_k\coloneqq  \mathbb{K}[u,v,x,y]/(\sigma_k,J_{\sigma_k})$ is the Tjurina algebra of $R_k$, where
    \[
    \sigma_k=\left(uv-xy((x+1)^k+y-1)\right)
    \]
    is the ideal defining the singularity $R_k$ and $J_{\sigma_k}$ is its Jacobian ideal. Furthermore, $\widehat{R}_k=\mathbb{K}[\![u,v,x,y]\!]/\sigma_k$ and $\widehat{\mathcal T}_k = \mathbb{K}[\![u,v,x,y]\!]/(\sigma_k,J_{\sigma_k})$ are their completions. However, note that the functor $D_\mathit{Sg}^\mathit{dg}(R_k)\rightarrow D_\mathit{Sg}^\mathit{dg}(\widehat{R}_k)$ induced by the completion of $R_k$ is an idempotent completion of $D_\mathit{Sg}^\mathit{dg}(R_k)$, see \cite[Theorem 5.7]{td}. Since the idempotent completion is Morita equivalent to the original category, the Hochschild cohomology is unaffected after replacing $\widehat{R}_k$ with $R_k$. We also have the isomorphism $\widehat{\mathcal{T}}_k\cong\mathcal{T}_k$, since $R_k$ is an isolated singularity for $k\geq1$. Hence we conclude that
    \[
    \mathit{HH}^0(\mathcal{A}_k)\cong\mathit{HH}^0(D_\mathit{Sg}^\mathit{dg}(R_k))\cong \mathcal{T}_k.
    \]
    Finally, the corresponding Tjurina number $\dim(\mathcal{T}_k)$ is finite as long as $R_k$ is an isolated singularity, which is the case for $k\geq 1$.
    \end{proof}
    
    \subsection{Classification of spherical summands}\label{section:classification}
    
    Next, we show that the existence of a cyclic quasi-dilation over fields of finite characteristics can be used to rule out certain spherical summands in the prime decomposition of an exact Lagrangian submanifold $L\subset W_k$.
    
    Recall that for a spherical $3$-manifold $Q$, the fundamental group $\pi_1(Q)$ is either cyclic or a central extension of a dihedral, tetrahedral, octahedral, or icosahedral group by a cyclic group of even order. More precisely, there are the following possibilities.
    \begin{itemize}
    	\item $\pi_1(Q)$ is cyclic, then $Q$ is a lens space.
    	\item $\pi_1(Q)$ has the presentation
        \[
        \pi_1(Q) = \langle x,y \mid xyx^{-1} = y^{-1}, \; x^{2m} = y^n\rangle,
        \]
    where $m\geq1$, $n\geq2$ and $\gcd(m,n)=1$. The center $Z(\pi_1(Q))$ is isomorphic to $\mathbb{Z}_{2m}$, and the quotient $\pi_1(Q)/Z(\pi_1(Q))$ is isomorphic to the dihedral group $D_{2n}$. In this case, $Q$ is called a \textit{prism manifold}.
    	\item $\pi_1(Q)$ is a product of $\mathbb{Z}_m$ with a group of order $24$, where $m\geq1$ and $\gcd(m,6)=1$. $Z(\pi_1(Q))$ is isomorphic to $\mathbb{Z}_{2m}$, and the quotient $\pi_1(Q)/Z(\pi_1(Q))$ is the alternating group $A_4$. In this case, we say that $Q$ is a \textit{spherical manifold of type \textbf{T}}.
    	\item $\pi_1(Q)$ is a product of $\mathbb{Z}_m$, where $\gcd(m,6)=1$, with binary octahedral group $2O$. In this case, we say that $Q$ is a \textit{spherical manifold of type \textbf{O}}.
    	\item $\pi_1(Q)$ is a product of $\mathbb{Z}_m$, where $\gcd(m,30)=1$, with the binary icosahedral group $2I$. In this case, we say that $Q$ is a \textit{spherical manifold of type \textbf{I}}. Note that when $m=1$, $Q$ is the Poincar\'{e} homology sphere.
    \end{itemize}
    
    In \cite{gp}, a dilation over $\mathbb{F}_3$ is used to exclude the existence of an exact Lagrangian Poincar\'{e} homology sphere in $W_1$. Here we prove a stronger result based on \cref{lemma:cyclic-d}.
    
    \begin{proposition}\label{proposition:TOI}
    Let $k\geq1$. If $L\subset W_k$ is an oriented closed exact Lagrangian submanifold, then no prime summand of $L$ can be a spherical manifold of type \textbf{T}, \textbf{O}, or \textbf{I}.
    \end{proposition}
    \begin{proof}
    Let $L_i$ be a prime summand of $L$. As in the proof of \cref{lemma:aspherical}, we make use of the map $H_\ast(\mathcal{L}L;\mathbb{K})\rightarrow H_\ast(\mathcal{L}L_i;\mathbb{K})$ induced by the degree $1$ map $L\rightarrow L_i$. Let $B\pi_1(L_i)$ be the classifying space of the fundamental group of $L_i$. The classifying map $L_i\rightarrow B\pi_1(L_i)$ induces a map $H_\ast(\mathcal{L}L_i;\mathbb{K})\rightarrow H_\ast\left(\mathcal{L}B\pi_1(L_i);\mathbb{K}\right)$ on free loop space homologies, which among other things, is compatible with the loop product. Composing it with the Viterbo transfer map $\mathit{SH}^\ast(W_k;\mathbb{K})\rightarrow H_{3-\ast}(\mathcal{L}L;\mathbb{K})$ yields a composition
    \begin{equation}\label{eq:vit}
    \upsilon\colon\mathit{SH}^\ast(W_k;\mathbb{K})\longrightarrow H_{3-\ast}(\mathcal{L}L;\mathbb{K}) \longrightarrow H_{3-\ast}(\mathcal{L}L_i;\mathbb{K}) \longrightarrow H_{3-\ast}\left(\mathcal{L}B\pi_1(L_i);\mathbb{K}\right)
    \end{equation}
    that is compatible with the product structures.

    Now, suppose that the $i$-th prime summand $L_i$ of $L$ is a spherical manifold \textbf{T}, \textbf{O}, or \textbf{I}. Then there is a surjective homomorphism $\pi_1(L_i)\rightarrow\Gamma$ from the fundamental group to a (non-trivial) finite group with trivial center. Depending on the type of $L_i$, the group $\Gamma$ can be taken as follows.
    \begin{itemize}
    	\item If $L_i$ is a type \textbf{T} manifold, we have the quotient map $\pi_1(L_i)\rightarrow \pi_1(L_i)/Z(\pi_1(L_i)) \cong A_4$, so we can pick $\Gamma=A_4$.
    	\item If $L_i$ is a type \textbf{O} manifold, we have the composition $\pi_1(L_i)\rightarrow2O\rightarrow S_4$, where the first map is the natural projection, and the second map comes from the definition of $2O$ as an extension of the octahedral group $O\cong S_4$ by $\mathbb{Z}_2$. In this case, we pick $\Gamma=S_4$.\footnote{Note that it's also possible to pick $\Gamma=S_3$, which does not affect the argument.}
    	\item If $L_i$ is a type \textbf{I} manifold, consider the composition $\pi_1(L_i)\rightarrow2I\rightarrow A_5$, where the first map is the natural projection, and the second map comes from the definition of $2I$ as an extension of the icosahedral group $I\cong A_5$ by $\mathbb{Z}_2$. We pick $\Gamma=A_5$ in this case.
    \end{itemize}
    The group homomorphism $\pi_1(L_i)\rightarrow\Gamma$ induces a map $B\pi_1(L_i)\rightarrow B\Gamma$ between the corresponding classifying spaces, therefore also a map $H_\ast(\mathcal{L}B\pi_1(L_i);\mathbb{K})\rightarrow H_\ast(\mathcal{L}B\Gamma;\mathbb{K})$ preserving the loop product. Composing this map with the map \eqref{eq:vit} yields a map
    \begin{equation}\label{eq:gamma}
    \vartheta\colon \mathit{SH}^\ast(W_k;\mathbb{K})\longrightarrow H_{3-\ast}(\mathcal{L}B\Gamma;\mathbb{K}).
    \end{equation}
    In the same vein, we also have a map $\tilde{\vartheta}\colon \mathit{SH}_{S^1}^\ast(W_k;\mathbb{K})\rightarrow H_{3-\ast}^{S^1}(\mathcal{L}B\Gamma;\mathbb{K})$, which is the $S^1$-equivariant analog of \eqref{eq:gamma}, and a commutative diagram
    \begin{equation}\label{eq:cd}
    	\begin{tikzcd}
    	\mathit{SH}^\ast_{S^1}(W_k;\mathbb{K}) \arrow[d,"\mathbf{B}"'] \arrow[r,"\tilde{\vartheta}"] &H_{3-\ast}^{S^1}(\mathcal{L}B\Gamma;\mathbb{K}) \arrow[d,"\mathbf{B}_\mathcal{L}"] \\
    	\mathit{SH}^{\ast-1}(W_k;\mathbb{K}) \arrow[r,"\vartheta"] &H_{4-\ast}(\mathcal{L}B\Gamma;\mathbb{K})
    	\end{tikzcd}
    \end{equation}
    where $\mathbf{B}_\mathcal{L}$ is the marking map on string homology introduced by Chas--Sullivan \cite[Section 6]{cs} (cf.\@ \cite[Proposition 2.9]{bo}). Note that the compatibility with the marking map follows from the fact that the Viterbo transfer and its $S^1$-equivariant analog preserve the underlying chain level $S^1$-complex structures, and that the maps on the free loop spaces induced by the classifying map and the group homomorphism $\pi_1(L_i)\rightarrow\Gamma$ preserve loop rotations.
    
    Based on known computations of group homologies, we have
    \begin{align*}
        H_3(BA_4;\mathbb{Z}) &\cong H_3(A_4;\mathbb{Z})\cong\mathbb{Z}_6 \\
        H_3(BS_4;\mathbb{Z})&\cong H_3(S_4;\mathbb{Z})\cong\mathbb{Z}_2\oplus\mathbb{Z}_4\oplus\mathbb{Z}_3 \\
        H_3(BA_5;\mathbb{Z})&\cong H_3(A_5;\mathbb{Z})\cong\mathbb{Z}_2\oplus\mathbb{Z}_3\oplus\mathbb{Z}_5.
    \end{align*}
    We now use the flexibility of choosing the field $\mathbb{K}$ provided by \cref{lemma:cyclic-d}, and take $\mathbb{K}=\mathbb{F}_3$ when $\Gamma=A_4$ or $S_4$, and $\mathbb{K}=\mathbb{F}_2$ when $\Gamma=A_5$. It follows from the computations above that in each case we have isomorphisms
    \begin{equation}\label{eq:iso}
    H_3(L_i;\mathbb{K})\cong H_3(B\pi_1(L_i);\mathbb{K})\cong H_3(B\Gamma;\mathbb{K}).
    \end{equation}
    
    By \cref{lemma:cyclic-d}, there is a class $\tilde{b}\in\mathit{SH}_{S^1}^1(W_k;\mathbb{K})$ with $\mathbf{B}(\tilde{b})=h$ for some $h\in\mathit{SH}^0(W_k;\mathbb{K})^\times$. Since the map $\vartheta$ preserves the product structures, by \eqref{eq:iso} the element $\mathbf{B}_\mathcal{L}(\tilde{\vartheta}(\tilde{b}))$ defines a central unit
    \begin{equation}\label{eq:cut}
    \vartheta(h)\in H_3(\mathcal{L}B\Gamma;\mathbb{K})\cong\mathit{HH}^0(\mathbb{K}[\Gamma])\cong Z\left(\mathbb{K}[\Gamma]\right)
    \end{equation}
    in the group algebra $\mathbb{K}[\Gamma]$. Next, the central unit $\vartheta(h)$ must be the identity, a claim we prove in \cref{lemma:cut} below. From the commutativity of \eqref{eq:cd} and the isomorphisms \eqref{eq:iso}, we see that the image of the class $\tilde{\vartheta}(\tilde{b})\in H_2^{S^1}(\mathcal{L}B\Gamma;\mathbb{K})$ in the marking map $\mathbf{B}_\mathcal{L}$ is non-vanishing. Since $\mathbf{B}_\mathcal{L}$ preserves the homotopy classes of loops, the fact that $\vartheta(h)=1$ implies that the cycle $\tilde{\vartheta}(\tilde{b})$ lies in the space of contractible loops in $B\Gamma$. As $B\Gamma$ is a $K(\Gamma,1)$ space, the image of $\tilde{\vartheta}(\tilde{b})$ under $\mathbf{B}_\mathcal{L}$ must vanish (cf.\@ \cite[Example 6.2]{ss}), and we arrive at a contradiction. Alternatively, one can pass to the chain model $C_\ast^\Diamond(\mathcal{L}B\Gamma;\mathbb{K})$ of the free loop space homology constructed by Cohen--Ganatra \cite{cg}, which is obtained as a quotient of $C_\ast(\mathcal{L}B\Gamma;\mathbb{K})$. Since $C_\ast^\Diamond(\mathcal{L}B\Gamma;\mathbb{K})$ is a strict $S^1$-complex, this reduces the situation to that of \cite[Lemma 5.17]{gp}.
    \end{proof}
    
    \begin{lemma}\label{lemma:cut}
    We have $\vartheta(h)=1$ in \eqref{eq:cut}.
    \end{lemma}
    \begin{proof}
    To prove the claim, we need to work over the integers, and consider the map
    \[
    \vartheta_\mathbb{Z}\colon\mathit{SH}^0(W_k;\mathbb{Z})\longrightarrow H_3(\mathcal{L}B\Gamma;\mathbb{Z})\cong Z(\mathbb{Z}[\Gamma]),
    \]
    which is the integral lift of the map $\vartheta$ defined in \eqref{eq:gamma}, where $\mathbb{K}=\mathbb{F}_2$ or $\mathbb{F}_3$. Via the open-closed string map, this map can also be regarded as a map induced by the $A_\infty$-algebra morphism $\mathcal{W}_k^\mathbb{Z}\rightarrow\mathbb{Z}[\Gamma]$, where by $\mathcal{W}_k^\mathbb{Z}$ we mean the endomorphism $A_\infty$-algebra of the Lagrangian cocores in $\mathcal{W}(W_k;\mathbb{Z})$. There is a commutative diagram
    \begin{equation}\label{eq:cd1}
    	\begin{tikzcd}
    	\mathit{HH}^0(\mathcal{W}_k^\mathbb{Z}) \arrow[d] \arrow[rr,"\mathit{CO}^{-1}\circ\vartheta_\mathbb{Z}"] &&Z(\mathbb{Z}[\Gamma]) \arrow[d,"\rho"] \\
    	\mathit{HH}^0(\mathcal{W}_k) \arrow[rr,"\mathit{CO}^{-1}\circ\vartheta"] &&Z(\mathbb{K}[\Gamma])
    	\end{tikzcd}
    \end{equation}
    where the vertical maps are induced by the unique ring map $\mathbb{Z}\rightarrow\mathbb{K}$, and the field $\mathbb{K}$ is taken to be $\mathbb{F}_3$ for $\Gamma=A_4$ or $\Gamma=S_4$, and $\mathbb{F}_2$ for $\Gamma=A_5$, as in the proof of \cref{proposition:TOI}.
    
    Consider the central unit $\vartheta(h)\in Z(\mathbb{K}[\Gamma])$. For the same reason as in the proof of \cref{lemma:aspherical}, the invertible element $h\in\mathit{HH}^0(\mathcal{W}_k)\cong\mathit{SH}^0(W_k;\mathbb{K})$ actually lies in the subalgebra $H^0(\mathcal{W}_k)$. Since our computations in \cref{section:computation} hold over $\mathbb{Z}$ (cf.\@ \cref{remark:Z}), it follows that any unit in $H^0(\mathcal{W}_k)$ admits a lift to $H^0(\mathcal{W}_k^\mathbb{Z})$, which allows us to conclude from the commutativity of \eqref{eq:cd1} that $\vartheta(h)$ also admits a lift in the integral group ring $\mathbb{Z}[\Gamma]$, which we denote by
    \[
    \vartheta(h)_\mathbb{Z}\in Z(\mathbb{Z}[\Gamma]). 
    \]
    To prove the lemma, it suffices to show that any central unit in $\mathbb{Z}[\Gamma]$ gets mapped by $\rho\colon\mathbb{Z}[\Gamma]\rightarrow\mathbb{K}[\Gamma]$ to a trivial central unit in the group algebra $\mathbb{K}[\Gamma]$. This follows from a case-by-case analysis.
    
    When $\Gamma=S_4$, it is proved by Ritter--Sehgal \cite{rsi} that a central unit in $\mathbb{Z}[S_4]$ must be trivial. The same conclusion holds for $\Gamma=A_4$ by \cite[Theorem 4.6]{raf}. The case of $\Gamma=A_5$ is studied in \cite{lpc}, where it is shown that the group of central units in $\mathbb{Z}[A_5]$ is generated by a single non-trivial element
    \[
    u=49+26C_1-10C_2-16C_4,
    \]
    where the $C_i$'s are sums of elements in the conjugacy classes in $A_5$, with $C_1$ being the sum of elements conjugate to $(12345)$, $C_2$ is the sum of elements conjugate to $(13524)$, and $C_4$ is the sum of all elements which are the product of two disjoint transpositions. Since it is clear that $u$ is mapped to the identity in $\mathbb{F}_2[A_5]$ by $\rho$, we arrive at the conclusion that in all three cases
    \[
    \vartheta(h)=\rho\left(\vartheta(h)_\mathbb{Z}\right)=1.
    \]
    \end{proof}
    
    \begin{remark}\label{remark:JR}
    It is not true that any central unit in the group algebras $\mathbb{F}_3[A_4]$, $\mathbb{F}_3[S_4]$ and $\mathbb{F}_2[A_5]$ is a scalar multiple of the identity. In fact, for any finite group algebra $\mathbb{K}[G]$, where $\mathrm{char}(\mathbb{K})$ divides the order of $G$, the element $1+\sum_{g\in G}g$ is a central unit. Thus passing to integral lifts is essential for the argument above.
    \end{remark}
    
    Essentially the same technique as in the proof of \cref{proposition:TOI} can be used to deduce restrictions on the prime summands of an exact Lagrangian $L\subset W_k$ that are diffeomorphic to lens spaces and prism manifolds.
    
    \begin{proposition}\label{proposition:lensprism}
    Suppose $k\geq1$, and that $L\subset W_k$ is a closed, oriented exact Lagrangian submanifold. Let $L_i$ be a summand in the prime decomposition of $L$.
    \begin{itemize}
    \item If $L_i$ is diffeomorphic to a lens space and $p$ is a prime factor of $|\pi_1(L_i)|$, then $p$ also divides $k$.
    \item If $L_i$ is diffeomorphic to a prism manifold, then $\pi_1(L_i)/Z(\pi_1(L_i))\cong D_{2n}$, where $n$ is an even number.
    \end{itemize}
    \end{proposition}
    \begin{proof}
    We first consider the case when $L_i$ is diffeomorphic to a lens space. Let $p$ be a prime factor of $|\pi_1(L_i)|$. Since $\pi_1(L_i)$ is cyclic, there is a surjective homomorphism $\pi_1(L_i)\rightarrow\mathbb{Z}_p$. As in the proof of \cref{proposition:TOI}, we have a commutative diagram
    \begin{equation}\label{eq:cd2}
    	\begin{tikzcd}
    	\mathit{SH}^\ast_{S^1}(W_k;\mathbb{K}) \arrow[d,"\mathbf{B}"'] \arrow[r,"\tilde{\vartheta}"] &H_{3-\ast}^{S^1}(\mathcal{L}B\mathbb{Z}_p;\mathbb{K}) \arrow[d,"\mathbf{B}_\mathcal{L}"] \\
    	\mathit{SH}^{\ast-1}(W_k;\mathbb{K}) \arrow[r,"\vartheta"] &H_{4-\ast}(\mathcal{L}B\mathbb{Z}_p;\mathbb{K})
    	\end{tikzcd}
    \end{equation}
    where the cyclic group $\mathbb{Z}_p$ now plays the role of $\Gamma$. We take $\mathbb{K}$ to be a field of characteristic $p$. By \cref{lemma:cyclic-d}, there is a cyclic quasi-dilation $(\tilde{b},h)\in\mathit{SH}_{S^1}^1(W_k;\mathbb{K})\times\mathit{SH}^0(W_k;\mathbb{K})^\times$. By transferring using the diagram \eqref{eq:cd2}, we obtain a cyclic quasi-dilation $\left(\tilde{\vartheta}(\tilde{b}),\vartheta(h)\right)$ on the classifying space $B\mathbb{Z}_p$, which is just the infinite-dimensional lens space. Since $H_3(B\mathbb{Z}_p;\mathbb{Z})\cong\mathbb{Z}_p$, it follows from our choice of $\mathbb{K}$ that $H_3(B\mathbb{Z}_p;\mathbb{K})\cong H_3(L_i;\mathbb{K})$. Then a similar argument as in the proof of \cref{proposition:TOI} shows that $\vartheta(h)\neq1$. By rescaling, we actually have $\vartheta(h)\neq\alpha$ for any non-zero scalar $\alpha\in\mathbb{K}^\times$. It follows that under the isomorphism
    \[
    H_3(\mathcal{L}B\mathbb{Z}_p;\mathbb{K})\cong \mathbb{K}[\mathbb{Z}_p]\cong \mathbb{K}[z]/(z^p),
    \]
    $\vartheta(h)$ is mapped to a non-trivial $p$-th root of unity. Once again, the map $\vartheta$ comes from the dg algebra morphism $\mathcal{W}_k\rightarrow\mathbb{K}[\mathbb{Z}_p]$, and $h\in\mathit{SH}^0(W_k;\mathbb{K})^\times$ can be regarded as an element of $H^0(\mathcal{W}_k)^\times$. Since $\vartheta(h)$ is a non-trivial $p$-th root of unity, its preimage $h\in H^0(\mathcal{W}_k)^\times$ must satisfy $h^n=1$ for some multiple $n$ of $p$. By \cref{corollary:H0}, the order of any non-trivial torsion unit in $H^0(\mathcal{W}_k)$ divides $k$, so it follows that $p\mid k$.
    
    Next we consider the case when $L_i$ is diffeomorphic to a prism manifold. Then by definition $|\pi_1(L_i)| = 4mn$ and $\pi_1(L_i)/Z(\pi_1(L_i))\cong D_{2n}$ for some $m \geq 1$ and $n\geq 2$. Consider the surjective homomorphism $\pi_1(L_i)\rightarrow\pi_1(L_i)/Z(\pi_1(L_i))\cong D_{2n}$. Instead of \eqref{eq:cd2} we have a commutatve diagram
    \[
        \begin{tikzcd}
            \mathit{SH}^\ast_{S^1}(W_k;\mathbb{K}) \arrow[d,"\mathbf{B}"'] \arrow[r,"\tilde{\vartheta}"] &H_{3-\ast}^{S^1}(\mathcal{L}BD_{2n};\mathbb{K}) \arrow[d,"\mathbf{B}_\mathcal{L}"] \\
            \mathit{SH}^{\ast-1}(W_k;\mathbb{K}) \arrow[r,"\vartheta"] &H_{4-\ast}(\mathcal{L}BD_{2n};\mathbb{K})
    	\end{tikzcd}
    \]
    where $\mathbb{K}$ is a field of characteristic $p$, with $p$ being a prime factor of $n$. Known computations of the homologies of the dihedral groups show $H_3(BD_{2n};\mathbb{Z})\cong\mathbb{Z}_{2n}$, and hence $H_3(BD_{2n};\mathbb{K})\cong H_3(L_i;\mathbb{K})$. Arguing as above we see that $\vartheta(h)$ is mapped to a central unit in $\mathbb{K}[D_{2n}]$ that is not a scalar multiple of the identity. On the other hand, the ring structure of $\mathit{HH}^0(\mathbb{K}[D_{2n}])$ has been computed explicitly in \cite[Proposition 4.5]{th}, which in particular shows that when $n$ is odd, $\mathit{HH}^0(\mathbb{K}[D_{2n}])$ does not contain a torsion unit. Since the map $\vartheta$ is induced from the dg algebra map $\mathcal{W}_k\rightarrow\mathbb{K}[D_{2n}]$, and every non-trivial unit in $H^0(\mathcal{W}_k)$ is torsion, we see that the image of $\vartheta(h)$ must be torsion, which is a contradiction. Thus $n$ must be even.
    \end{proof}
    
    \subsection{Proof of the main result}\label{section:proof_main_result}
    
    Let $Q_0,Q_1\subset W_k$ be two exact Lagrangian spheres meeting along a circle $Z=Q_0\cap Q_1$ (that is possibly knotted in each sphere). There is a minimal model of the Fukaya $A_\infty$-algebra
    \[
    \mathcal{Q}_k^\mathbb{Z}\coloneqq \bigoplus_{i,j\in \{0,1\}}\mathit{CF}^\ast(Q_i,Q_j)
    \]
    over $\mathbb{Z}$ such that the Massey products are given by
    \[
    \mu^3(e,f,e)=\lambda f^\vee,\quad \mu^3(f,e,f)=\lambda'e^\vee
    \]
    for some $\lambda,\lambda'\in\mathbb{Z}$, where $e\in\mathit{HF}^1(Q_0,Q_1)$, $f\in\mathit{HF}^1(Q_1,Q_0)$ are degree $1$ generators, and $e^\vee\in\mathit{HF}^2(Q_1,Q_0)$, $f^\vee\in\mathit{HF}^2(Q_0,Q_1)$ are the corresponding Poincar\'{e} dual generators. The integers $\lambda$ and $\lambda'$ are independent of the choice of almost complex structures on $W_k$, so they are symplectic invariants. 
    
    Let $K \subset W_k(\kappa_0,\kappa_1)$, where $\kappa_0$ and $\kappa_1$ may be non-trivial, denote the exact Lagrangian submanifold obtained by performing Lagrangian surgery of $Q_0$ and $Q_1$ along their clean intersection. It is unique up to diffeomorphism.
    
    \begin{lemma}[{\cite[Lemma 2.4]{sw}}]\label{lemma:lambda}
    We have $\lambda = \lambda' = \pm k$. If $\lambda=0$, then $H^\ast(K;\mathbb{Z})\cong H^\ast(S^1\times S^2;\mathbb{Z})$. If $|\lambda|>0$, then $K$ is a homology lens space with $H^1(K;\mathbb{Z})\cong \mathbb{Z}_{|\lambda|}$.
    \qed
    \end{lemma}
    
    With all the ingredients at hand, we now prove \cref{theorem:main}.
    
    \begin{proof}[Proof of \cref{theorem:main}]
    We assume $k\geq1$. Suppose that there is a Hamiltonian isotopy which changes the knot types of the unknot to $\kappa_0\colon Z'\hookrightarrow Q_0'$ and $\kappa_1\colon Z'\hookrightarrow Q_1'$, respectively, where $Q_i'$ is the Hamiltonian isotopy of $Q_i$ and $Z'=Q_0'\cap Q_1'$. Then there is a Weinstein embedding $W_n(\kappa_0,\kappa_1)\subset W_k$ (cf.\@ \cref{def:double_bubble}), and we assume at least one $\kappa_i$ is non-trivial. Since Hamiltonian isotopies preserve quasi-isomorphism classes of objects in the Fukaya category, $Q_i'$ is quasi-isomorphic to $Q_i$ in the compact Fukaya category $\mathcal{F}(W_k;\mathbb{Z})$. Hence \cref{lemma:lambda} implies $n=k$. 
    
    We first consider the case where one of $\kappa_i$'s is trivial (without loss of generality assume $\kappa_1$ is trivial). Then the Lagrangian surgery $K$ of $Q_0'$ and $Q_1'$ is given by the Dehn surgery on the knot $\kappa_0$ with integer slope $k$. It is known that $K$ has a prime decomposition with at most three summands \cite[Corollary 5.3]{jh}. Moreover, the same result further says that if there actually were three summands, then one of them would have to be an integral homology sphere. However, this is not possible by \cref{proposition:TOI}. Thus $K$ has at most two summands. By \cref{lemma:aspherical}, none of these summands can be aspherical. Also, by our assumption that $k\geq1$, \cref{lemma:lambda} implies $H^1(K;\mathbb{Z})\cong\mathbb{Z}_k$, which means that there can be no $S^1\times S^2$ summand either. It follows that either $K$ itself is a spherical space form, or both of the summands in the prime decomposition of $K$ are spherical space forms. By \cref{proposition:TOI,proposition:lensprism} it then follows that $K$ is one of the following:
    \begin{enumerate}[(i)]
    	\item a lens space such that every prime factor in $|\pi_1(K)|$ also divides $k$,
        \item a prism manifold with $\pi_1(K)/Z(\pi_1(K))\cong D_{2n}$ for some even number $n$,
    	\item a connected sum of two spherical manifolds as in (i) and (ii).
    \end{enumerate}
    
    We first show that $K$ cannot be a prism manifold. If $K$ was diffeomorphic to a prism manifold, then a presentation of its fundamental group would be given by
    \[
    \pi_1(K) = \langle x,y \mid xyx^{-1} = y^{-1}, \; x^{2m} = y^n\rangle,
    \]
    where $n$ is even by \cref{proposition:lensprism}. We note however that this implies that its abelianization $H_1(K;\mathbb{Z})$ is not cyclic, contradicting \cref{lemma:lambda}. This shows that if $K$ is irreducible, then it must be a lens space. One can then apply the deep result of Kronheimer--Mrowka--Ozsv\'{a}th--Szab\'{o} \cite[Theorem 1.1]{kmos} to conclude that $\kappa_0$ must be the unknot.
    
    It remains to show that $K$ cannot be the connected sum of two spherical manifolds as in (i) and (ii). First notice that since $K$ is reducible, it follows from the proof of \cite[Theorem 1]{gl} that the only possibility is that $K$ is the connected sum of two (non-trivial) lens spaces $K_0$ and $K_1$. As in the proof of \cref{proposition:TOI}, we consider the map \eqref{eq:vit} and its $S^1$-equivariant analog
    \[
    \tilde{\upsilon}\colon\mathit{SH}_{S^1}^\ast(W_k;\mathbb{K})\longrightarrow H_{3-\ast}^{S^1}(\mathcal{L}B\pi_1(K);\mathbb{K}).
    \]
    Under $\tilde{\upsilon}$, the class $\tilde{b}\in\mathit{SH}_{S^1}^1(W_k;\mathbb{K})$ in the cyclic quasi-dilation gets mapped to a class $\tilde{\upsilon}(\tilde{b})\in H_2^{S^1}(\mathcal{L}B\pi_1(K);\mathbb{K})$, whose image under the marking map $\mathbf{B}_\mathcal{L}$ gives a central unit $\upsilon(h)$ in the fundamental group algebra $\mathbb{K}[\pi_1(K)]$. As $K = K_0 \# K_1$, we have 
    \begin{equation}\label{eq:fp}
    \pi_1(K)\cong\pi_1(K_0)\ast\pi_1(K_1)\cong\mathbb{Z}_m\ast\mathbb{Z}_n
    \end{equation}
    where $m$ and $n$ are coprime, where $\ast$ denotes the free product. Let $p$ be any prime number dividing $m$ or $n$, and let $\mathbb{K}$ be a field of characteristic $p$. It follows that
    \[
    H_3(B\pi_1(K);\mathbb{K})\cong H_3\left(B\pi_1(K_0);\mathbb{K}\right)\oplus H_3(B\pi_1(K_1);\mathbb{K})\cong H_3(K;\mathbb{K}). 
    \]
    The same argument as in the proof of \cref{proposition:TOI} then implies that $\upsilon(h)\neq1$ for any cyclic quasi-dilation $(\tilde{b},h)$ over $\mathbb{K}$. It means that under the isomorphism $H_3(\mathcal{L}B\pi_1(K);\mathbb{K})\cong Z(\mathbb{K}[\pi_1(K)])$, $\upsilon(h)$ defines a central unit of $\mathbb{K}[\pi_1(K)]$ that is not the identity. However, it follows from \eqref{eq:fp} that the center of the group algebra $\mathbb{K}[\pi_1(K)]$ is trivial, see e.g.\@ \cite[Lemma 6.2]{ail}. Thus $\upsilon(h)=1$ up to rescaling of $\tilde{b}$, and we get a contradiction.
    
    We can assume from now on that both $\kappa_0$ and $\kappa_1$ are non-trivial. Then the surgery $K$ is irreducible and contains an incompressible torus by \cite[Proposition 2.3]{hkmp}. Since $K\subset W_k$, it admits a cyclic quasi-dilation over a field $\mathbb{K}$ of characteristic $0$ by Viterbo functoriality (and its $S^1$-equivariant analog), see \cite[Section 5.1]{yle}. By \cref{proposition:cover}, it is finitely covered by $S^1\times\Sigma_g$, where $g\geq1$, or a repeated connected sum of $S^1\times S^2$. In the former case, $K$ must be aspherical, and we get a contradiction by \cref{lemma:aspherical}. Thus $K$ is finitely covered by $\#_r(S^1\times S^2)$ for some $r\geq0$. Since $H^1(K;\mathbb{Z})=\mathbb{Z}_k$ and $K$ is irreducible, it must be spherical, which contradicts the existence of an incompressible torus.
    \end{proof}
    
    We end this paper by remarking on the differences between our argument in the general case with the argument of Ganatra--Pomerleano in the case $k=1$ \cite[Section 6.4]{gp}. First, the spherical summands of types \textbf{T}, \textbf{O}, and \textbf{I} in the prime decompositions of exact Lagrangians $L\subset W_1$ are ruled out in \cite{gp} using the fact that $W_1$ admits a dilation over characteristic $3$, therefore by \cite[Lemma 5.17]{gp}, the order of the fundamental group of this spherical summand is not divisible by $3$. This argument is not applicable in the general case since $W_k$ does not admit a dilation (or even a cyclic dilation) over $\mathbb{F}_3$ if $3|k$. On the other hand, it is not hard to see that our argument in \cref{proposition:TOI} actually proves that no spherical exact Lagrangians of type \textbf{T}, \textbf{O}, or \textbf{I} can exist in a Liouville manifold admitting a cyclic quasi-dilation over fields of characteristics $2$ and $3$. In \cite{gp}, the nonexistence of exact Lagrangian prism manifolds and exact Lagrangians $L\subset W_1$ that are diffeomorphic to a connected sum of two spherical manifolds is proved by equipping $L$ with different local systems, which gives rise to $|H_1(L;\mathbb{Z})|$ orthogonal objects in the Fukaya category $\mathcal{F}(W_1;\mathbb{F}_3)$. One can then derive a contradiction from Seidel's work \cite{psd}, which in particular implies that exact Lagrangian $\mathbb{K}$-homology spheres which are orthogonal in the Fukaya category of a Liouville manifold with dilations (over $\mathbb{K}$) cannot be squeezed into the same homology class. Unfortunately, this argument does not generalize to Liouville manifolds with cyclic (quasi-)dilations, see \cite[Section 5.3]{yle} for an explanation. Because of this, we did not manage to prove the nonexistence of exact Lagrangian prism manifolds or connected sums of spherical $3$-manifolds in $W_k$ for a general $k$, our argument only applies when $L$ is a Dehn surgery, which turns out to be enough for our purposes.

    \bibliographystyle{alpha}
    \bibliography{ref.bib}

    \Addresses
    
    \end{document}